\documentclass[a4paper,11pt,twoside]{article}
\usepackage[utf8]{inputenc}
\usepackage{amsmath}  
\usepackage{amsthm,amsfonts,amssymb,esint,cases}
\usepackage[margin=1.5in, marginparwidth=1.2in]{geometry}
\usepackage{enumitem}
\usepackage{graphicx}
\usepackage{fancyhdr}
\fancypagestyle{plain}{

\fancyhf{}
\fancyfoot[L]{\footnotesize \today}
\fancyfoot[C]{\thepage}
}
\usepackage{tocloft}
\usepackage{appendix}
\usepackage[
pdfauthor={Markus Tempelmayr}, pdftitle={Title}, pdftex]{hyperref}
\hypersetup{colorlinks=false, pdfborder={0 0 0}}
\usepackage{autonum} 
\renewcommand{\-}{\hspace{-.3ex}-\hspace{-.3ex}}
\newcommand{\+}{\hspace{-.3ex}+\hspace{-.3ex}}

\renewcommand{\d}{{\rm d}}

\renewcommand{\epsilon}{\varepsilon}
\newcommand{\E}{\mathbb{E}}
\newcommand{\G}{\Gamma^*}
\newcommand{\Gtau}{\Gamma^{*(\tau)}}
\newcommand{\Gtaut}{\Gamma^{*(\tau')}}
\renewcommand{\H}{\mathbb{H}}
\newcommand{\id}{{\rm id}}

\newcommand{\m}{\mathbf{m}}

\newcommand{\n}{\mathbf{n}}
\newcommand{\N}{\mathbb{N}}
\renewcommand{\phi}{\varphi}

\renewcommand{\preceq}{\preccurlyeq}
\newcommand{\Q}{{\mathbb Q}}
\newcommand{\R}{\mathbb{R}}
\renewcommand{\S}{\mathcal{S}}

\newcommand{\T}{\mathsf{T}^*}
\newcommand{\Tbar}{\mathsf{\bar{T}}^*}
\newcommand{\Ttilde}{\mathsf{\tilde T}^*}

\newcommand{\tG}{\widetilde\Gamma^*}
\newcommand{\tGtau}{\widetilde\Gamma^{*(\tau)}}
\newcommand{\vertiii}[1]{{\vert\kern-0.25ex\vert\kern-0.25ex\vert #1 \vert\kern-0.25ex\vert\kern-0.25ex\vert}}

\newcommand{\Z}{\mathbb{Z}}
\newcommand{\z}{\mathsf{z}}
\newcommand{\0}{\mathbf{0}}
\theoremstyle{plain}
\newtheorem{theorem}{Theorem}[section]
\newtheorem*{theorem*}{Theorem}
\newtheorem{proposition}[theorem]{Proposition}
\newtheorem{corollary}[theorem]{Corollary}
\newtheorem{lemma}[theorem]{Lemma}
\theoremstyle{definition}
\newtheorem{definition}[theorem]{Definition}
\newtheorem{assumption}[theorem]{Assumption}
\newtheorem{remark}[theorem]{Remark}
\newtheorem*{notation}{Notation}
\theoremstyle{remark}

\numberwithin{equation}{section}
\usepackage{bbding}

\title{\bf\Large Characterizing models in regularity structures:\\ 
A quasilinear case}
\author{Markus Tempelmayr
}
\date{}
\renewenvironment{abstract}
{
\begin{center}
\begin{minipage}{.9\textwidth}\small\textbf{Abstract.}
}
{
\end{minipage}
\end{center}
}
\newenvironment{keywords}
{
\begin{center}
\begin{minipage}{.9\textwidth}\small\textbf{Keywords:}%
}
{
\end{minipage}
\end{center}
}
\newenvironment{msc}
{
\begin{center}
\begin{minipage}{.9\textwidth}\small\textbf{MSC 2020:}%
}
{
\end{minipage}
\end{center}
}
\begin{document}

\maketitle

\begin{abstract}
We give a novel characterization of the centered model in regularity structures 
which persists for rough drivers even as a mollification fades away. 
We present our result for a class of quasilinear equations driven by noise, 
however we believe that the method is robust 
and applies to a much broader class of subcritical equations. 

Furthermore, we prove that a convergent sequence of noise ensembles, 
satisfying uniformly a spectral gap assumption, 
implies the corresponding convergence of the associated models. 
Combined with the characterization, 
this establishes a universality-type result.
\end{abstract}




\begin{msc} 
60H17, 
60L30, 
60H07. 
\end{msc}


\tableofcontents


\section{Introduction}\label{intro}

We continue the program initiated in \cite{OW19} and 
consider quasilinear parabolic equations of the form
\begin{equation}\label{spde}
(\partial_0-\Delta) u = a(u)\Delta u + \xi
\end{equation}
with a rough random forcing $\xi$. 
We are interested in the regime where the nonlinear term $a(u)\Delta u$ 
is singular but renormalizable. 
The nonlinearity $a$ is assumed to be scalar-valued, 
smooth and such that the equation is uniformly parabolic. 

In \cite{OSSW21}, the authors suggested an alternative 
point of view on Hairer's regularity structures \cite{Hai14}, 
the main differences being a different index set than trees 
and working with objects that are dual to Hairer's. 
We comment throughout the text on differences and similarities between 
the two approaches, 
see in particular Remarks~\ref{rem:comparison} and \ref{rem:ossw}. 
However, the main philosophy of splitting the problem 
into a probabilistic construction of the model and 
a deterministic solution theory stays the same. 
Using the notion of modelled distributions, 
an a-priori estimate was established in \cite{OSSW21} 
by assuming existence of a model that is suitably 
(stochastically) estimated. 
The main ingredients of a regularity structure and 
the notion of model in this setting were introduced there. 

In \cite{LOT23}, this regularity structure for quasilinear 
equations was systematically obtained and algebraically characterized. 
It was shown that, analogous to Hairer's original approach, 
a Hopf-algebraic structure lies behind the structure group. 
Moreover, working with a dual perspective allowed for the insight 
that this Hopf algebra arises from a natural Lie algebra: 
the generators of actions on nonlinearities and solutions. 
While this provides a deeper insight into regularity structures, 
it is logically not necessary for \cite{OSSW21}, \cite{LOTT21} or this work. 

In \cite{LOTT21}, the construction and stochastic estimates of the 
model were provided, which serves as an input for \cite{OSSW21}. 
Although the main result can be compared (in much less generality, of course) 
to the one in \cite{CH16}, 
the technique to obtain this result is drastically different. 
It is based on Malliavin calculus and a spectral gap assumption 
on the noise ensemble, 
which can be seen as a substitute for Gaussian calculus. 
Recently, this technique has been picked up and applied 
to the tree setting in great generality \cite{HS23}, 
which shows the robustness of the method. 
One further difference compared to the works in the framework
of \cite{Hai14}, is that in \cite{LOTT21}
the renormalization procedure is top-down rather than 
bottom-up. 
By this we mean that, guided by symmetries, we postulate a counterterm 
in the equation with as little degree of freedom as possible. 
The counterterm is then translated to the model, 
and the remaining degree of freedom is fixed\footnote{
That there is no choice left is specific to the assumption of a spectral gap inequality with $\alpha\not\in\Q$ in combination with estimating 
small and large scales.} 
by the BPHZ-choice of renormalization. 
This strategy has the advantage that symmetries are built in, 
which comes at the price that it is a-priori not clear whether the 
equation can be renormalized with such a restricted counterterm. 

Although the Malliavin derivative may be seen as a mere technical tool 
in \cite{LOTT21}, 
it actually allows for a conceptual gain: 
it allows for a characterization of models which persists for rough noise,
as we shall establish in this work. 
This is in contrast to the existing literature on regularity structures, 
where the model is only tangible for smooth approximations, 
but not in the limit as the smoothness fades away. 
We believe that this characterization will prove useful 
in establishing qualitative properties of solutions of singular SPDEs, 
that are not present on the level of approximations. 
We think in particular of lattice approximations, 
as e.g.~obtained in \cite{MW17, SW18} for two dimensional 
Ising-like models and in \cite{GMW23} for three dimensional ones, 
which properly rescaled and with suitable interaction converge to 
the dynamical $\Phi^{2n}_2$ resp.~$\Phi^4_3$ model;
for further results on discrete approximations we refer to \cite{EH21} and references therein. 
A first simple application is given in Corollary~\ref{lem:covariant}, 
where we show scale invariance of the model. 
We also think that the method is robust 
and can be applied to other equations; 
in Remark~\ref{rem:other_equations} we provide two examples, 
a multiplicative stochastic heat equation 
and a version of the stochastic thin-film equation, 
and we outline what has to be modified for those equations. 
We also expect that weak universality results can be obtained,
which we aim to address in future work. 

In the following we briefly summarize the main results obtained in this work.

\subsection*{Main results}

\begin{itemize}
\item {\bf Characterization.} 
In Definition~\ref{def} we make precise what 
we mean by a \emph{model}. 
With help of the Malliavin derivative, 
the definition is weak enough to allow for rough noise, 
but strong enough for a unique characterization. 
In Theorem~\ref{thm} we establish existence and uniqueness 
of the model associated to \eqref{spde} 
under a spectral gap assumption on the driving noise. 
\item {\bf Convergence.} 
If a sequence of noise ensembles converges 
in law (resp.~in probability resp.~in $L^p$), 
and satisfies uniformly a spectral gap inequality, 
then the corresponding models obtained via Theorem~\ref{thm} converge 
in law (resp.~in probability resp.~in $L^p$), 
see Theorem~\ref{thm2} for a precise statement. 
\item {\bf Universality.} 
The limit obtained in the previous item 
is independent of the approximating sequence. 
This is an immediate consequence of combining Theorem~\ref{thm} and Theorem~\ref{thm2}, 
and establishes universality in the class of (not necessarily Gaussian) 
noise ensembles satisfying uniformly a spectral gap inequality. 
We want to emphasize that such situations 
(in particular for non-Gaussian approximations) 
naturally appear in applications, 
e.g.~in thin ferromagnetic films whose idealized 
physical model is triggered by white noise, 
but the actual physical approximation is not Gaussian due to the 
polycrystallinity of the material \cite{IO19,IORT20}. 
\item {\bf Lifting symmetries.} 
The characterizing Definition~\ref{def} 
allows to propagate symmetries of the noise ensemble to the associated model, 
cf.~Corollary~\ref{lem:covariant}. 
This is of particular interest for symmetries that are not preserved 
under approximation, e.g.~scale invariance for smooth noise. 
\end{itemize}
Let us now comment on related work. 
A characterization of models in the rough setting 
as obtained in this work has not been 
carried out in the existing literature on regularity structures, 
although a characterization by fixing the expectation, the BPHZ-model, is available, 
see \cite[Theorem~6.18]{BHZ19}.
However, this BPHZ characterization is restricted to the setting of smooth random models, 
and depends on a (non-canonical!) truncation of the solution operator. 
In view of the recent work \cite{HS23} 
that exploits the Malliavin derivative in the tree setting, 
it is conceivable that the characterization of the present work 
can also be achieved in the tree setting, 
possibly by appealing to the dictionary recently developed in \cite{BN22}.
Let us point out that convergence of models was already obtained 
in \cite{Hai14} for specific equations, 
and in \cite{CH16} in great generality. 
However, the convergence and universality obtained there is within the class 
of noise ensembles satisfying uniform cumulant bounds. 
A similar convergence and universality result was obtained in \cite{Duc21} via the renormalization group flow approach,
again in the realm of cumulant bounds. 
Let us also mention the work \cite{FG19ii}, 
which establishes a weak universality result for a class of reaction diffusion equations (driven by Gaussian noise) 
by making use of the Malliavin derivative in the framework of paracontrolled distributions \cite{GIP15}; 
for further results on weak universality we refer to references therein. 
A first universality result in the context of a spectral gap assumption 
was obtained in \cite{IORT20} for a mildly singular SPDE; 
it relies on a characterization of the limiting model 
by suitable commutator estimates 
instead of the Malliavin derivative. 
Although such commutator estimates could also serve for a characterization of the model associated to \eqref{spde}, 
we believe that the approach via the Malliavin derivative pursued here 
is more suitable for extensions to other equations. 

\begin{notation}
We denote a generic space-time point by $x=(x_0,\dots,x_d)\in\R^{1+d}$, 
where we think of $0$ as the time-like coordinate and 
$1,\dots,d$ for $d\geq1$ as the space-like coordinates. 
If not further specified, we take suprema over space-time points over all of $\R^{1+d}$. 
A typical element of $\N_0$ is denoted by $k$, where we will often 
write $k\geq0$ instead of $k\in\N_0$. 
Similarly, $\n$ denotes a typical element of $\N_0^{1+d}$, 
where we shall write $\n\neq\0$ instead of $\n\in\N_0^{1+d}\setminus\{\0\}$.
Partial derivatives are denoted by 
$\partial^\n=\partial_0^{n_0}\cdots\partial_d^{n_d}$, 
and for the spatial Laplacian we write $\Delta=\partial_1^2+\dots+\partial_d^2$.  
In line with the parabolic operator $(\partial_0-\Delta)$ we measure
length with respect to 
\begin{equation}
|x|=\sqrt{|x_0|}+|x_1|+\cdots+|x_d| \, , 
\end{equation}
so that the effective dimension of $\R^{1+d}$ equipped with this distance is given by $D=2+d$. 
By $\|\cdot\|_{\dot H^s}$ we denote an anisotropic version of the 
homogeneous Sobolev norm of order $s\in\R$:
\begin{equation}\label{eq:sob}
\|f\|_{\dot H^s}:=\Big( \int_{\R^{1+d}} dx\, 
\big( (-\partial_0^2+\Delta^2)^\frac{s}{4} f(x)\big)^2 \Big)^\frac{1}{2}\, . 
\end{equation}
We will write $\eqref{estPi}_\beta$ when we refer to the corresponding statement for a specific multi-index $\beta$, 
and sometimes write $\eqref{eh14}_\beta^\gamma$ when we want to specify also a second multi-index. 
By $\lesssim$ like in \eqref{Pi-} we mean $\leq C$ for a constant $C<\infty$, 
where the constant may depend on the integrability exponent $p<\infty$ and
the multi-index $\beta$, 
but is uniform in $x,y\in\R^{1+d}$ and the mollification length scale $t>0$ 
(see \eqref{psi} for the definition of the scaled test function $\psi_t$).
\end{notation}


\subsection{The centered model}\label{sec:model}

In this section we recall in a rather heuristic discussion 
the notion of centered model from \cite{LOTT21}, 
which should serve as a motivation for Definition~\ref{def}. 
Most of it is a repetition of \cite[Section~2.2]{LOTT21} and \cite[Section~5]{OST23}, 
however we prefer to provide a self contained introduction. 
The reader familiar with this setting may directly jump to Section~\ref{sec:stable}. 
For more details and rigorous arguments we refer the reader to the aforementioned references or \cite{GT23}, 
and for an introduction and additional motivation we recommend to have a look at \cite{OST23}. 
We provide a brief comparison to the model in \cite{Hai14} in Remark~\ref{rem:comparison}.

We follow a top-down renormalization ansatz, meaning that 
we replace $\xi$ by a smooth version $\xi_\tau$, 
and guided by symmetries we postulate a counterterm $h$ in \eqref{spde}: 
\begin{equation}
(\partial_0-\Delta) u = a(u)\Delta u + \xi_\tau - h \, .
\end{equation}
The aim is to control the solution manifold 
in the limit of vanishing smoothness $\xi_\tau\to\xi$, 
by fine-tuning the counterterm $h$ 
in a $\tau$ dependent way that preserves as much symmetries of the 
original equation \eqref{spde} as possible. 
We are interested in ensembles $\xi$ (and $\xi_\tau$) that are stationary 
in space-time and reflection invariant in space, 
see Assumption~\ref{ass} i) and ii), 
which suggests a counterterm of the form 
\begin{equation}
h(u(x))=c[a(\cdot+u(x))]
\end{equation}
for some deterministic functional $c$ on the space of nonlinearities $a$. 
For such a modified equation, it is expected that for fixed $a$, 
the space of solutions $u$ (up to constants, and satisfying the equation modulo space-time analytic functions) is parametrized by 
space-time analytic functions
$p$ satisfying $p(x)=0$ for a distinguished point $x\in\R^{1+d}$. 
The model components $\Pi_{x\beta}$ are then introduced by the formal power series ansatz 
\begin{equation}
u(y)-u(x) 
= \sum_\beta \Pi_{x\beta}[\xi](y) 
\prod_{k\geq0} \Big(\frac{1}{k!}\frac{d^k a}{du^k}(u(x))\Big)^{\beta(k)}
\prod_{\n\neq\0} \Big(\frac{1}{\n!}\partial^\n p(x)\Big)^{\beta(\n)} \, ,
\end{equation}
where we sum over multi-indices $\beta:\N_0 \, \dot{\cup} \, (\N_0^{1+d}\setminus\{\0\})\to\N_0$. 
Introducing for $k\geq0$ and $\n\neq\0$ the coordinates 
\begin{equation}
\z_k[a]:= \frac{1}{k!}\frac{d^k a}{du^k}(0), \quad 
\z_\n[p]:=\frac{1}{\n!} \partial^\n p(0) 
\end{equation}
on the space of nonlinearities $a$ and space-time analytic functions $p$ with $p(0)=0$,
and for multi-indices $\beta$ as above the monomials
\begin{equation}
\z^\beta := \prod_{k\geq0} \z_k^{\beta(k)} \prod_{\n\neq\0} \z_\n^{\beta(\n)} \, ,
\end{equation}
the previous power series ansatz can be more compactly written as 
\begin{equation}\label{ansatz}
u(y)-u(x) 
= \sum_\beta \Pi_{x\beta}[\xi](y) \, \z^\beta[a(\cdot+u(x)),p(\cdot+x)-p(x)] \, .
\end{equation}
From this ansatz one can deduce for multi-indices 
$\beta\neq0$ satisfying $\beta(k)=0$ for all $k\geq0$ that 
\begin{equation}\label{polypart}
\Pi_{x\beta}(y) = 
\left\{
\begin{array}{cl}
(y-x)^\n & \text{if }\beta=e_\n, \\
0 & \text{otherwise,}
\end{array}
\right.
\end{equation}
where $e_\n$ denotes the multi-index that maps $k\geq0$ to $0$ and $\m$ to $\delta_\n^\m$. 
For all other multi-indices, 
one can together with the equation for $u$ read off an equation for the coefficients $\Pi_{x\beta}$ 
\begin{subequations}\label{Pi-Pi}
\begin{align}
(\partial_0-\Delta)\Pi_{x\beta} &= \Pi^-_{x\beta} \, , \quad\text{where} \label{linPDE}\\
\Pi^-_{x\beta} &= \big( \sum_{k\geq0} \z_k \Pi_x^k \Delta\Pi_x + \xi_\tau 1
- \sum_{l\geq0} \tfrac{1}{l!} \Pi_x^l (D^{(\0)})^l c \big)_\beta \, , \label{defPi-components} 
\end{align}
\end{subequations}
which is an infinite hierarchy of linear PDEs as we shall see below. 
First, let us comment on the individual components of this PDE. 
By $\Pi_x$ we denote the formal power series whose coefficients are the 
space-time functions $\Pi_{x\beta}$, 
i.e.~$\Pi_x(y) = \sum_\beta \Pi_{x\beta}(y) \, \z^\beta\in\R[[\z_k,\z_\n]]$. 
Products have to be understood in the algebra $\R[[\z_k,\z_\n]]$, 
where $1$ denotes the neutral element. 
Similar to $\Pi_x$, the counterterm $c$ can -- at least formally -- 
be expressed as $c[a]=\sum_\beta c_\beta \, \z^\beta[a]$; 
we thus write $c=\sum_\beta c_\beta \, \z^\beta\in\R[[\z_k]]$.
By $D^{(\0)}$ we denote the derivation on $\R[[\z_k,\z_\n]]$ defined by 
\begin{equation}\label{D0}
D^{(\0)}:=\sum_{k\geq0} (k+1) \z_{k+1} \partial_{\z_k} \, ,
\end{equation}
which is made such that for $v\in\R$ and polynomial $c\in\R[\z_k]$ we have 
\begin{equation}\label{shift}
c[a(\cdot+v)] =  \big( \sum_{l\geq0}\tfrac{1}{l!} v^l (D^{(\0)})^l c\big)[a] \, ,
\end{equation}
see the more general \eqref{transformation_hopf}.
Together with $a(v)=\sum_{k\geq0} \z_k[a] v^k$ 
which is a consequence of the definition of $\z_k$, 
this should serve as enough motivation for how \eqref{defPi-components} arises from the equation for $u$. 
Let us point out that all seemingly infinite sums over $k,l\geq0$ 
in the previous expressions are in fact effectively finite sums, 
meaning that they are finite sums on the level of every $\beta$-component. 

We turn to the notion of \emph{homogeneity}, 
which arises from a formal scaling argument. 
Denoting by $S$ the parabolic rescaling $x\mapsto(s^2x_0,sx_1,\dots,sx_d)$ for $s>0$, we define $\hat{u}=s^{-\alpha}u(S\cdot)$, 
$\hat{a}=a(s^\alpha\cdot)$ and $\hat{\xi}=s^{2-\alpha}\xi(S\cdot)$. 
It is then straightforward to check that if $(u,a,\xi)$ is a solution to \eqref{spde}, 
then $(\hat{u},\hat{a},\hat{\xi})$ is again a solution to \eqref{spde}. 
Postulating that the renormalized equation inherits this invariance, 
and rescaling $p$ consistently with $u$ by $\hat{p}=s^{-\alpha}p(S\cdot)$, 
one can read off the power series ansatz \eqref{ansatz}, that 
\begin{equation}\label{scaling}
\Pi_{x\beta}[s^{2-\alpha}\xi(S\cdot)](y) 
= s^{-|\beta|} \Pi_{Sx \beta}[\xi](Sy) \, , 
\end{equation}
where the homogeneity $|\beta|$ is given by 
\begin{equation}\label{homogeneity}
|\beta|:=\alpha\Big(1+\sum_{k\geq0}k\beta(k)-\sum_{\n\neq\0}\beta(\n)\Big)
+\sum_{\n\neq\0} |\n|\beta(\n), 
\quad
|\n|:=2n_0+n_1+\cdots+n_d \, .
\end{equation}
Here, one should think of $1+\sum k\beta(k)-\sum\beta(\n)$ as 
the number of instances of $\xi$ contained in $\Pi_{x\beta}$ 
(in form of a multilinear expression), 
while $\sum|\n|\beta(\n)$ is the (parabolic) degree of polynomials contained in $\Pi_{x\beta}$. 

In view of this, it is natural to consider only components $\Pi_{x\beta}$ 
that contain at least one instance of $\xi$, 
together with the monomials from \eqref{polypart}, 
which motivates the following definition. 
We call a multi-index $\beta$ \emph{populated}, if 
\begin{equation}
[\beta]\geq0 \quad\text{or}\quad\beta=e_\n\text{ for some }\n\neq\0 \, ,
\end{equation}
where we call the latter \emph{purely polynomial}, 
and $[\beta]$ is defined by 
\begin{equation}\label{noisehomogeneity}
[\beta]:=\sum_{k\geq0}k\beta(k)-\sum_{\n\neq\0}\beta(\n) \, .
\end{equation}
We therefore consider $\Pi_x$ as a $\T$-valued map, 
where $\T\subset\R[[\z_k,\z_\n]]$ is defined by 
\begin{equation}\label{T}
\T:=\big\{\textstyle\sum_\beta\pi_\beta\,\z^\beta \in\R[[\z_k,\z_\n]] \, 
\big| \, \pi_\beta\neq0\implies\beta\text{ populated}\big\} \, .
\end{equation}
For later use we also define the purely polynomial part $\Tbar$ of $\T$ by 
\begin{equation}
\Tbar:=\big\{\textstyle\sum_\beta\pi_\beta\,\z^\beta \in\R[[\z_k,\z_\n]] \, 
\big| \, \pi_\beta\neq0\implies\beta\text{ purely polynomial}\big\} \, , 
\end{equation}
and its natural complement $\Ttilde$ in $\T$ such that 
\begin{equation}
\T = \Tbar\oplus\Ttilde \, .
\end{equation}
Similar to $\Pi_x$, we want to consider $\Pi^-_x$ as taking values in $\T$, 
where we observe the following: 
on the one hand $\pi,\pi'\in\T$ does not\footnote{
e.g. $\z_{\n_1},\z_{\n_2}\in\T$ but $\z_1\z_{\n_1}\z_{\n_2}\not\in\T$} 
imply $\z_k\pi^k\pi'\in\T$, however if the product is contained in $\T$, 
then due to the presence of $\z_k$ it automatically belongs to $\Ttilde$; 
on the other hand, one can check that $\pi\in\T$ and $c\in\R[[\z_k]]\subset\Ttilde$ does imply $\pi^l (D^{(\0)})^l c\in\Ttilde$. 
For that reason, we introduce the projection $P$ from $\R[[\z_k,\z_\n]]$ to $\Ttilde$ 
in the definition of $\Pi^-_x$ to obtain the $\Ttilde$-valued map 
\begin{equation}\label{defPi-}
\Pi^-_x = P\sum_{k\geq0}\z_k\Pi^k_x\Delta\Pi_x + \xi_\tau 1 
-\sum_{l\geq0}\tfrac{1}{l!} \Pi^l_x (D^{(\0)})^l c \, .
\end{equation}

Let us also mention that similar to \eqref{scaling}, 
the invariance of \eqref{spde} under $\hat{u}=u(\cdot+z)$, 
$\hat{a}=a$ and $\hat{\xi}=\xi(\cdot+z)$ for $z\in\R^{1+d}$
propagates via \eqref{ansatz} to $\Pi_x$, yielding 
\begin{equation}\label{translation}
\Pi_{x \beta} [\xi(\cdot+z)](y) 
= \Pi_{x+z \, \beta}[\xi](y+z) \, .
\end{equation}
Thinking of a noise $\xi$ whose law is invariant under translation 
$\xi\mapsto\xi(\cdot+z)$ and the rescaling 
$\xi\mapsto s^{2-\alpha}\xi(S\cdot)$, 
we expect by \eqref{translation} that the law of 
$|y-x|^{-|\beta|} \Pi_{x\beta}(y)$ coincides with the law of 
$|y-x|^{-|\beta|} \Pi_{0\beta}(y-x)$, 
which by \eqref{scaling} coincides with 
the law of $\Pi_{0\beta}\big(\tfrac{y-x}{|y-x|}\big)$.
Since $\Pi_{x\beta}$ is expected to be a (H\"older-) 
continuous function and the range of $(y-x)/|y-x|$ is compact, 
this indicates that 
\begin{equation}\label{estPi}
\sup_{x\neq y} |y-x|^{-|\beta|} \, \E^\frac{1}{p}| \Pi_{x\beta}(y) |^p <\infty \, .
\end{equation}

We turn to the reexpansion maps $\G$, 
where we will mainly follow \cite{OST23}. 
Recall that the distinguished point $x$ that we fixed for the 
power series ansatz led to $\Pi_x$. 
Of course, we could as well have chosen any other point $y$, 
leading to $\Pi_y$. 
It is then clear from 
\begin{equation}\label{eh48}
u = \left\{
\begin{array}{l}
u(x)+ \Pi_x [a(\cdot+u(x)),p_x] \\
u(y)+ \Pi_y \, [a(\cdot+u(y)),p_y] 
\end{array}
\right.
\end{equation}
for some polynomials $p_x,p_y$ vanishing at the origin, 
that there is an algebraic relation between $\Pi_x$ and $\Pi_y$. 
Indeed, we recall from \cite[Proposition~5.1]{LOT23} that if we define $\G_{xy}$ by 
\begin{align}
\G_{xy} &= \sum_{l\geq0}\tfrac{1}{l!} \Pi_x^l(y) (D^{(\0)})^l 
\quad \text{on}\quad \R[[\z_k]] \, , \label{Gamma_Ttilde} \\
\G_{xy}\z_\n &= \z_\n+\pi^{(\n)}_{xy} \quad\text{for some}\quad \pi^{(\n)}_{xy}\in\T \, , \label{Gamma_zn}
\end{align}
and imposing that $\G_{xy}$ is an algebra endomorphism\footnote{
that is, for $\pi,\pi'\in\R[[\z_k,\z_\n]]$ it holds $\G_{xy}\pi\pi'=(\G_{xy}\pi)(\G_{xy}\pi')$} 
of $\R[[\z_k,\z_\n]]$, then formally 
\begin{equation}\label{transformation_hopf}
\G_{xy} \pi[a,p] = \pi\big[a\big(\cdot+\Pi_x[a,p](y)\big), 
p + \sum_{\n\neq\0}\pi_{xy}^{(\n)}[a,p] (\cdot)^\n\big] \, ,
\end{equation}
for functionals $\pi$ of pairs $(a,p)$. 
Hence we see from subtracting the two equations in \eqref{eh48} 
and the ansatz \eqref{ansatz}, 
that for a suitable choice of $\pi^{(\n)}_{xy}$ we have 
\begin{equation}\label{recenter}
\Pi_x = \G_{xy}\Pi_y + \Pi_x(y) \, . 
\end{equation}
Note that $\G_{xy}$ being an algebra endomorphism is compatible with \eqref{Gamma_Ttilde}, 
and because of that \eqref{Gamma_Ttilde} may be equivalently expressed as 
\begin{equation}\label{Gamma_zk}
\G_{xy}\z_k = \sum_{l\geq0}\tbinom{k+l}{k} \z_{k+l} \Pi^l_x(y) \, . 
\end{equation}
We want to point out that it is not always possible to extend 
a linear map from the coordinates $\z_k,\z_\n$ in a multiplicative 
way to the space of formal power series $\R[[\z_k,\z_\n]]$ 
-- as opposed to extending from $\z_k,\z_\n$ to the space of polynomials $\R[\z_k,\z_\n]$. 
That this is possible in our specific setting relies on the restriction of $\pi^{(\n)}_{xy}$ to 
\begin{equation}\label{rest_pin}
\pi^{(\n)}_{xy\beta} \neq 0 \quad \implies \quad |\n|<|\beta| \, , 
\end{equation}
which is in line with the required recentering, and which we shall assume; 
a proof of this fact can be found in \cite[Lemma~3]{OST23}. 
It is also worth mentioning that $\G_{xy}$ preserves $\T$ and $\Ttilde$ in the sense of 
\begin{equation}\label{Gamma_preserves}
\G_{xy} \T \subset \T\, , \quad 
\G_{xy} \Ttilde \subset \Ttilde \, , 
\end{equation}
which is a consequence of $\Pi_x(y), \pi^{(\n)}_{xy}\in\T$, 
see \cite[(2.59)]{LOTT21}. 

Let us also mention that the recentering \eqref{recenter} of $\Pi_x$ 
is inherited by $\Pi^-_x$ through \eqref{defPi-}. 
Due to the presence of the projection $P$ in $\Pi^-_x$, 
it takes the form of 
\begin{equation}\label{recenter-}
\Pi^-_x = \G_{xy} \Pi^-_y 
+ P \sum_{k\geq0} \z_k \big(\G_{xy}(\id-P)\Pi_y + \Pi_x(y)\big)^k \, 
\G_{xy}(\id-P)\Delta\Pi_y \, ,
\end{equation}
cf.~\cite[(2.63)]{LOTT21}. 
Here, $\id-P$ takes the role of projecting to $\Tbar$, 
i.e.~$(\id-P)\Pi_y = \sum_{\n\neq\0} \z_\n (\cdot-y)^\n$, 
which can be used to see\footnote{
for a detailed argument see \cite[(2.60) and (2.63)]{LOTT21}} 
that the $\beta$-component 
of the second right-hand side term of \eqref{recenter-} 
is a polynomial of degree $\leq|\beta|-2$, i.e. 
\begin{equation}\label{recenter-diff}
(\Pi^-_x - \G_{xy} \Pi^-_y)_\beta \ 
\text{ is a polynomial of degree} \leq|\beta|-2 \, .
\end{equation}
Furthermore, the scaling and translation invariances \eqref{scaling} and \eqref{translation} 
together with the recentering \eqref{recenter} suggest 
for scaling and translation invariant noise that 
the law of $(\G_{xy})_\beta^\gamma$ coincides with the one of 
$(\G_{x+z \, y+z})_\beta^\gamma$, and the law of 
$s^{|\beta|-|\gamma|} (\G_{xy})_\beta^\gamma$ coincides with the one of 
$(\G_{SxSy})_\beta^\gamma$. 
Here, $(\G)_\beta^\gamma$ denote the matrix coefficients of $\G$ w.r.t.~the monomials $\{\z^\beta\}_\beta$ in the sense of 
\begin{equation}
(\G)_\beta^\gamma = (\G \z^\gamma)_\beta \, .
\end{equation}
Since $\G_{SxSy}$ converges to $\G_{00}$ as $s\to0$, 
and we expect $\G_{00}$ to be the identity, 
this indicates that $\G_{xy}$ is strictly triangular in the sense of 
\begin{equation}\label{triGamma_hom}
(\G_{xy}-\id)_\beta^\gamma \neq 0 
\quad\implies\quad 
|\gamma|<|\beta| \, . 
\end{equation}
By a similar argument that led to the estimate \eqref{estPi} for $\Pi$,
we expect 
\begin{equation}\label{eh14}
\sup_{x\neq y} |x-y|^{-(|\beta|-|\gamma|)} \, 
\mathbb{E}^\frac{1}{p}|(\G_{xy})_{\beta}^{\gamma}|^p <\infty \, .
\end{equation}
%

\subsection{Seeking for a stable characterization}\label{sec:stable}

At this point we want to recapitulate, 
and clarify what of the above 
can potentially survive the limit of vanishing smoothness $\xi_\tau\to\xi$. 
For smooth noise $\xi_\tau$ we obtained $\Pi_x$, 
of which the purely polynomial components are given by \eqref{polypart}, 
and the remaining populated and not purely polynomial components 
are determined\footnote{
It is legitimate to use the term ``determined'' here, 
since there is a unique solution to \eqref{linPDE} 
in the class of $\Pi_{x\beta}$ satisfying \eqref{estPi}, 
provided $\Pi^-_{x\beta}$ satisfies an appropriate growth condition 
which is satisfied here, see \cite[Lemma~1]{OST23}. } 
by \eqref{Pi-Pi}. 
Moreover, we have introduced the reexpansion maps $\G_{xy}$ 
that recenter $\Pi$ and $\Pi^-$ according to \eqref{recenter} and \eqref{recenter-}. 

Clearly, \eqref{polypart}, \eqref{linPDE} and \eqref{recenter} 
do make sense for rough noise $\xi$ as well, 
where in line with Assumption~\ref{ass}~iii) we think of $\xi$ as having $(\alpha-2)$-H\"older continuous realizations for $\alpha\in(0,1)$, 
and expect correspondingly $\Pi_{x\beta}$ and $\Pi^-_{x\beta}$ to be 
$\alpha$ respectively $(\alpha-2)$-H\"older continuous. 
Also \eqref{recenter-} does not cause any problems in the rough setting: 
although $\Pi^-_{x\beta}$ is expected to be a distribution in this case, 
the first right-hand side term $\G_{xy}\Pi^-_y$ of \eqref{recenter-} 
is component wise a well-defined linear combination of distributions, 
while the second right-hand side term of \eqref{recenter-} 
contains only products of smooth objects due to the presence of the projection $\id-P$. 
However, even if $\Pi^-_{x\beta}$ can be given a sense for rough noise, 
there is no way that its individual constituents in \eqref{defPi-components} 
survive a limiting procedure, 
except for $\beta=0$ where
\begin{equation}
\Pi^-_{x\beta=0} = \xi_\tau - c_0 \, .
\end{equation}
Nevertheless, we claim that the transformation behaviour \eqref{recenter-} 
of $\Pi^-_x$ is enough to provide a stable characterization, 
at least for multi-indices $\beta$ with $|\beta|>2$. 
Indeed, observe that \eqref{estPi} implies 
\begin{equation}
\E^\frac{1}{p}| \Pi^-_{y\beta} (\phi_y^\lambda) |^p 
\lesssim \lambda^{|\beta|-2} \, , 
\end{equation}
where $\phi_y^\lambda$ is a Schwartz function, 
recentered at $y\in\R^{1+d}$ and rescaled by $\lambda>0$. 
We will thus informally write $\Pi^-_{y\beta}(y) = 0$ for $|\beta|>2$. 
With this, we observe that $(\G_{xy}\Pi^-_y)_\beta(y) = ((\G_{xy}-\id)\Pi^-_y)_\beta(y)$, 
where by the triangularity \eqref{triGamma_hom} the right-hand side involves only $\Pi^-_{y\gamma}$ 
for $|\gamma|<|\beta|$. 
Provided also the second right-hand side term of \eqref{recenter-} 
involves only components $\Pi_{\gamma}$ for $\gamma$ ``smaller'' than $\beta$, 
we see that \eqref{recenter-} determines $\Pi_{x\beta}(y)$ in a recursive way. 
This argument will be made rigorous in Step~2 of the proof of Proposition~\ref{unique}. 

For multi-indices $\beta$ with $|\beta|<2$, the above argument fails 
since one cannot pass from $\G_{xy}$ to $\G_{xy}-\id$ and benefit from strict triangularity anymore. 
This is where we will profit from the Malliavin derivative as we shall argue now. 
We will introduce the Malliavin derivative rigorously in Section~\ref{sec:ass}, 
for the moment we think of $\delta$ as being the operation of
taking the derivative of an object like $\Pi_{x\beta}(y)$ 
as a functional of $\xi$, in the direction of an infinitesimal perturbation $\delta\xi$. 
Applying $\delta$ to $\Pi^-_{x\beta}(y)=(\G_{xy}\Pi^-_y)_\beta(y)$, 
which is a consequence of \eqref{recenter-diff}, we obtain by Leibniz' rule
\begin{equation}
\delta\Pi^-_{x\beta}(y) 
= (\delta\G_{xy}\Pi^-_y)_\beta(y) + (\G_{xy}\delta\Pi^-_y)_\beta(y) \, .
\end{equation}
The first right-hand side term involves only $\Pi^-_{y\gamma}$ for $|\gamma|<|\beta|$, 
since the triangularity \eqref{triGamma_hom} is inherited by $\delta\G_{xy}$, 
and therefore allows for a recursive argument. 
To investigate the second right-hand side term, 
we first note that $\Pi_y(y)=0$ as a consequence of the ansatz \eqref{ansatz}, 
which together with \eqref{defPi-components} yields 
\begin{equation}
\delta\Pi^-_y (y) = \z_0\Delta\delta\Pi_y (y) + \delta\xi_\tau(y)1 \, ,
\end{equation}
where we have used that $c$ is deterministic and hence $\delta c=0$, 
and we have dropped the projection $P$ since we will use this identity 
only for $\gamma$-components with $|\gamma|<2$ where $P$ is inactive anyway. 
Applying $\G_{xy}$ to this identity, 
we obtain since $\G_{xy}$ is multiplicative, 
satisfies $\G_{xy}\z_0=\sum_k \z_k \Pi_x^k(y)$ 
as a consequence of \eqref{Gamma_Ttilde}, 
and from \eqref{recenter} which implies by Leibniz' rule
$\delta\Pi_x = \delta\G_{xy}\Pi_y+\G_{xy}\delta\Pi_y+\delta\Pi_x(y)$, 
that 
\begin{equation}
\G_{xy}\delta\Pi^-_y(y) 
= \sum_{k\geq0} \z_k \Pi_x^k(y) \Delta \big(\delta\Pi_x - \delta\Pi_x(y) - \delta\G_{xy}\Pi_y\big)(y)
+ \delta\xi_\tau(y)1 \, .
\end{equation}
This right-hand side has again the potential for a recursive argument, 
however it contains the problematic product of $\Pi_x^k$ and 
$\Delta(\delta\Pi_x-\delta\Pi_x(y)-\delta\G_{xy}\Pi_y)$. 
Here we note that applying $\delta$ to a multilinear expression like $\Pi_x$ 
amounts to replacing one instance of $\xi$ by $\delta\xi$. 
Although $\delta\xi$ gains $D/2$ orders of regularity compared to $\xi$, 
see \eqref{directional_derivative} and \eqref{MallSobNorm}, 
$\delta\Pi$ does not benefit from this gain of regularity 
since there might be other instances of $\xi$ left. 
However, one can expect that $\delta\Pi$ is modelled to order $\alpha+D/2$ in terms of $\Pi$. 
For the above product to be given a sense in the rough setting, 
we would then need that $\alpha + (\alpha+D/2-2)>0$, 
which is the reason for the restriction $\alpha>1-D/4$ in Assumption~\ref{ass}. 
Unfortunately, $\delta\G$ is not the right object to express this gain in modelledness by $D/2$: 
since $\alpha+D/2>1$, every $\beta$-component of $\delta\G_{xy}\Pi_y$ 
would have to contain the first order monomials $\Pi_{y e_\n}=(\cdot-y)^\n$ for $|\n|=1$, 
which contradicts the triangularity of $\delta\G_{xy}$ with respect to 
the homogeneity $|\cdot|$ for multi-indices $|\beta|<1$. 

We therefore aim to replace $\delta\G_{xy}$ by a map that allows
for this gain of modelledness, 
while keeping the necessary algebraic properties of $\delta\G_{xy}$ that led to 
\begin{equation}
\delta\Pi^-_x(y) 
= \delta\G_{xy}\Pi^-_y(y)
+ \sum_{k\geq0} \z_k \Pi_x^k(y) \Delta \big(\delta\Pi_x - \delta\Pi_x(y) - \delta\G_{xy}\Pi_y\big)(y)
+ \delta\xi_\tau(y)1 \, .
\end{equation}
The latter can be achieved for any $\d\G_{xy}\in{\rm End}(\R[[\z_k,\z_\n]])$ 
that coincides with $\delta\G_{xy}$ on $\R[[\z_k]]$, 
and that satisfies for $\pi,\pi'\in\R[[\z_k,\z_\n]]$ 
\begin{equation}
\d\G_{xy}(\pi\pi') 
= \d\G_{xy}\pi \, \G_{xy}\pi' + \G_{xy}\pi \, \d\G_{xy}\pi' \, .
\end{equation}
Indeed, applying $\d\G_{xy}$ to $\Pi^-_y(y) = \z_0\Delta\Pi_y(y) +\xi_\tau(y) 1 - c$, we obtain 
\begin{equation}
\d\G_{xy}\Pi^-_y(y) 
= \d\G_{xy}\z_0 \, \Delta\G_{xy}\Pi_y(y) 
+ \G_{xy}\z_0 \, \Delta\d\G_{xy}\Pi_y(y) 
- \d\G_{xy} c \, .
\end{equation}
By \eqref{Gamma_Ttilde}, to which we apply $\delta$, 
and by \eqref{recenter} which implies $\Delta\G_{xy}\Pi_y=\Delta\Pi_x$, 
the first right-hand side term coincides with 
$\sum_k \z_k \delta(\Pi_x^k(y)) \Delta\Pi_x(y)$. 
Since $c\in\R[[\z_k]]\subset\Ttilde$, the same argumentation leads to 
$\d\G_{xy} c = \sum_{l} \tfrac{1}{l!} \delta(\Pi_x^l(y)) (D^{(\0)})^l c$. 
Again by \eqref{Gamma_Ttilde}, we see that the second right-hand side term coincides with 
$\sum_{k} \z_k \Pi_x^k(y) \Delta\d\G_{xy}\Pi_y(y)$, 
which altogether yields the desired 
\begin{equation}\label{magic}
\delta\Pi^-_x(y) 
= \d\G_{xy}\Pi^-_y(y)
+ \sum_{k\geq0} \z_k \Pi_x^k(y) \Delta \big(\delta\Pi_x - \delta\Pi_x(y) - \d\G_{xy}\Pi_y\big)(y)
+ \delta\xi_\tau(y)1 \, ,
\end{equation}
which coincides with \cite[(4.50)]{LOTT21}.

By \eqref{Gamma_Ttilde}, we also see that 
$\delta\G_{xy} = \delta\Pi_x(y)\G_{xy}D^{(\0)}$ on $\R[[\z_k]]$, 
hence such a map $\d\G_{xy}$ has to coincide with 
$\delta\Pi_x(y)\G_{xy}D^{(\0)} 
+ \sum_{\n\neq\0} \d\pi^{(\n)}_{xy} \G_{xy} \partial_{\z_\n}$ 
for some $\d\pi^{(\n)}_{xy}\in\R[[\z_k,\z_\n]]$. 
To ease the notation we shall write 
$\sum_\n \d\pi^{(\n)}_{xy} \G_{xy} D^{(\n)}$ with 
$\d\pi^{(\0)}_{xy}=\delta\Pi_x(y)$ and 
\begin{equation}\label{Dn}
D^{(\n)}:=\partial_{\z_\n} \quad\text{for}\quad \n\neq\0 \, .
\end{equation}
We will only ever apply such a map $\d\G_{xy}$ to elements in $\T$ 
that are populated on multi-indices $|\beta|<2$, 
hence we can drop the summands with $|\n|\geq2$. 
To obtain equally nice mapping properties as for $\G_{xy}$ in \eqref{Gamma_preserves}, we will restrict $\d\pi^{(\n)}_{xy}$ 
for $\n\neq\0$ from $\R[[\z_k,\z_\n]]$ to $\Ttilde$, which yields 
together with $D^{(\n)}\T\subset\Ttilde$ and \eqref{Gamma_preserves}
\begin{equation}\label{dGamma_preserves}
\d\G_{xy} \T \subset \Ttilde \, .
\end{equation}

Summarizing this discussion, \eqref{magic} determines $\delta\Pi^-_x(y)$, 
and hence $\Pi^-_x(y)$ up to its expectation by an application of the spectral gap inequality \eqref{sg}, 
provided we can choose $\d\pi^{(\n)}_{xy}\in\Ttilde$ for $|\n|=1$ such that 
\begin{equation}\label{modelledness}
\delta\Pi_x - \delta\Pi_x(y) - \d\G_{xy}\Pi_y 
= O(|\cdot-y|^{\alpha+\frac{D}{2}}) \, , 
\end{equation}
where 
\begin{equation}\label{dGamma}
\d\G_{xy} = \sum_{|\n|\leq1}\d\pi^{(\n)}_{xy}\G_{xy}D^{(\n)}\, . 
\end{equation}
Moreover, we have to come up with an ordering $\prec$ on multi-indices 
with respect to which $\d\G_{xy}$ is strictly triangular, 
which is necessarily different from the homogeneity $|\cdot|$ as we saw above.  
We want to emphasize that \eqref{magic} and \eqref{modelledness} 
have the potential to make sense even in the rough setting: 
neither singular products nor diverging constants appear. 
The remaining difficulty to overcome is to find a suitable 
probabilistic and weak formulation that avoids point evaluations of distributions. 
This will be achieved in Definition~\ref{def} in \eqref{ho29} and \eqref{ho28}. 

Finally, let us point out the following geometric interpretation 
recently given in \cite{BOT24} (which is set up for a semilinear equation with additive noise, however the following applies to the quasilinear case as well).
There, for fixed nonlinearity $a$, the map $p\mapsto\Pi_x[a,p]$ 
is informally interpreted as an inverse chart on the solution manifold of \eqref{spde},
and $\Gamma_{yx}$ as a transition function between two inverse charts $\Pi_x$ and $\Pi_y$. 
Furthermore, $\d\G_{xy}\Pi_y$ maps the parameter-space $p$ 
to the tangent space of the solution manifold of \eqref{spde}.
Hence \eqref{modelledness} expresses that $\delta\Pi_x$ is 
approximately a parameterization of this tangent space. 
Informally, our main result thus characterizes the inverse charts $\Pi_x$ 
via the transition maps $\Gamma_{yx}$ and 
via the tangent space through \eqref{magic} and \eqref{modelledness}. 
This is stable for rough noise as opposed to 
the description of the inverse charts directly via \eqref{Pi-Pi}.


\subsection{Assumption and main results}\label{sec:ass}

In this section we make the heuristic discussion from the previous Section~\ref{sec:model} and Section~\ref{sec:stable} precise. 
We start by introducing the Malliavin derivative. 
Consider cylindrical functionals of the form 
\begin{equation}
F[\xi] = f((\xi,\zeta_1),\dots,(\xi,\zeta_N))
\end{equation}
for some $f\in C^\infty(\R^{1+d})$ and Schwartz functions 
$\zeta_1,\dots,\zeta_N\in\S(\R^{1+d})$, 
where $(\cdot,\cdot)$ denotes the pairing between a tempered distribution and a Schwartz function. 
For such cylindrical functionals we may define 
\begin{equation}\label{mall}
\frac{\partial F}{\partial\xi} [\xi] 
= \sum_{i=1}^N \partial_i f((\xi,\zeta_1),\dots,(\xi,\zeta_N)) \, \zeta_i \, ,
\end{equation}
and for suitable $\delta\xi:\R^{1+d}\to\R$ 
\begin{equation}\label{directional_derivative}
\delta F(\xi;\delta\xi)
= \Big(\frac{\partial F}{\partial\xi}[\xi], \delta\xi \Big)
= \int_{\R^{1+d}} dx\, \frac{\partial F}{\partial\xi}[\xi](x)\delta\xi(x) \, ; 
\end{equation}
we shall typically suppress the dependence on $\xi$ and $\delta\xi$ in the notation and simply write $\delta F$. 
We denote by $\|\cdot\|_*$ the Sobolev norm 
$\|\cdot\|_{\dot{H}^{2-\alpha-D/2}}$ defined in \eqref{eq:sob}, 
and the Malliavin-Sobolev space $\H$ is given by 
the completion of cylindrical functionals with respect to the stochastic norm 
\begin{equation}\label{MallSobNorm}
\|F\|_{\H} = \E^\frac{1}{2}\big( |F|^2
+ \big\|\tfrac{\partial F}{\partial\xi}\big\|_*^2 \big)^2 \, . 
\end{equation}
By $C(\H)$ we denote the linear space of continuous functions 
from $\R^{1+d}$ to $\H$, 
endowed with the topology given by the family of semi-norms 
$\sup_{y\in K} \|F(y)\|_{\H}$ for $K\subset\R^{1+d}$ compact, 
that grow at most polynomially at infinity, i.e.~there exists $n\in\N$ such that 
\begin{equation}\label{polygrowth}
\sup_{y\in\R^{1+d}} (1+|y|)^{-n} \, \|F(y)\|_\H <\infty \, .
\end{equation}

Next, we introduce the semigroup convolution, 
which we shall use as a convenient tool for weak formulations. 
We define the Schwartz function $\psi_t\in\S(\R^{1+d})$ through
\begin{equation}\label{psi}
\begin{split}
\partial_t \psi_t + (\partial_0-\Delta)(-\partial_0-\Delta)\psi_t &= 0, \\
\psi_{0} &= \text{Dirac at origin.}
\end{split}
\end{equation}
We will then typically write $f_t$ as shorthand for $f*\psi_t$. 
Note the semigroup property 
\begin{equation}\label{semigroup}
(f_t)_s = f_{t+s} \, , 
\end{equation}
and the scaling identity 
\begin{equation}\label{scaling_psi}
\psi_t(x) = (\sqrt[4]{t})^{-D} \psi_{t=1}\big(\tfrac{x_0}{(\sqrt[4]{t})^2} , \tfrac{x_1}{\sqrt[4]{t}} \, ,\dots,\tfrac{x_d}{\sqrt[4]{t}}\big) \, ,
\end{equation}
that we shall use frequently. 
From this scaling identity, since $\psi_{t=1}$ is a Schwartz function, 
we derive the moment bound 
\begin{equation}\label{momentbound}
\int_{\R^{1+d}} dz\, |\partial^\n \psi_t(y-z)| 
(\sqrt[4]{t} + |x-y| + |y-z|)^{\theta} 
\lesssim (\sqrt[4]{t})^{-|\n|} (\sqrt[4]{t}+|x-y|)^\theta \, ,
\end{equation}
for all $x,y\in\R^{1+d}$, $t>0$, $\theta>-D$ and $\n$. 

Equipped with this, we have all the tools to formulate Definition~\ref{def}, 
which makes the heuristic discussion from Section~\ref{sec:stable} precise. 

\begin{definition}\label{def}
Given an ensemble $\E$ on the space of tempered distributions $\S'(\R^{1+d})$, 
on which we denote realizations by $\xi$, 
we call the family of random fields $(\Pi,\G)=\{\Pi_{x\beta}(y),(\G_{xy})_\beta^\gamma\}_{x,y,\beta,\gamma}$ 
a \emph{model} for \eqref{spde} if the following holds 
almost surely:
\begin{itemize}[leftmargin=1.5em]
\item The realizations $\xi$ are related to the model by 
\begin{equation}\label{eh11}
(\partial_0-\Delta)\Pi_{x\,\beta=0}=\xi-\E\xi \, . 
\end{equation}
\item $\Pi_x$ takes values in $\T$, 
the purely polynomial part of $\Pi$ is given by \eqref{polypart}, 
its expectation is shift invariant in the sense of 
\begin{equation}\label{expectationshiftinvariant}
\E\Pi_{x\beta}(y)=\E\Pi_{x+z\,\beta}(y+z) \quad\text{for all }x,y,z\in\R^{1+d} \, ,
\end{equation}
and it satisfies for all $p<\infty$ and all multi-indices $\beta$ 
the stochastic estimate \eqref{estPi}. 
\item $\G_{xy}$ is an algebra endomorphism of $\R[[\z_k,\z_\n]]$ 
which satisfies \eqref{Gamma_zn} and \eqref{Gamma_zk}, 
the former for some $\pi^{(\n)}_{xy}\in\T$ subject to \eqref{rest_pin}. 
Furthermore, $\G$ is related to $\Pi$ and $\Pi^-:=(\partial_0-\Delta)P\Pi$ 
via the recentering \eqref{recenter} and \eqref{recenter-}, 
and it satisfies for all $p<\infty$ and all populated $\beta,\gamma$ the stochastic estimate \eqref{eh14}.
\item For singular multi-indices $|\beta|<2$ we have $\Pi_{x\beta}\in C(\H)$.
Moreover, for $\delta\xi\in\S$, $|\n|=1$ and $x\neq y$ there exist random 
$\d\pi^{(\n)}_{xy}\in Q\Ttilde$, 
where $Q$ denotes the projection to multi-indices satisfying $|\beta|<2$, 
such that $\d\G_{xy}$ defined by \eqref{dGamma} 
with $\d\pi^{(\0)}_{xy}=Q\delta\Pi_x(y)$ satisfies 
\begin{equation}
\lim_{z\rightarrow y}|y-z|^{-1} \, \mathbb{E} \big|\big(\delta\Pi_x-\delta\Pi_x(y)
-\d\G_{xy} Q \Pi_y\big)_\beta(z) \big| = 0 \, , \label{ho29} 
\end{equation}
and 
\begin{equation}\label{ho28}
\lim_{t\to0} \E\big|
\big(\delta\Pi^-_x\-\d\G_{xy} Q \Pi^-_{y} 
\-\sum_{k\ge 0}\mathsf{z}_k\Pi_x^k(y)
\Delta (\delta\Pi_x\-\d\G_{xy} Q \Pi_y)
\-\delta\xi\mathsf{1}\big)_{\beta t} (y) \big| =0 \, .
\end{equation}
Note that taking the Malliavin derivative of $\Pi^-_x$ and $\Delta\Pi_x$ 
is well-defined, since $\Pi_{x\beta}\in C(\H)$ implies 
$\Pi^-_{x\beta\,t},\, \Delta\Pi_{x\beta\,t} \in C(\H)$, 
see \cite[Lemma~A.2]{LOTT21}. 
\end{itemize}
\end{definition}

\noindent
To guarantee the existence of such a model, 
we make in line with \cite{LOTT21} the following assumption on the underlying noise ensemble. 

\begin{assumption}\label{ass}
The ensemble $\E$ is a probability measure on the space of tempered distributions $\S'(\R^{1+d})$, that w.l.o.g.~is centered, and 
\begin{itemize}[leftmargin=2em]
\item[i)] is invariant under space-time shifts $x\mapsto x+z$ for $z\in\R^{1+d}$, 
\item[ii)] is invariant under spatial reflections 
$x\mapsto (x_0,x_1,\dots,x_{i-1},-x_i,x_{i+1},\dots,x_d)$ for $i=1,\dots,d$,
\item[iii)] satisfies a spectral gap inequality: 
for $\alpha\in(\max\{0,1-D/4\},1)\setminus\Q$ there exists $C>0$ such that for all integrable cylindrical functionals $F$ holds 
\begin{equation}\label{sg}
\E(F-\E F)^2\leq C\,\E \big\|\tfrac{\partial F}{\partial \xi}\big\|_*^2\, . 
\end{equation}
In addition, we assume that the operator \eqref{mall} is closable 
w.r.t.~the topologies of $\E^\frac{1}{2}|\cdot|^2$ and $\E^\frac{1}{2}\|\cdot\|_*^2$.
\end{itemize}
\end{assumption}

\noindent
Let us briefly comment on the restriction of $\alpha$ in the spectral gap assumption. 
Note that the choice of $\|\cdot\|_*$ is made such that realizations $\xi$ 
will have annealed H\"older regularity $\alpha-2$. 
Hence the restriction $\alpha>0$ is dictated by subcriticality. 
The restriction $\alpha>1-D/4$ is necessary in reconstruction, see \eqref{kappa}, 
and is only present in a single spatial dimension $d=1$. 
Equation \eqref{spde} is not singular in the regime $\alpha>1$, 
thus no renormalization is required and simpler tools can be applied; 
see \cite{LO22} where the framework of the current work is applied to this simpler regime.
Finally, we consider $\alpha\not\in\Q$ to avoid integer exponents in 
Liouville-type arguments, 
which is crucial for our uniqueness result. 
This last restriction is much less dramatic when one is interested 
in model components $\Pi_\beta$ up to a certain threshold, 
e.g.~in the regime $\alpha\in(1/4,1/2)$ all statements remain true 
for all model components $\Pi_\beta$ with $|\beta|<2$, 
provided we avoid that $\alpha,2\alpha,\dots,7\alpha$ are integers. 

Under Assumption~\ref{ass}, the main result is as follows.

\begin{theorem}[Existence and Uniqueness]\label{thm}
Assume that the ensemble $\E$ satisfies Assumption~\ref{ass}. 
Then there exists a unique model $(\Pi,\G)$ for \eqref{spde}. 
\end{theorem}

\noindent
We prove the uniqueness part of Theorem~\ref{thm} in Section~\ref{sec:unique}, 
and the existence part in Section~\ref{sec:exist}. 
We use this existence and uniqueness for the following convergence result, 
the proof of which is given in Section~\ref{sec:convergence}. 

\begin{theorem}[Convergence]\label{thm2}
Assume that a sequence of ensembles $(\E_n)_{n\in\N}$ satisfies Assumption~\ref{ass}, 
uniformly in $n$, 
and that $\E_n$ converges weakly to an ensemble $\E$ as $n\to\infty$. 

Then $\E$ satisfies Assumption~\ref{ass}, 
and the associated models $(\Pi^{(n)},\Gamma^{*(n)})$ converge componentwise in law with respect to the distance induced by 
\begin{align}
\vertiii{\Pi}_\beta 
&= \sup_{x\in\R^{1+d}} \sup_{x'\neq y'\in B_1(x)}
|x'-y'|^{-|\beta|+\kappa} (1+|x|)^{-2\kappa} \,| \Pi_{x'\beta}(y') | \, , 
\label{normModelPi} \\
\vertiii{\G}_\beta 
&= \sup_\gamma \sup_{x\in\R^{1+d}} \sup_{x'\neq y'\in B_1(x)} 
|x'-y'|^{-|\beta|+|\gamma|+\kappa} (1+|x|)^{-2\kappa} \,| (\G_{x'y'})_\beta^\gamma | \, , \label{normModelGamma}
\end{align}
to the model $(\Pi,\G)$ associated to $\E$, for any $\kappa>0$. 

Furthermore, the convergence holds in probability (resp.~in $L^p$ for all $p<\infty$), 
provided the underlying sequence of random variables $(\xi_n)_{n\in\N}$
converges in probability (resp.~in $L^p$ for all $p<\infty$) 
to a random variable $\xi$, 
and the corresponding laws satisfy Assumption~\ref{ass} uniformly in $n$. 
\end{theorem}

As a simple application of the uniqueness of models, 
we obtain the following covariance properties of the model. 

\begin{corollary}\label{lem:covariant}
Assume that $(\Pi,\G)$ is a model for \eqref{spde}. 
If the ensemble $\E$ is invariant under the transformation $T$, 
where $T$ denotes either the space time shift $x\mapsto x+z$ for some $z\in\R^{1+d}$, 
or the spatial orthogonal transformation $x\mapsto(x_0,\bar{O}(x_1,\dots,x_d))$ 
for some orthogonal $\bar{O}\in\R^{d\times d}$, 
then almost surely  
\begin{equation}
\Pi_{Tx}[\xi](Ty) = \Pi_x[\xi(T\cdot)](y) \, , \quad 
\G_{Tx\,Ty}[\xi] = \G_{xy}[\xi(T\cdot)] \, . 
\end{equation}
Furthermore, if the ensemble $\E$ is invariant under 
$\xi\mapsto s^{2-\alpha}\xi(S\cdot)$, 
where $S$ denotes the anisotropic rescaling 
$x\mapsto(s^2x_0,sx_1,\dots,sx_d)$ for $s>0$, 
then almost surely
\begin{align}
s^{-|\beta|}\Pi_{Sx\,\beta}[\xi](Sy) 
&= \Pi_{x\beta}[s^{2-\alpha}\xi(S\cdot)](y) \, , \label{scalecovariantPi} \\
s^{-|\beta|+|\gamma|} \big(\G_{Sx\,Sy}[\xi]\big)_\beta^\gamma 
&= \big(\G_{xy}[s^{2-\alpha}\xi(S\cdot)]\big)_\beta^\gamma \, . \label{scalecovariantGamma}
\end{align}
\end{corollary}

\noindent
Let us emphasize that Corollary~\ref{lem:covariant} makes a couple of arguments 
from the heuristic discussion in Section~\ref{sec:model} rigorous, 
in particular \eqref{scaling} and \eqref{translation}. 
Moreover, it allows to appeal to invariances of the noise that are not preserved under approximation, 
where we think in particular of scaling invariance. 
We also want to point out that the listed transformations should rather be seen as exemplary, 
and many other transformations lead to analogous covariant transformations of the model. 

\begin{proof}[Proof of Corollary~\ref{lem:covariant}]
We will only give the argument for the scaling covariance. 
The statements for shift- and reflection covariance and for covariance under orthogonal transformations follow similarly. 
Let $(\Pi,\G)$ be the model associated to $\E$, and define
$\widetilde\Pi_{x\beta}(y):=s^{-|\beta|}\Pi_{Sx\,\beta}(Sy)$ 
and $(\tG_{xy})_\beta^\gamma = s^{-|\beta|+|\gamma|}(\G_{Sx\,Sy})_\beta^\gamma$. 
By the uniqueness part of Theorem~\ref{thm}, 
it is enough to show that $(\widetilde\Pi,\tG)$ satisfies 
Definition~\ref{def} with $\widetilde\xi$ given by $s^{2-\alpha}\xi(S\cdot)$, where we note that the corresponding $\widetilde\E$ coincides with $\E$ by assumption. 
The first item \eqref{eh11} follows from 
\begin{equation}
(\partial_0-\Delta)\widetilde\Pi_{x\,\beta=0}
= s^{2-\alpha} (\partial_0-\Delta)\Pi_{Sx\,\beta=0}(S\cdot)
= s^{2-\alpha} \xi(S\cdot) - s^{2-\alpha} \E\xi(S\cdot) \, .
\end{equation}
For purely polynomial multi-indices we observe 
$\widetilde\Pi_{xe_\n}(y) 
= s^{-|\n|} (Sy-Sx)^\n = (y-x)^\n$, 
and the shift invariance \eqref{expectationshiftinvariant} of the expectation of $\widetilde\Pi$ follows immediately from the one of $\Pi$. 
For the estimate \eqref{estPi} we note that 
\begin{equation}
\E^\frac{1}{p}|\widetilde\Pi_{x\beta}(y)|^p
= s^{-|\beta|} \E^\frac{1}{p}|\Pi_{Sx\,\beta}(Sy)|^p
\lesssim s^{-|\beta|} |Sx-Sy|^{|\beta|}
= |x-y|^{|\beta|} \, . 
\end{equation}
Clearly, $\tG_{xy}$ is still an algebra endomorphism, 
and \eqref{Gamma_zn} and \eqref{Gamma_zk} follow from the corresponding properties of $\G_{xy}$. 
The transformation \eqref{recenter} of $\widetilde\Pi$ follows by 
\begin{equation}
(\tG_{xy}\widetilde\Pi_y)_\beta 
= \sum_\gamma s^{-|\beta|+|\gamma|} (\G_{Sx\,Sy})_\beta^\gamma \, 
s^{-|\gamma|} \Pi_{Sy\,\gamma} 
= s^{-|\beta|} (\G_{Sx\,Sy}\Pi_{Sy})_\beta
\end{equation}
from the corresponding transformation of $\Pi$. 
Similarly, one can see the transformation \eqref{recenter-} of $\widetilde\Pi^-$ and the stochastic estimate \eqref{eh14}. 
Malliavin differentiability of $\widetilde\Pi_{x\beta}(y)$ is inherited from $\Pi$, 
and we define $\d\tG_{xy}:=s^{-|\beta|+|\gamma|}\d\G_{Sx\,Sy}$. 
With this, we obtain 
\begin{align}
&|y-z|^{-1} \E\big| \big( \delta\widetilde\Pi_x 
- \delta\widetilde\Pi_x(y) - \d\tG_{xy}Q\widetilde\Pi_y \big)_\beta (z) \big| \\
&= s^{1-|\beta|} |Sy-Sz|^{-1}  \E\big| \big( \delta\Pi_{Sx} 
- \delta\Pi_{Sx}(Sy) - \d\G_{Sx\,Sy}Q\Pi_{Sy} \big)_\beta (Sz) \big| \, ,
\end{align}
which for fixed $s>0$ converges to $0$ as $z\to y$. 
For the last item \eqref{ho28}, it is enough to note that 
for any distribution $f$ we have by duality and the scaling property \eqref{scaling_psi} that $(f(S\cdot))_t(y) = f_{st}(Sy)$. 
\end{proof}

\begin{remark}[Other equations]\label{rem:other_equations}
We note that in Definition~\ref{def}, 
only the recentering \eqref{recenter-} of $\Pi^-$, the order of vanishing in \eqref{ho29}, and \eqref{ho28} are equation specific. 
Working for instance with a multiplicative stochastic heat equation,
\begin{equation}
(\partial_0-\Delta) u = a(u)\xi \, ,
\end{equation}
we derive from the same principles as in Section~\ref{sec:model} a counterterm of the form $c[a(\cdot+u)]$, 
and the corresponding $\Pi_x^-$ is given by 
\begin{equation}
\Pi_x^-=\sum_{k\geq0}\z_k\Pi_x^k\xi-\sum_{l\geq0}\tfrac{1}{l!}\Pi_x^l (D^{(\0)})^l c \, .
\end{equation}
Accordingly, \eqref{recenter-} and \eqref{ho28} have to be replaced by the simpler 
\begin{equation}
\Pi_{x}^- = \G_{xy}\Pi_y^- 
\end{equation}
and 
\begin{equation}
\lim_{t\to0} \E\big|\big(\delta\Pi^-_x-\d\G_{xy} Q \Pi^-_{y} 
-\sum_{k\ge 0}\mathsf{z}_k\Pi_x^k(y)\delta\xi \big)_{\beta t}(y)\big| = 0 \, ,
\end{equation}
and \eqref{ho29} stays the same. 
For a version of the stochastic thin-film equation, 
\begin{equation}
(\partial_0+\Delta^2) u = \nabla\cdot(a(u)\nabla\Delta u + b(u)\xi) \, ,
\end{equation}
where the noise $\xi$ takes now values in $\R^d$, 
is stationary, invariant under $\xi \mapsto \bar{O}^T\xi(O\cdot)$ with $Ox=(x_0,\bar{O}(x_1,\dots,x_d))$ for orthogonal $\bar{O}\in\R^{d\times d}$, and $a,b$ are scalar-valued nonlinearities, 
we derive from similar principles a counterterm of the form $\nabla\cdot(c[a(\cdot+u), b(\cdot+u)] \nabla u)$ for a scalar valued functional $c$, 
and $\Pi^-_x$ takes the form of 
\begin{equation}
\Pi^-_x 
= \nabla\cdot \Big( \sum_{k\geq0} \z_k^{(a)} \Pi_x^k \nabla\Delta\Pi_x 
+ \sum_{\ell\geq0} \z_\ell^{(b)} \Pi_x^\ell \xi 
- \sum_{l\geq0} \tfrac{1}{l!} \Pi_x^l (D^{(\0)})^l c \nabla\Pi_x \Big) \, . 
\end{equation}
Here, $\z_k^{(a)}$ and $\z_\ell^{(b)}$ denote the coordinate functionals 
of the nonlinearities $a$ and $b$, respectively, 
and $D^{(\0)}$ is now given by 
$\sum_{k\geq0}(k+1)\z_{k+1}^{(a)}\partial_{\z_k^{(a)}} 
+ \sum_{\ell\geq0}(\ell+1)\z_{\ell+1}^{(b)}\partial_{\z_\ell^{(b)}}$. 
In this case, \eqref{recenter-} and \eqref{ho28} are in line with \cite{GT23} given by  
\begin{equation}
\Pi^-_x = \G_{xy} \Pi^-_y 
+ P \sum_{k\geq0} \z_k^{(a)} \big(\G_{xy}(\id-P)\Pi_y + \Pi_x(y)\big)^k \, 
\G_{xy}(\id-P)\nabla\Delta\Pi_y \, ,
\end{equation}
and 
\begin{align}
\lim_{t\to0} \E\big|\big(\delta\Pi^-_x-\d\G_{xy} Q \Pi^-_{y} 
&-\sum_{k\geq 0}\z_k^{(a)}\Pi_x^k(y)
\nabla\Delta (\delta\Pi_x-\d\G_{xy} Q \Pi_y) \\
&-\sum_{\ell\geq0}\z_\ell^{(b)}\Pi_x^\ell(y)\delta\xi\big)_{\beta t}(y) \big| = 0\, ,
\end{align}
and \eqref{ho29} is replaced by the higher order vanishing 
\begin{equation}
\lim_{z\to y} |y-z|^{-2} \, \mathbb{E} \big|\big(\delta\Pi_x-\delta\Pi_x(y)
-\d\G_{xy} Q \Pi_y\big)_\beta(z) \big| = 0 \, .
\end{equation}
\end{remark}

The following lemma collects several consequences of Definition~\ref{def} that will be useful later on. 

\begin{lemma}\label{lem:consequence}
Assume that $(\Pi,\G)$ is a model for \eqref{spde}. 
Then 
\begin{itemize}
\item[\rm i)] the expectation of $\Pi^-$ is shift invariant in the sense of 
\begin{equation}\label{expectationshiftinvariant-}
\E\Pi^-_{x\beta t}(y)=\E\Pi^-_{x+z \, \beta t}(y+z)
\quad\text{for all }x,y,z\in\R^{1+d}, \ t>0 \, , 
\end{equation}
and $\Pi^-$ satisfies for all multi-indices $\beta$ and $p<\infty$ 
\begin{equation}\label{Pi-} 
\mathbb{E}^\frac{1}{p}|\Pi^-_{x\beta t}(y)|^p 
\lesssim (\sqrt[4]{t})^{\alpha-2}(\sqrt[4]{t}+|x-y|)^{|\beta|-\alpha} \, . 
\end{equation}
\item[\rm ii)] $\G$ is triangular with respect to the homogeneity in the sense of \eqref{triGamma_hom}. 
Furthermore, $\G$ is transitive 
\begin{equation}\label{transitive}
\G_{xy}\G_{yz}=\G_{xz} \, , 
\end{equation}
and in addition to the mapping property \eqref{Gamma_preserves}, 
$\G_{xy}$ respects $\Tbar$ in the following sense, 
\begin{equation}\label{Gamma_preserves_Tbar}
(\G_{xy})_{e_\n}^\gamma \neq 0
\quad\implies\quad 
\gamma = e_\m \text{ for some }\m\neq\0 \, . 
\end{equation}
\item[\rm iii)] $\Pi$ and $\G$ are (annealed) H\"older continuous, 
more precisely, for every $\beta$ there exists $\epsilon>0$ such that 
for all $x',x'',y',y''\in\R^{1+d}$, $p<\infty$ and $\gamma$
\begin{align}
\E^\frac{1}{p} |\Pi_{x'\beta}(y')-\Pi_{x''\beta}(y'')|^p 
&\lesssim |x'-x''|^\epsilon (|x'-x''|+|x'-y'|)^{|\beta|-\epsilon}\\
&\quad+ |y'-y''|^\alpha (|y'-y''|+|x''-y''|)^{|\beta|-\alpha} \, , \\
\E^\frac{1}{p} |(\G_{x'y'}-\G_{x''y''})_\beta^\gamma|^p
&\lesssim |x'-x''|^\epsilon (|x'-x''|+|x'-y'|)^{|\beta|-|\gamma|-\epsilon}\\
&\quad+ |y'-y''|^\epsilon (|y'-y''|+|x''-y''|)^{|\beta|-|\gamma|-\epsilon} \, ,
\end{align}
with the understanding that all exponents are non-negative.
\end{itemize}
\end{lemma}

\begin{proof}
The stationarity \eqref{expectationshiftinvariant-} is an immediate consequence of 
$\Pi^-_{x\beta} = (\partial_0-\Delta)(P\Pi_x)_\beta$ 
and the corresponding shift invariance \eqref{expectationshiftinvariant} of $\E\Pi_{x\beta}(y)$. 
The estimate on $\Pi^-$ may be seen as follows. 
Since $(\partial_0-\Delta)\psi_t$ integrates to $0$ on $\R^{1+d}$, 
we obtain from the recentering \eqref{recenter} that 
\begin{align}
\Pi^-_{x\beta t}(y) 
&= \int_{\R^{1+d}} dz\, (\partial_0-\Delta)\psi_t(y-z) (P\Pi_x)_\beta(z) \\
&= \int_{\R^{1+d}} dz\, (\partial_0-\Delta)\psi_t(y-z) (P\G_{xy}\Pi_y)_\beta(z) \, .
\end{align}
Applying $\E^\frac{1}{p}|\cdot|^p$, the triangle inequality, 
\eqref{estPi} and \eqref{eh14}, this yields
\begin{equation}
\E^\frac{1}{p}| \Pi^-_{x\beta t}(y) |^p
\lesssim \int_{\R^{1+d}}dz\, |(\partial_0-\Delta)\psi_t(y-z)| 
\sum_{\gamma} |x-y|^{|\beta|-|\gamma|} |y-z|^{|\gamma|} \, , 
\end{equation}
where the sum is over finitely many $\gamma$ satisfying $|\gamma|\leq|\beta|$. 
By the moment bound \eqref{momentbound}, this is estimated by 
\begin{equation}
\sum_{\gamma} (\sqrt[4]{t})^{|\gamma|-2} |x-y|^{|\beta|-|\gamma|} \, , 
\end{equation}
which by $|\gamma|\geq\alpha$ is further bounded by the right-hand side of \eqref{Pi-}. 

We turn to ii). 
For a proof of \eqref{triGamma_hom} see \cite[Lemma~3]{OST23}. 
The transitivity is proven in \cite[Proposition~5.4]{LOTT21}, 
at least on $\T$; it immediately generalizes to $\R[[\z_k,\z_\n]]$ by multiplicativity. 
To see \eqref{Gamma_preserves_Tbar}, we first note that as a consequence of 
\eqref{Gamma_zn}, \eqref{Gamma_zk} and \eqref{rest_pin}, 
it is satisfied for multi-indices $\gamma$ of length one, 
i.e.~$\gamma=e_k,e_\m$. 
The general case follows by multiplicativity from 
$(\G_{xy})^{e_k}_0 = 0 = (\G_{xy})^{e_\m}_0$, which is again a 
consequence of \eqref{Gamma_zn}, \eqref{Gamma_zk} and \eqref{rest_pin}. 

We turn to the H\"older continuity iii). 
Using the recentering \eqref{recenter} we obtain
\begin{equation}\label{eh01}
\Pi_{x'}(y')-\Pi_{x''}(y'') 
= (\G_{x' x''}-\id)\Pi_{x''}(y')+\Pi_{x'}(x'')+\G_{x''y''}\Pi_{y''}(y') \, , 
\end{equation} 
which by H\"older's inequality together with the estimates 
\eqref{estPi} of $\Pi$ and \eqref{eh14} of $\G$ and the triangularity \eqref{triGamma_hom} yields 
\begin{align}
\E^\frac{1}{p} |\Pi_{x'\beta}(y')-\Pi_{x''\beta}(y'')|^p 
&\lesssim \sum_{|\gamma|<|\beta|} 
|x'-x''|^{|\beta|-|\gamma|}|x''-y'|^{|\gamma|}
+ |x'-x''|^{|\beta|} \\
&\quad+ \sum_{|\gamma|\leq|\beta|} |x''-y''|^{|\beta|-|\gamma|}|y'-y''|^{|\gamma|} \, , 
\end{align}
where we note that the sums over $\gamma$ are finite. 
In the first sum we use this finiteness and $|\gamma|<|\beta|$, 
to choose $\epsilon>0$ small enough such that $|\beta|-|\gamma|\geq\epsilon$. 
In the second sum we use that the homogeneity is bounded below by $\alpha$, which implies the desired continuity for $\Pi$. 
The argument for $\G$ is similar, based on 
\begin{equation}\label{eh02}
\G_{x'y'}-\G_{x''y''} 
= (\G_{x'x''}-\id)\G_{x''y'} - \G_{x''y'}(\G_{y'y''}-\id) \, , 
\end{equation}
which is a consequence of the transitivity \eqref{transitive}. 
\end{proof}

\begin{remark}[Comparison to \cite{Hai14}]\label{rem:comparison}
Let us briefly compare our definition of model to Hairer's \cite{Hai14}. 
Here, the model is a random object, which is referred to as random model in Hairer's language. 
The ``anchoring'' \eqref{eh11} and \eqref{polypart} is identical in both settings, 
as well as the estimate \eqref{estPi} of $\Pi$ apart from the fact that here it is formulated in 
a probabilistic as opposed to a deterministic way.
The shift invariance \eqref{expectationshiftinvariant} of $\E\Pi$ 
is assumed to obtain the corresponding invariance \eqref{expectationshiftinvariant-} of $\E\Pi^-$. 
By the BPHZ-choice of renormalization, see \eqref{bphz}, 
this expresses the fact that $c$ has no space-time dependence. 
The recentering \eqref{recenter} and \eqref{recenter-} and the estimate \eqref{eh14} of $\G$ 
coincide (again up to probabilistic/deterministic formulations). 
The properties \eqref{Gamma_zn}, \eqref{Gamma_zk} and multiplicativity of $\G$ are mainly made 
such that Lemma~\ref{lem:consequence}~ii) holds, which is what we will use in the sequel, 
and all of which is also present in \cite{Hai14} (note that \eqref{Gamma_preserves_Tbar} 
corresponds to $\Gamma\mathsf{\bar{T}}\subset\mathsf{\bar{T}}$). 
Furthermore, \eqref{Gamma_zn}, \eqref{Gamma_zk} and multiplicativity imply 
that $(\G_{xy})_\beta^\gamma$ for $\gamma$ not purely polynomial is determined by 
$\Pi_\gamma$ for $\gamma$ ``smaller'' than $\beta$, 
see Step~1 in the proof of Proposition~\ref{unique} for a precise statement. 
Similar results are also available in \cite{Hai14}. 
The main difference is the last item in Definition~\ref{def}, 
which has no analogue in \cite{Hai14}. 
It is however reasonable that it can be related to the pointed modelled 
distribution $H_{\tau}^x$ introduced in \cite{HS23}. 
For a more concise comparison of the two settings we refer to \cite[Section~5.3]{LOT23} for algebraic aspects 
and to \cite[Section 2.6]{LOTT21} for analytic and probabilistic aspects. 
\end{remark}

\begin{remark}[Comparison to \cite{OSSW21}]\label{rem:ossw}
Note that the model in \cite{OSSW21} rather corresponds to the equation 
\begin{equation}
(\partial_0-a_0\Delta)u=(a(u)-a_0)\Delta u+\xi
\end{equation}
with $a_0=a(0)$, 
which is equivalent to \eqref{spde} modulo the substitution ${a\mapsto a+1}$.
Adopting this perspective affects the model in two ways:
On the one hand, 
the Ansatz \eqref{ansatz} features only multi-indices $\beta$ with $\beta(k=0)=0$, 
hence $\beta$ is a multi-index over the smaller set $k\geq1$ and $\n\neq\0$.
On the other hand, 
the model $(\Pi,\Gamma^*)$ inherits a dependence on $a_0$ 
through the differential operator $(\partial_0-a_0\Delta)$.
As shown in \cite[Proposition~2.7]{LOTT21}, 
the stochastic estimates \eqref{estPi} of $\Pi$ and \eqref{eh14} of $\Gamma^*$ hold then locally uniformly for $a_0>0$.
We could therefore work in this setting as well and modify Definition~\ref{def} accordingly, 
with all results holding verbatim 
(the convergence in Theorem~\ref{thm2} holds then locally uniformly for $a_0>0$). 
Note that such a model satisfies then also \cite[Assumptions~1~and~2]{OSSW21}, 
up to the qualitative smoothness of $\Pi_x$ which is necessary to write down the renormalized equation.
The main result of \cite{OSSW21} is that under this assumption, 
any smooth solution of the renormalized equation is (uniformly in the qualitative smoothness) approximated 
by a suitable truncation of the series in \eqref{ansatz}, 
validating our ansatz.
Together with a robust reformulation of \eqref{spde}
in the space of modelled distributions, 
which makes sense in the rough setting as well, 
one would expect to be able to combine 
this a-priori estimate with a continuity method 
to construct a solution of \eqref{spde};
see the upcoming work \cite{BOS24+} for an implementation of this approach 
in the setting of a space-time periodic semilinear equation with additive noise. 
\end{remark}


\section{Uniqueness of the model}\label{sec:unique}

The aim of this section is to prove the following proposition, 
which establishes the uniqueness part of Theorem~\ref{thm}. 

\begin{proposition}[Uniqueness]\label{unique}
Assume that $\E$ satisfies Assumption~\ref{ass}. 
Then a model $(\Pi,\G)$ for \eqref{spde} is unique. 
\end{proposition}

\noindent
The proof proceeds inductive with respect to an ordering $\prec$ 
on multi-indices which we shall introduce now. 
As is clear from the discussion in Section~\ref{sec:stable}, 
this ordering can not be given by the homogeneity on multi-indices. 
Instead, we consider the following ordinal 
\begin{equation}
|\beta|_\prec 
:= [\beta] + \lambda_1 \sum_{\n\neq\0}|\n|\beta(\n) 
+ \lambda_2 \beta(k=0) \, , 
\end{equation}
where we fix $1>\lambda_1>\lambda_2>0$, and $[\beta]$ is defined in \eqref{noisehomogeneity}. 
We will then write 
\begin{align}
\beta'\prec\beta &\quad\iff\quad |\beta'|_\prec<|\beta|_\prec \, , \\
\beta'\preceq\beta &\quad\iff\quad (\beta'\prec\beta \text{ or } \beta'=\beta) \, .
\end{align}
The following lemma collects properties of $\prec$ that we will use in the sequel. 

\begin{lemma}\label{lem:ordering}
The ordering $\prec$ is coercive on populated multi-indices, 
meaning that $\{\gamma\,|\,\gamma\text{ populated and }\gamma\prec\beta\}$ is finite for every $\beta$. 
Furthermore, for all $\beta,\gamma$ and populated $\beta_1,\dots,\beta_{k+1}$ we have 
\begin{align}
e_k+\beta_1+\dots+\beta_{k+1}=\beta 
&\quad\implies\quad 
\beta_1,\dots,\beta_{k+1}\prec\beta \, , \label{triProduct_prec} \\
(\G_{xy}-\id)_\beta^\gamma \neq 0
&\quad\implies\quad {\gamma}\prec{\beta} \, , \label{triGamma_prec} \\
(\d\G_{xy})_\beta^\gamma \neq 0 
&\quad\implies\quad {\gamma}\prec{\beta} \, . \label{tridGamma_prec} 
\end{align}
\end{lemma}

\begin{proof}
The coercitivity is an immediate consequences of the definition of $\prec$. 
A proof of \eqref{triProduct_prec} can be found in \cite[(8.6)]{LOTT21}. 
For \eqref{triGamma_prec} we note that $\G_{xy}$ defined through 
\eqref{Gamma_zn} and \eqref{Gamma_zk} coincides with $\G_{xy}$ 
defined through the exponential formula \cite[(2.44)]{LOTT21}. 
Hence \eqref{triGamma_prec} coincides with \cite[(8.9)]{LOTT21} for populated $\beta,\gamma$, 
and follows by multiplicativity for all $\beta,\gamma$. 
The proof of \eqref{tridGamma_prec} follows as in \cite[(8.11)]{LOTT21} from its definition \eqref{dGamma} and the corresponding property  \eqref{triGamma_prec} of $\G_{xy}$. 
\end{proof}

\noindent
Another crucial ingredient for uniqueness is the following Liouville type result. 

\begin{lemma}[Liouville]\label{liouville}
Let $\eta>0$, $\eta\not\in\N$, 
and assume for a random field $f$ that $\sup_x |x|^{-\eta}\E|f(x)|<\infty$,  
and that $(\partial_0-\Delta)f$ is a random polynomial of degree $\leq\eta+2$.
Then $f=0$.
\end{lemma}

\begin{proof}
The assumption implies 
$(\partial_0-\Delta)\partial^\n f = 0$ for $|\n|>\eta+2$, 
and by the defining property \eqref{psi} of $\psi_t$ we obtain 
$\partial_t\partial^\n f_t 
= (\partial_0-\Delta)(\partial_0+\Delta)\partial^\n f_t=0$
for $|\n|>\eta$. 
In particular, $\partial^\n f_t$ is independent of $t$ for $|\n|>\eta$. 
Furthermore, by assumption and the moment bound \eqref{momentbound} 
\begin{equation}
\E| \partial^\n f_t(y) |
\leq \int_{\R^{1+d}}dx\, |\partial^\n\psi_t(y-x)| \,\E| f(x) |
\lesssim (\sqrt[4]{t})^{-|\n|}(\sqrt[4]{t}+|x|)^\eta \, .
\end{equation}
From the limit $t\to\infty$, we deduce that almost surely 
$\partial^\n f_t(y)=0$ for all $|\n|>\eta$, $y\in\R^{1+d}$ and $t>0$. 
Hence $f$ is a polynomial of degree $\leq\eta$, 
and since $\eta\not\in\N$ this strengthens to $<\eta$. 
Since $f$ vanishes by assumption at the origin to order $\eta$, 
we learn $f=0$. 
\end{proof}

\begin{proof}[Proof of Proposition~\ref{unique}]
Assume that we are given two families of random fields 
$(\Pi,\G)$ and $(\widetilde\Pi,\tG)$, 
both satisfying Definition~\ref{def}. 
We will prove that their respective $\beta$-components coincide almost surely. 
First, note that it is enough to consider populated multi-indices. 
Indeed, for $\beta$ not populated, $\Pi_{x\beta}=0=\widetilde\Pi_{x\beta}$. 
For $(\G_{xy})_\beta^\gamma$ we note that populated multi-indices fully determine $\pi^{(\n)}_{xy}\in\T$ by $(\G_{xy}-\id)_\beta^{e_\n} = \pi^{(\n)}_{xy\beta}$, see \eqref{Gamma_zn}, 
which then determines $(\G_{xy})_\beta^\gamma$ for arbitrary multi-indices by multiplicativity. 

For the remaining populated multi-indices, 
we start with purely polynomial multi-indices $\beta=e_\n$.
From \eqref{polypart} we learn that $\Pi_{xe_\n}(y) = (y-x)^\n = \widetilde\Pi_{xe_\n}(y)$. 
Turning to $(\G_{xy})_{e_\n}^\gamma$, we observe that 
by \eqref{Gamma_preserves_Tbar} the only non-vanishing components 
are given by $(\G_{xy})_{e_\n}^{e_\m}$ for $\m\neq\0$, 
and \eqref{triGamma_hom} restricts to $|\m|\leq|\n|$.
From the $(\beta=e_\n)$-component of \eqref{recenter} we learn that 
\begin{align}
(\cdot-x)^\n 
= \Pi_{xe_\n} 
&= \sum_{0<|\m|\leq|\n|} (\G_{xy})_{e_\n}^{e_\m} \Pi_{ye_\m} + (y-x)^\n \\ 
&= \sum_{0<|\m|\leq|\n|} (\G_{xy})_{e_\n}^{e_\m} (\cdot-y)^{\m} + (y-x)^\n\, ,
\end{align}
which by the binomial formula shows that 
\begin{equation}\label{Gamma_poly}
(\G_{xy})_{e_\n}^{e_\m}=\tbinom{\n}{\m}(y-x)^{\n-\m} \, .
\end{equation} 
The same argument can be repeated for $\tG$, thus 
$(\G_{xy})_{e_\n}^\gamma = (\tG_{xy})_{e_\n}^\gamma$ for all $\gamma$. 

For the remaining populated but not purely polynomial multi-indices we proceed by induction with respect to $\prec$. 
It is convenient to explicitly include uniqueness of $\d\G_{xy}$ into the induction, 
we therefore assume that for $\delta\xi\in\S(\R^{1+d})$ and $x\neq y$ 
we are also given $\d\G_{xy}$ and $\d\tG_{xy}$ from Definition~\ref{def} 
corresponding to $(\Pi_x,\G_{xy})$ and $(\widetilde\Pi_x,\tG_{xy})$, respectively.
For the base case $\beta=0$, 
we recall from \eqref{eh11} 
that $(\partial_0-\Delta)(\Pi_{x0}-\widetilde\Pi_{x0})=0$. 
Using the estimate \eqref{estPi} and $|\beta=0|=\alpha\not\in\Q$, 
Liouville's principle Lemma~\ref{liouville} yields $\Pi_{x0}-\widetilde\Pi_{x0}=0$. 
For $(\G_{xy})_0^\gamma$ we note that by the triangularity  \eqref{triGamma_hom} 
the only non-vanishing component is $(\G_{xy})_0^0=1=(\tG_{xy})_0^0$.
By the triangularity \eqref{tridGamma_prec} of $\d\G_{xy}$, 
the only non-vanishing components of $(\d\G_{xy})_0^\gamma$ 
are given by $(\d\G_{xy})_0^{e_\n}$ with $|\n|=1$.
Using \eqref{polypart}, this yields $(\d\G_{xy}Q\Pi_y)_0 = \sum_{|\m|=1} (\cdot-y)^\m (\d\G_{xy})_0^{e_\m}$, 
which evaluated at $y+\lambda\n$ for $\lambda>0$ and $|\n|=1$ equals $\lambda(\d\G_{xy})_0^{e_\n}$.
Since the same holds for $\d\tG_{xy}$, by using the already established 
$\delta\Pi_{x0} = \delta\widetilde\Pi_{x0}$ we obtain 
\begin{align}
\lambda (\d\G_{xy}-\d\tG_{xy})_0^{e_\n}
&= (\delta\widetilde\Pi_x-\delta\widetilde\Pi_x(y)-\d\tG_{xy} Q \widetilde\Pi_y)_0(y+\lambda\n) \nonumber \\ 
&\, - (\delta\Pi_x-\delta\Pi_x(y)-\d\G_{xy} Q \Pi_y)_0(y+\lambda\n) \, .
\end{align}
Dividing by $\lambda$ and taking $\lambda\to 0$, we obtain from \eqref{ho29} 
that $(\d\G_{xy}-\d\tG_{xy})_0^{e_\n}=0$ almost surely, 
finishing the argument for the base case.

In the induction step we assume uniqueness of 
$\Pi_{x\beta'}$, $(\G_{xy})_{\beta'}^\gamma$ and $(\d\G_{xy})_{\beta'}^\gamma$ for $\beta'\prec\beta$ and all populated $\gamma$,
and show uniqueness of the corresponding $\beta$-components.
We first note that $\Pi_{x\beta'}$ determines unique 
$\Pi^-_{x\beta'}=(\partial_0-\Delta)(P\Pi_{x})_{\beta'}$, 
$\delta\Pi_{x\beta'}$ and $\delta\Pi^-_{x\beta' t}$.

{\bf Step 1.} 
$(\G_{xy})_\beta^\gamma$ is unique for $\gamma$ not purely polynomial. 
This is a consequence of the induction hypothesis together with 
the fact that for $\gamma$ not purely polynomial 
\begin{equation}
(\G_{xy})_\beta^\gamma 
\quad\textnormal{is determined by }\Pi_{x\beta'}\textnormal{ and } (\G_{xy})_{\beta'}
\textnormal{ for }\beta'\prec\beta \, , 
\end{equation}
which we shall establish now. 
Since $\gamma$ is assumed to be populated and not purely polynomial, 
we can write $\z^\gamma=\z_{k_1}\cdots \z_{k_i} \z_{\n_1}\cdots \z_{\n_j}$ 
for some $k_1,\dots,k_i\geq0$, $i\geq1$, 
and $\n_1,\dots,\n_j\neq\0$, $j\geq0$.
In this case, $[\gamma]\geq0$ translates into $k_1+\dots+k_i\geq j$.
We learn from \eqref{Gamma_zk} and the multiplicativity of $\G_{xy}$  
that $(\G_{xy})_\beta^\gamma$ is a sum consisting of terms of the form
\begin{equation}
\Pi_{x\beta_1^1}(y) \cdots \Pi_{x\beta_{l_1}^1}(y) \cdots
\Pi_{x\beta_1^i}(y) \cdots \Pi_{x\beta_{l_i}^i}(y)
(\G_{xy})_{\bar\beta_1}^{e_{\n_1}} \cdots
(\G_{xy})_{\bar\beta_j}^{e_{\n_j}} \, ,
\end{equation}
where $l_1,\dots,l_i\geq0$ and 
$ e_{k_1+l_1} + \cdots + e_{k_i+l_i} + 
\beta_1^1 + \cdots + \beta_{l_1}^1 + \cdots +
\beta_1^i + \cdots + \beta_{l_i}^i + 
\bar\beta_1 + \cdots + \bar\beta_j = \beta $.
We first claim that $|e_{k+l}+\beta_1+\dots+\beta_l|_\prec \geq k$. 
Indeed, if $k+l=0$, we obtain $|e_0|_\prec = \lambda_2 > 0 = k$. 
If $k+l>0$ and $l=0$, we obtain $|e_k|_\prec = k$,
and if $l>0$ we obtain by additivity of $|\cdot|_\prec\geq \lambda_1-1$ that 
$|e_{k+l}+\beta_1+\dots+\beta_l|_\prec \geq k+l+l(\lambda_1-1) > k$.
We now show that $\beta_1^1\prec\beta$, which in particular means that $l_1\geq1$.
Then by additivity of $|\cdot|_\prec$ and the above, we obtain
$|\beta|_\prec\geq k_1+l_1 + |\beta_1^1|_\prec + (l_1-1)(\lambda_1-1) + k_2+\cdots+k_i+j(\lambda_1-1)$,
which by $k_1+\dots+k_i\geq j$ is larger than $|\beta_1^1|_\prec+1+(l_1-1)\lambda_1+j\lambda_1>|\beta_1^1|_\prec$.
We come to $\bar\beta_1\prec\beta$, which in particular means that $j\geq1$. 
Similarly as above, we obtain
$|\beta|_\prec \geq 
k_1 + \dots + k_i + |\bar\beta_1|_\prec+(j-1)(\lambda_1-1)
\geq |\bar\beta_1|_\prec +1+ (j-1)\lambda_1 
> |\bar\beta_1|_\prec$.
By symmetry, this concludes the proof.

{\bf Step 2.} 
$\Pi^-_{x\beta}$ is unique.
In this step we distinguish $|\beta|<2$ from $|\beta|>2$.

{\bf Case} $|\beta|>2${\bf.} 
We rewrite \eqref{recenter-} to 
\begin{equation}
\Pi^-_x-\Pi^-_y 
= (\G_{xy}-\id)\Pi^-_y 
+ P \sum_{k\geq0}\z_k \big(\G_{xy}(\id-P)\Pi_y+\Pi_x(y)\big)^k \, 
\G_{xy}(\id-P)\Delta\Pi_y \, , 
\end{equation}
and note that the $\beta$-component of the right-hand side is unique by Step~1, 
the triangular properties \eqref{triProduct_prec} and \eqref{triGamma_prec},
and the induction hypothesis.
We therefore obtain $(\Pi^-_{x}-\Pi^-_{y})_\beta = (\widetilde\Pi^-_{x}-\widetilde\Pi^-_{y})_\beta$,
and hence
$(\Pi^-_{x}-\widetilde\Pi^-_{x})_{\beta\,t}(y) = (\Pi^-_{y} - \widetilde\Pi^-_{y})_{\beta\,t}(y)$.
By the triangle inequality and \eqref{Pi-}, we obtain 
$\E | (\Pi^-_{x}-\widetilde\Pi^-_{x})_{\beta\, t}(y) | 
\lesssim (\sqrt[4]{t})^{|\beta|-2}$, 
which by $|\beta|>2$ yields almost surely $\Pi^-_{x\beta}=\widetilde\Pi^-_{x\beta}$.

{\bf Case} $|\beta|<2${\bf.}
The proof is more complex for $|\beta|<2$ 
and relies on Malliavin differentiability and the extra ingredient $\d\G$.

{\bf Step 2a.} 
$\E\Pi^-_{x\beta\, t}(y)$ is unique.
Since $|\beta|<2$, we obtain from the estimate \eqref{Pi-} on $\Pi^-_{x\beta}$ in particular 
$\lim_{t\to\infty} \E\Pi^-_{x\beta\, t}(y) = 0$.
We can therefore write 
\begin{equation}
\E \Pi^-_{x\beta\, t}(y) 
= - \int_t^\infty \partial_s \E\Pi^-_{x\beta\, s}(y) \, ds \, .
\end{equation}
Hence it remains to check that $\partial_s \E\Pi^-_{x\beta\, s}(y)$ is unique.
By use of the semigroup property \eqref{semigroup} 
and the defining property \eqref{psi} of $\psi_t$,
we obtain
\begin{equation}
\partial_s \E\Pi^-_{x\beta\, s}(y) 
= \int_{\R^{1+d}} dz\, (\partial_0^2-\Delta^2)\psi_{s-\tau}(y-z) \E\Pi^-_{x\beta\,\tau}(z) \, . 
\end{equation}
Again by $|\beta|<2$, \eqref{recenter-diff} yields $\Pi^-_{x\beta}=(\G_{xz}\Pi^-_z)_\beta$.
Furthermore, by Lemma~\ref{lem:consequence} we know that $\E\Pi^-_{z\beta\tau}(z)$ does not depend on $z$, 
and since $(\partial_0^2-\Delta^2)\psi_{s-\tau}$ integrates to $0$, we obtain 
\begin{equation}
\partial_s \E\Pi^-_{x\beta\, s}(y) 
= \int_{\R^{1+d}} dz\, (\partial_0^2-\Delta^2)\psi_{s-\tau}(y-z) 
\E \big((\G_{xz}-\id)\Pi^-_{z}\big)_{\beta\,\tau}(z) \, . 
\end{equation}
Since $\Pi^-_z\in\Ttilde$ has no purely polynomial components, 
the right-hand side is unique by Step~1 
and the strict triangularity \eqref{triGamma_prec} 
combined with uniqueness of $\Pi^-_{z\beta'}$ for $\beta'\prec\beta$.

{\bf Step 2b.} 
$(\d\G_{xy})_\beta^\gamma$ is unique for $\gamma$ not purely polynomial. 
By definition \eqref{dGamma} of $\d\G_{xy}$ we have
\begin{equation}
(\d\G_{xy})_\beta^\gamma = \sum_{|\n|\leq 1}\sum_{\beta'+\beta''=\beta} 
\d\pi^{(\n)}_{xy\beta'} (\G_{xy}D^{(\n)})_{\beta''}^\gamma\, ,
\end{equation}
where $\d\pi^{(\0)}_{xy\beta'}=\delta\Pi_{x\beta'}(y)$, 
$\d\pi^{(\n)}_{xy\beta'}=(\d\G_{xy})_{\beta'}^{e_\n}$ for $|\n|=1$ 
and $(\G_{xy}D^{(\n)})_{\beta''}^\gamma 
= \sum_{\gamma'} (\G_{xy})_{\beta''}^{\gamma'}(D^{(\n)})_{\gamma'}^\gamma$. 
Since $\gamma$ is not purely polynomial by assumption, 
and $\gamma'$ inherits this property from $D^{(\n)}$ 
as can be easily seen from its definition \eqref{D0} and \eqref{Dn}, 
the claim follows from Step~1 and the induction hypothesis, 
provided $\beta'\prec\beta$ and $\beta''\preceq\beta$. 
Since $\d\pi^{(\n)}_{xy}\in\Ttilde$, we can restrict to 
$\beta'$ that are not purely polynomial, 
which implies $0\preceq\beta'$ and hence $\beta''\preceq\beta$. 
Similarly, since $\gamma'$ is not purely polynomial, 
$\beta''$ inherits this property from $\G_{xy}$ by \eqref{Gamma_preserves}, 
and thus $\beta'\preceq\beta$. 
If $\beta'=\beta$, then $\beta''=0$, and by the triangularity 
\eqref{triGamma_hom} of $\G_{xy}$ also $\gamma'=0$. 
However, $(D^{(\n)})_0^\gamma$ is only non-vanishing for purely polynomial $\gamma$, as can be seen from its definition \eqref{D0} and \eqref{Dn}, 
which contradicts our assumption and thus $\beta'\prec\beta$. 

{\bf Step 2c.} 
$\Pi^-_{x\beta}$ is unique.
We look at \eqref{ho28} and learn from Step~2b 
and the triangular properties \eqref{triProduct_prec} and  \eqref{tridGamma_prec} 
that all terms besides $\delta\Pi^-_{x\beta t}$ are unique at this stage of the induction.
Hence for $x\neq y$
\begin{align}
&(\delta\Pi^-_x-\delta\widetilde\Pi^-_x)_{\beta\, t}(y) \\
&= \big(\delta\Pi^-_x\-\d\G_{xy} Q \Pi^-_{y} 
\-\sum_{k\ge 0}\mathsf{z}_k\Pi_x^k(y)
\Delta (\delta\Pi_x\-\d\G_{xy} Q \Pi_y)
\-\delta\xi\mathsf{1}\big)_{\beta t} (y) \\
&- \big(\delta\widetilde\Pi^-_x\-\d\tG_{xy} Q \widetilde\Pi^-_{y} 
\-\sum_{k\ge 0}\mathsf{z}_k\widetilde\Pi_x^k(y)
\Delta (\delta\widetilde\Pi_x\-\d\tG_{xy} Q \widetilde\Pi_y)
\-\delta\xi\mathsf{1}\big)_{\beta t} (y) \, . 
\end{align}
Applying $\E|\cdot|$ and the triangle inequality, we learn from \eqref{ho28} that $\E|(\delta\Pi^-_x-\delta\widetilde\Pi^-_x)_{\beta t}(y)|\to0$ as $t\to0$. 
Note also that the semigroup property \eqref{semigroup} yields 
\begin{align}
\E| \delta\Pi^-_{x\beta \,T}(y) - \delta\widetilde\Pi^-_{x\beta \, T}(y)| 
&= \lim_{t\to0} \E|\delta\Pi^-_{x\beta \,T+t}(y) 
- \delta\widetilde\Pi^-_{x\beta \,T+t}(y)| \\
&\leq \lim_{t\to0} \int_{\R^{1+d}}dz\, |\psi_T(y-z)| \, 
\E|\delta\Pi^-_{x\beta \,t}(z) - \delta\widetilde\Pi^-_{x\beta \,t}(z)| \, . 
\end{align}
By the polynomial growth \eqref{polygrowth}, 
which by \cite[Lemma~A.2]{LOTT21} is inherited by $\delta\Pi^-_{x\beta t}$,
and since $\psi_T$ is a Schwartz function, 
the dominated convergence theorem thus yields 
for all $T>0$, $x\neq y\in\R^{1+d}$ and $\delta\xi\in\S(\R^{1+d})$, almost surely
\begin{equation}
(\delta\Pi^-_x-\delta\widetilde\Pi^-_x)_{\beta\, T}(y) =0 \, .
\end{equation}
Since $\S(\R^{1+d})$ is separable, we obtain 
$\frac{\partial}{\partial\xi} (\Pi^-_x - \widetilde\Pi^-_x)_{\beta\,T}(y)=0$ on $\S$ almost surely, 
and by density of $\S(\R^{1+d})$ in $\dot{H}^{\alpha-2+D/2}$ we obtain 
$\frac{\partial}{\partial\xi} (\Pi^-_x - \widetilde\Pi^-_x)_{\beta\,T}(y)=0$. 
Together with Step~2a, an application of the spectral gap inequality \eqref{sg} yields $(\Pi^-_x-\widetilde\Pi^-_x)_{\beta\,T}(y)=0$ almost surely, 
and by continuity of $(\Pi^-_x-\widetilde\Pi^-_x)_{\beta\,T}$ it vanishes also at $y=x$.
Since $T>0$ and $y\in\R^{1+d}$ were arbitrary, $\Pi^-_{x\beta}=\widetilde\Pi^-_{x\beta}$.

{\bf Step 3.} 
$\Pi_{x\beta}$ is unique.
Note that since $\beta$ is populated and not purely polynomial,
$|\beta|\not\in\N$ as a consequence of $\alpha\not\in\Q$. 
Since $(\partial_0-\Delta)(\Pi_{x\beta}-\widetilde\Pi_{x\beta})=\Pi^-_{x\beta}-\widetilde\Pi^-_{x\beta}=0$ by Step~2, 
we obtain $\Pi_{x\beta}=\widetilde\Pi_{x\beta}$ as an immediate consequence 
of the estimate \eqref{estPi} by Liouville's principle Lemma~\ref{liouville}. 

{\bf Step 4.} 
$(\G_{xy})_\beta^\gamma$ is unique. 
In view of Step~1 it remains to prove uniqueness of $(\G_{xy})_\beta^\gamma$ for purely polynomial $\gamma$. 
We rewrite \eqref{recenter} to $\Pi_x-\Pi_x(y)-\G_{xy}P\Pi_y = \G_{xy}(\id-P)\Pi_y$ 
and note that the $\beta$-component of the left-hand side is at this stage of the induction 
unique by Step~1, the triangularity \eqref{triGamma_prec} and Step~3.
By \eqref{polypart}, the $\beta$-component of the right-hand side equals 
\begin{equation}
\sum_{\n\neq\0} (\G_{xy})_\beta^{e_\n} (\cdot - y)^\n \, ,
\end{equation}
which is a finite sum due to \eqref{triGamma_prec} 
and establishes uniqueness of $(\G_{xy})_\beta^{e_\n}$.

{\bf Step 5.} 
$(\d\G_{xy})_\beta^\gamma$ is unique. 
By Step~2b it remains to establish uniqueness of $(\d\G_{xy})_\beta^\gamma$ for purely polynomial $\gamma$, 
and from the definition \eqref{dGamma} of $\d\G_{xy}$ we see that this is only non vanishing for $\gamma=e_\n$ with $|\n|=1$. 
Together with \eqref{polypart}, the already established 
$\delta\Pi_{x\beta} = \delta\widetilde\Pi_{x\beta}$, 
Step~2b and the triangularity \eqref{tridGamma_prec} of $\d\G_{xy}$, 
this yields
\begin{align}
&\sum_{|\m|=1}(\cdot-y)^\m (\d\G_{xy}-\d\tG_{xy})_\beta^{e_\m} 
= ((\d\G_{xy}-\d\tG_{xy}) Q (\id-P)\Pi_y)_\beta \\
&= -(\delta\Pi_x-\delta\Pi_x(y)-\d\G_{xy} Q \Pi_y)_\beta 
+ (\delta\widetilde\Pi_x-\delta\widetilde\Pi_x(y)-\d\tG_{xy} Q \widetilde\Pi_y)_\beta \, .
\end{align}
Evaluating at $y+\lambda\n$ for $\lambda>0$ and $|\n|=1$ and dividing by $\lambda$, we obtain
from \eqref{ho29} almost surely $(\d\G_{xy}-\d\tG_{xy})_\beta^{e_\n}=0$, which finishes the proof. 
\end{proof}


\section{Existence of a model}\label{sec:exist}

To establish the existence part of Theorem~\ref{thm}, 
we consider $(\Pi^{(\tau)},\Gtau)$ associated to the mollified noise $\xi_\tau:=\xi*\psi_\tau$, 
which was constructed in \cite{LOTT21}. 
We start by checking that this construction indeed yields a model 
in the sense of Definition~\ref{def}, 
provided $\xi$ is replaced by the smooth $\xi_\tau$. 
Notice that the construction in \cite{LOTT21} is carried out for $d=1$ and $\alpha<1/2$, 
so that the regularity index $2-\alpha-D/2$ of the norm $\|\cdot\|_*$ is positive. 
However, by small adaptations of exponents and weights, 
see Appendix~\ref{weights}, 
the proof holds in arbitrary dimension and for $\alpha<1$. 
We refer to \cite{GT23}, 
where the same strategy is applied for $\alpha<1$ and in arbitrary dimension -- 
although for a slightly different equation. 

\begin{lemma}\label{smooth_model}
Assume that $\E$ satisfies Assumption~\ref{ass}. 
Then $(\Pi^{(\tau)},\Gtau)$ constructed in \cite{LOTT21} is a model for \eqref{spde} with $\xi$ replaced by $\xi_\tau$. 
\end{lemma}

\begin{proof}
For the proof we just collect all the necessary properties of 
$\Pi^{(\tau)}$ and $\Gtau$ from \cite{LOTT21}. 
Indeed, $\Pi^{(\tau)}_{x0}$ satisfies the linear stochastic heat equation \eqref{eh11} with $\xi$ replaced by $\xi_\tau$ as a consequence of \cite[(2.18), (2.35)]{LOTT21}. 
The purely polynomial part $\Pi_{x e_\n}$ is in agreement with \eqref{polypart} by \cite[(2.21)]{LOTT21}, 
and the shift invariance \eqref{expectationshiftinvariant} follows from \cite[(5.2)]{LOTT21}. 
The mapping properties \eqref{Gamma_zn} and \eqref{Gamma_zk} of $\Gtau$ 
are captured by \cite[(2.51)]{LOTT21} and \cite[(2.50)]{LOTT21}, respectively, 
where $\pi^{(\n)(\tau)}$ is restricted by \eqref{rest_pin} due to \cite[(2.46)]{LOTT21}. 
The recentering \eqref{recenter} of $\Pi^{(\tau)}$ and \eqref{recenter-} of $\Pi^{-(\tau)}$ 
is stated in \cite[(2.61), (2.64)]{LOTT21} and \cite[(2.63)]{LOTT21}, respectively. 
From \cite[Section 7]{LOTT21} we know 
that\footnote{note that compared to \cite{LOTT21}, 
$\H$ here is denoted by $\H^2$ there, 
and $C(\H)$ here corresponds to $C^0(\H)$ there with the additional constraint of polynomial growth, see \eqref{polygrowth}} 
$\Pi^{(\tau)}_{x\beta} \in C(\H)$, 
and the ansatz \eqref{dGamma} for $\d\Gtau$ is the same as in \cite[(2.64), (4.40)]{LOTT21}. 
The stochastic estimates \eqref{estPi} and \eqref{eh14} are established in \cite[(2.36), (2.37)]{LOTT21}, where we point out the uniformity in $\tau>0$, 
\begin{align}
\sup_{\tau>0} \, \sup_{x\neq y} |x-y|^{-|\beta|} \, 
\mathbb{E}^\frac{1}{p}|\Pi^{(\tau)}_{x\beta}(y)|^p 
&< \infty \, , \label{Pi} \\
\sup_{\tau>0} \, \sup_{x\neq y} |x-y|^{-(|\beta|-|\gamma|)} \, 
\mathbb{E}^\frac{1}{p}|(\Gtau_{xy})_{\beta}^{\gamma}|^p 
&<\infty \, . \label{Gamma}
\end{align}
It remains to check the properties \eqref{ho29} and \eqref{ho28}.
For this, we recall from \cite[(4.86)]{LOTT21} 
\begin{align}
&\mathbb{E}^\frac{1}{q'} \big|\big(\delta\Pi^{(\tau)}_x-\delta\Pi^{(\tau)}_x(y)
-\d\Gtau_{xy} Q \Pi^{(\tau)}_y\big)_\beta (z) \big|^{q'} \\
&\quad\lesssim |z-y|^{\kappa+\alpha} (|z-y|+|y-x|)^{|\beta|-\alpha} (w_x(y)+w_x(z)) \, , 
\label{eq:delta_pi_incr_generic} 
\end{align}
where $q'<q\leq2$, with $q$ denoting a generic conjugate exponent of $2\leq p<\infty$. 
Hence, the constant in $\lesssim$ also depends on $q'<q$ when it comes to estimates of Malliavin derivatives. 
The weights $w_x$ are a small adaptation of those in \cite{LOTT21}, 
and are given in Appendix~\ref{weights}.
Here, $\kappa$ is restricted by 
\begin{subnumcases}
{
    2 < \kappa+2\alpha < \label{kappa}
}
\tfrac{D}{2}+2\alpha, \label{kappa1} \\
\min\big\{|\beta| \, \big|\, \beta \text{ populated and } |\beta|>2 \big\} \label{kappa2},
\end{subnumcases}
where the restriction \eqref{kappa1} is for the weight $w_x(y)$ to be locally (square) integrable, see \cite[(4.36)]{LOTT21}, 
and the restriction $2<\kappa+2\alpha$ originates 
from a reconstruction argument, see \cite[(4.51)]{LOTT21}. 
Since $\alpha>1-D/4$, cf.~Assumption~\ref{ass}~(iii), 
it is possible to choose such a $\kappa$. 
Also the restriction \eqref{kappa2}, which is not present in \cite{LOTT21}, 
is an admissible choice since the set of homogeneities is locally finite. 
The reason for this further restriction on $\kappa$ is that then \cite[(4.71)]{LOTT21} can be stated as 
\begin{equation}\label{ho27}
\mathbb{E}^\frac{1}{q'} | F^{(\tau)}_{xy\beta t}(y)|^{q'} 
\lesssim  (\sqrt[4]{t})^{\kappa+2\alpha-2} (\sqrt[4]{t}+|y-x|)^{|\beta|-2\alpha} w_x(y)\, , 
\end{equation}
with the understanding that $|\beta|\geq2\alpha$ unless the left-hand side vanishes, 
and where $F^{(\tau)}_{xy}$ is given by 
\begin{equation}\label{Fxy}
F^{(\tau)}_{xy} := \delta\Pi^{-(\tau)}_x - \d\Gtau_{xy} Q \Pi^{-(\tau)}_{y} 
- \sum_{k\ge 0} \z_k (\Pi^{(\tau)}_x)^k(y) \Delta (\delta\Pi^{(\tau)}_x - \d\Gtau_{xy} Q \Pi^{(\tau)}_y)
-\delta\xi_\tau\mathsf{1} \, . 
\end{equation}
Since by assumption $\kappa+\alpha>2-\alpha>1$, 
and since $w_x(y)$ is finite for $\delta\xi\in\S$ and $x\neq y$,  
\eqref{eq:delta_pi_incr_generic} implies \eqref{ho29}. 
From \eqref{ho27} and since $\kappa+2\alpha>2$ by \eqref{kappa}, 
we obtain \eqref{ho28}. 
\end{proof}

In view of the already established uniqueness result, 
to prove convergence of $(\Pi^{(\tau)},\Gtau)$
it is tempting to try to appeal to tightness, 
which follows from the uniform bounds \eqref{Pi} and \eqref{Gamma}.
However, difficulties arise since on the level of the law of a random variable it seems to be tricky to preserve the relation between the random variable and its Malliavin derivative. 

Instead, we follow a different approach. 
In fact, the estimates established in \cite{LOTT21} upgrade to estimates of increments and yield that $(\Pi^{(\tau)},\Gtau)$ is Cauchy in $\tau$. 
More precisely, for every $\beta$ there exists $\epsilon>0$ such that
for all $p<\infty$, $x,y\in\R^{1+d}$ and $\tau,\tau'>0$
\begin{align}
\mathbb{E}^\frac{1}{p}|\Pi^{(\tau)}_{x\beta}(y) - \Pi^{(\tau')}_{x\beta}(y)|^p 
&\lesssim (\sqrt[4]{|\tau-\tau'|})^\epsilon \, |x-y|^{|\beta|-\epsilon} \, , \label{Pi_cauchy} \\
\mathbb{E}^\frac{1}{p}|(\Gtau_{xy})_{\beta}^{\gamma} - (\Gtaut_{xy})_{\beta}^{\gamma}|^p 
&\lesssim (\sqrt[4]{|\tau-\tau'|})^\epsilon \, |x-y|^{|\beta|-|\gamma|-\epsilon} \, , \label{Gamma_cauchy}
\end{align}
with the understanding that $|\beta|-|\gamma|-\epsilon>0$ unless the left-hand side vanishes. 
This allows to define the random variables 
\begin{equation}\label{defLimit}
\Pi_{x\beta}(y) := \lim_{\tau\to0} \Pi^{(\tau)}_{x\beta}(y)\, , \quad
(\G_{xy})_\beta^\gamma := \lim_{\tau\to0} (\Gtau_{xy})_\beta^\gamma\, , 
\end{equation}
where both limits have to be understood w.r.t.~$\E^\frac{1}{p}|\cdot|^p$.
To check that $(\Pi,\G)$ defined in this way is indeed a model for \eqref{spde}, 
we will also make use of the corresponding estimates of the Malliavin derivative
\begin{equation}
\E^{\frac{1}{q'}} |\delta\Pi^{(\tau)}_{x\beta}(y) - \delta\Pi^{(\tau')}_{x\beta}(y)|^{q'} 
\lesssim (\sqrt[4]{|\tau-\tau'|})^\epsilon \, |x-y|^{|\beta|-\epsilon} \, \bar w\, , \label{deltaPi_cauchy} 
\end{equation}
together with 
\begin{equation}\label{dGamma_cauchy}
\E^\frac{1}{q'} | (\d\Gtau_{xy})_\beta^\gamma - (\d\Gtaut_{xy})_\beta^\gamma |^{q'} 
\lesssim (\sqrt[4]{|\tau-\tau'|})^{\epsilon} \, |x-y|^{\kappa+|\beta|-|\gamma|-\epsilon} w_x(y) \, ,
\end{equation}
where the weights $\bar{w}$ and $w_x(y)$ are given in Appendix~\ref{weights}.
It is not difficult to see that the identical strategy of \cite{LOTT21} yields these estimates. 
For completeness we walk the reader through and point out the necessary adaptations, see Section~\ref{sec:cauchy}. 
Before we do so, we shall however argue that $(\Pi,\G)$ is indeed~a~model. 

\begin{proposition}[Existence]\label{exist}
Assume that $\E$ satisfies Assumption~\ref{ass}. 
Then there exists a model $(\Pi,\G)$ for \eqref{spde}. 
\end{proposition}

\begin{proof}
Assumption~\ref{ass} is exactly\footnote{
up to the restriction $d=1$ and $\alpha<1/2$ as mentioned at the beginning of this section} 
the setting in which \cite{LOTT21} was carried out. Therefore we just need to check that $(\Pi,\G)$ defined by \eqref{defLimit} satisfies Definition~\ref{def}. 
Applying triangle inequality and the estimate \eqref{Pi}, we have
\begin{align}
\E^\frac{1}{p}|\Pi_{x\beta}(y)|^p 
&\leq \E^\frac{1}{p}| \Pi_{x\beta}(y)-\Pi^{(\tau)}_{x\beta}(y) |^p 
+ \E^\frac{1}{p}| \Pi^{(\tau)}_{x\beta}(y)|^p \\
&\lesssim \E^\frac{1}{p}| \Pi_{x\beta}(y)-\Pi^{(\tau)}_{x\beta}(y) |^p 
+ |x-y|^{|\beta|} \, .
\end{align}
Taking the limit $\tau\to0$ 
yields \eqref{estPi}. 
The argument for \eqref{eh14} is similar. 

Preliminary for what follows, we note that $\Pi_x^-:=(\partial_0-\Delta)P\Pi_x$ satisfies for all $t>0$ and $y\in\R^{1+d}$
\begin{equation}\label{Pi-_converges}
\lim_{\tau\to0}\,\E^\frac{1}{p}| \Pi^{-}_{x\beta t}(y) - \Pi^{-(\tau)}_{x\beta t}(y)|^p =0 \, .
\end{equation}
Indeed, the left-hand side is bounded by 
\begin{equation}
\lim_{\tau\to0}\int_{\R^{1+d}} dz\, |(\partial_0-\Delta)\psi_t(y-z)| \, \E^\frac{1}{p} |\Pi_{x\beta}(z)-\Pi^{(\tau)}_{x\beta}(z)|^p \, ,
\end{equation}
and by the already established \eqref{estPi} we may appeal 
to the dominated convergence theorem to obtain \eqref{Pi-_converges}. 

We turn to \eqref{eh11}, for which we have to show that $\Pi^-_{x0}=\xi$. 
By the triangle inequality and \eqref{eh11} for $\Pi^{(\tau)}_{x0}$ we obtain
\begin{equation}
\E| \Pi^-_{x0t}(y) - \xi_t(y)| 
\leq \E| \Pi^-_{x0t}(y) - \Pi^{-(\tau)}_{x0t}(y)| 
+ \E| \xi_{\tau+t}(y) - \xi_t(y)| \, . 
\end{equation}
The first term on the right-hand side converges to $0$ as $\tau\to0$ by \eqref{Pi-_converges}, 
and so does the second term on the right-hand side, see e.g.~the upcoming \eqref{Pi-_cauchy}. 
Hence $\E|\Pi^-_{x0t}(y)-\xi_t(y)|=0$ for all $t>0$ and $y\in\R^{1+d}$, 
which yields \eqref{eh11}. 
For \eqref{polypart} we note that $\E|\Pi_{x e_\n}(y) - (y-x)^\n| = \E| \Pi_{x e_\n}(y) - \Pi^{(\tau)}_{x e_\n}(y)|$, 
which converges to $0$ as $\tau\to0$. 
Similarly, $\E\Pi_{x\beta}(y)-\E\Pi_{x+z\,\beta}(y+z) 
= \E(\Pi_{x\beta}(y)-\Pi^{(\tau)}_{x\beta}(y) )
+ \E(\Pi^{(\tau)}_{x+z\,\beta}(y+z)-\Pi_{x+z\,\beta}(y+z))$, 
where the right-hand side converges to $0$ as $\tau\to0$. 
The arguments for \eqref{recenter}, \eqref{Gamma_zn}, \eqref{Gamma_zk} 
and \eqref{recenter-} are similar. 

We turn to Malliavin differentiability. 
By dualizing \eqref{deltaPi_cauchy} we observe that $\frac{\partial}{\partial\xi}\Pi^{(\tau)}_{x\beta}(y)$ 
is a Cauchy sequence w.r.t.~$\E^\frac{1}{2}\|\cdot\|_*^2$. 
Together with \eqref{Pi_cauchy} and the closability assumption of $\frac{\partial}{\partial\xi}$, 
we obtain $\Pi_{x\beta}(y)\in\H$ and 
$\frac{\partial}{\partial\xi}\Pi^{(\tau)}_{x\beta}(y)\to\frac{\partial}{\partial\xi}\Pi_{x\beta}(y)$ w.r.t.~$\E^\frac{1}{2}\|\cdot\|_*^2$ as $\tau\to0$.
With this at hand we are in position to check \eqref{ho29} and \eqref{ho28}. 
First, note that \eqref{dGamma_cauchy} implies that for fixed $\delta\xi\in\S$ and $x\neq y$, 
$(\d\Gtau_{xy})_\beta^\gamma\to(\d\G_{xy})_\beta^\gamma$ w.r.t.~$\E^\frac{1}{q'}|\cdot|^{q'}$ as $\tau\to0$, 
which serves as the definition of $\d\G_{xy}$. 
Since $(\d\Gtau_{xy})_\beta^{e_\n} = \d\pi^{(\n)(\tau)}_{xy\beta}$, this implies in particular 
$\d\pi^{(\n)(\tau)}_{xy\beta}\to\d\pi^{(\n)}_{xy\beta}$ w.r.t.~$\E^\frac{1}{q'}|\cdot|^{q'}$ as $\tau\to0$, 
which we again see as the definition of $\d\pi^{(\n)}_{xy}$. 
By H\"older's inequality and $\Gtau_{xy}\to\G_{xy}$ w.r.t.~$\E^\frac{1}{p}|\cdot|^p$ as $\tau\to0$, 
one finds that the structure \eqref{dGamma} is preserved in the limit. 
Analogous to checking that \eqref{estPi} and \eqref{eh14} are preserved in the limit, 
we obtain that \eqref{eq:delta_pi_incr_generic} is preserved in the limit. 
In the same way we obtained \eqref{Pi-_converges}, 
we use \eqref{deltaPi_cauchy} to argue in favor of
\begin{equation}
\lim_{\tau\to0} \, \E^\frac{1}{q'}| \delta\Pi^-_{x\beta t}(y) - \delta\Pi^{-(\tau)}_{x\beta t}(y) |^{q'} = 0\, , 
\end{equation}
which allows together with the above established convergences to also preserve \eqref{ho27} in the limit. 
Having \eqref{eq:delta_pi_incr_generic} and \eqref{ho27}, 
we can proceed as in the proof of Lemma~\ref{smooth_model} to obtain \eqref{ho29} and \eqref{ho28}.
\end{proof}

\subsection{Proof of the Cauchy property}\label{sec:cauchy}

In this section we give the argument for \eqref{Pi_cauchy}, 
\eqref{Gamma_cauchy}, \eqref{deltaPi_cauchy} and \eqref{dGamma_cauchy}.
To not change the exponent $\epsilon$ from line to line, 
we fix a multi-index $\beta$, and choose $\epsilon$ that works for all $\beta'\preceq\beta$. 
As it turns out, this is the case for $\epsilon$ satisfying 
\begin{subnumcases}
{
    0<\epsilon< \label{epsilon}
}
	|\beta'|-|\beta''|, \quad \text{for all }\beta',\beta''\text{ satisfying } |\beta''|<|\beta'|\leq|\beta|, \label{epsilon1}\\
	\kappa+2\alpha-2. \label{epsilon2}
\end{subnumcases}
This is an admissible choice\footnote{
w.l.o.g. choose $\beta$ such that $|\beta|\geq 2\alpha$, 
so that there exist $\beta',\beta''$ with $|\beta''|<|\beta'|\leq|\beta|$}, 
since the set of homogeneities is locally finite and since $\kappa+2\alpha-2>0$, see \eqref{kappa}. 
In the following, we will establish 
$\eqref{Pi_cauchy}_{\beta'}$,
$\eqref{Gamma_cauchy}_{\beta'}$,
$\eqref{deltaPi_cauchy}_{\beta'}$ and 
$\eqref{dGamma_cauchy}_{\beta'}$
for this choice of $\epsilon$ and for all $\beta'\preceq\beta$.
Similar to \cite{LOTT21}, we shall 
include the upcoming estimates 
$\eqref{Pi-_cauchy}_{\beta'}$ -- 
$\eqref{deltaPi-_cauchy}_{\beta'}$, 
$\eqref{pin_cauchy}_{\beta'}$ --
$\eqref{deltaPi-_inc_cauchy}_{\beta'}$ and 
$\eqref{deltapin_cauchy}_{\beta'}$ -- 
$\eqref{dpin_cauchy}_{\beta'}$ 
for the same choice of $\epsilon$ and for all $\beta'\preceq\beta$. 

First, these estimates are verified for all purely polynomial multi-indices. 
In fact, all left-hand sides of all these estimates are identically zero for purely polynomial $\beta'$. 
For $\Pi_{x\beta'}$ and $\delta\Pi_{x\beta'}$ this is the case by \eqref{polypart}, 
and for $\Pi^-_{x\beta'}$ and $\delta\Pi^-_{x\beta'}$ this is a consequence thereof.
For $(\G_{xy})_{\beta'}$ and $(\delta\G_{xy})_{\beta'}$ this follows from \eqref{Gamma_preserves_Tbar} and \eqref{Gamma_poly}, 
and for $(\d\G_{xy})_{\beta'}$ this is a consequence of \eqref{dGamma_preserves}. 
For $\pi^{(\n)}_{xy\beta'}$, $\delta\pi^{(\n)}_{xy\beta'}$ and $\d\pi^{(\n)}_{xy\beta'}$ this follows 
from the corresponding statements of $(\G_{xy})_{\beta'}$, $(\delta\G_{xy})_{\beta'}$ and $(\d\G_{xy})_{\beta'}$.

For the remaining multi-indices we proceed by induction w.r.t.~$\prec$, 
where the already established estimates for the multi-indices $\beta=e_\n$ with $|\n|=1$ serve as the base case. 
In the induction step, we follow the strategy of \cite{LOTT21}, 
where in addition we will frequently use telescoping for products in form of
\begin{equation}\label{telescope}
\pi_1\cdots\pi_k - \pi'_1\cdots\pi'_k 
= \sum_{i=1}^k \pi_1\cdots\pi_{i-1} (\pi_i-\pi'_i) \pi'_{i+1}\cdots\pi'_k \, . 
\end{equation}
The main philosophy is as follows: 
we use the algebraic identities of \cite{LOTT21}, 
apply \eqref{telescope} on the products that arise, 
and estimate all factors as in \cite{LOTT21} except increments, 
where we apply the new established estimates. 
Clearly, algebraic properties like triangularity are preserved, 
so that the inductive structure of the proof does not change. 
Also the homogeneities add up in the same way as in \cite{LOTT21}, 
although sometimes (especially in integration) care has to be taken 
that the $\epsilon$-loss in the exponents does not create problems. 
The following lemma illustrates this in a simple setup. 

In all following statements, Lemma~\ref{alg1} -- Lemma~\ref{3pt4}, 
we shall always assume that all estimates from \cite{LOTT21} hold, 
without explicitly stating them in the assumption. 
Furthermore, all statements of the induction hypothesis will be implicitly 
assumed to hold for all integrability exponents $p<\infty$ and $q'<2$, 
all space-time points $x,y,z\in\R^{1+d}$, 
and every convolution parameter $t>0$. 
For example, by $\eqref{Pi_cauchy}_{\prec\beta}$ we mean the 
estimate for all $p<\infty$, $x,y\in\R^{1+d}$ and for all multi-indices $\beta'\prec\beta$. 

\begin{lemma}[Algebraic argument I']\label{alg1}
Assume that $\eqref{pin_cauchy}_{\prec\beta}$ holds.
Then $\eqref{Gamma_cauchy}_\beta^\gamma$ holds for all $\gamma$ not purely polynomial. 
\end{lemma}

\begin{proof}
We use the exponential formula \cite[(2.44)]{LOTT21} to see
\begin{align}\label{exponentialformula}
&(\Gtau_{xy})_\beta^\gamma-(\Gtaut_{xy})_\beta^\gamma\\ 
&= \sum_{\substack{k\geq0 \\ \n_1,\dots,\n_k \\ \beta_1+\cdots+\beta_{k+1}=\beta}} \big(\pi^{(\n_1)(\tau)}_{xy\beta_1} 
\cdots \pi^{(\n_k)(\tau)}_{xy\beta_k} - \pi^{(\n_1)(\tau')}_{xy\beta_1}
\cdots \pi^{(\n_k)(\tau')}_{xy\beta_k}\big) \big(D^{(\n_1)}\cdots D^{(\n_k)}\big)_{\beta_{k+1}}^\gamma \, .
\end{align}
After telescoping \eqref{telescope}, we follow the proof of \cite[Lemma~4.3]{LO22} to obtain \eqref{Gamma_cauchy} by H\"older's inequality and using estimates of $\pi^{(\n)(\tau)}_{xy\beta'}$. 
The only modification is that on the increment $\pi^{(\n_i)(\tau)}_{xy\beta_i}-\pi^{(\n_i)(\tau')}_{xy\beta_i}$ we use the estimate \eqref{pin_cauchy}.
Note that the triangularity \eqref{triGamma_hom} implies that the left-hand side of \eqref{Gamma_cauchy} vanishes for $\beta=\gamma$, 
and for all other multi-indices where the left-hand side does not vanish 
we have $|\beta|>|\gamma|$, 
which by the choice of $\epsilon$ in \eqref{epsilon1} implies that $|\beta|-|\gamma|-\epsilon>0$.
\end{proof}

Next, we establish 
\begin{equation}\label{Pi-_cauchy} 
\mathbb{E}^\frac{1}{p}|\Pi^{-(\tau)}_{x\beta t}(y) - \Pi^{-(\tau')}_{x\beta t}(y)|^p 
\lesssim (\sqrt[4]{|\tau-\tau'|})^\epsilon \, (\sqrt[4]{t})^{\alpha-2-\epsilon}(\sqrt[4]{t}+|x-y|)^{|\beta|-\alpha} \, , 
\end{equation}
where in the proof we distinguish multi-indices $\beta$ satisfying $|\beta|<2$ from $|\beta|>2$, and start with the latter. 

\begin{lemma}[Reconstruction I']\label{rec1}
Assume $|\beta|>2$, that $\eqref{Pi-_cauchy}_{\prec\beta}$ and
$\eqref{pin_cauchy}_{\prec\beta}$ hold, 
and that $\eqref{Gamma_cauchy}_\beta^{\gamma}$ holds for all $\gamma$ not purely polynomial.
Then $\eqref{Pi-_cauchy}_\beta$ holds.
\end{lemma}
\begin{proof}
As in \cite[Proposition~4.2]{LOTT21}, \eqref{Pi-_cauchy} follows by a standard reconstruction argument, provided we have continuity in the base point in form of 
\begin{align}\label{eh25}
\mathbb{E}^\frac{1}{p}|(\Pi^{-(\tau)}_{x}-\Pi^{-(\tau)}_{y})_{\beta t}(y) 
- (\Pi^{-(\tau')}_{x}-\Pi^{-(\tau')}_{y})_{\beta t}(y)|^p \\
\lesssim (\sqrt[4]{|\tau-\tau'|})^\epsilon \, (\sqrt[4]{t})^{\alpha-2-\epsilon}(\sqrt[4]{t}+|x-y|)^{|\beta|-\alpha} \, , 
\end{align}
and the sum of the exponents $\alpha-2-\epsilon$ and $|\beta|-\alpha$ is positive. 
This positivity is guaranteed by \eqref{epsilon1} due to the assumption $|\beta|>2$. 
Appealing to the recentering \eqref{recenter-} of $\Pi^-$ and telescoping \eqref{telescope}, 
the remaining part of the proof follows the lines of the one of \cite[Proposition~4.2]{LOTT21}, 
with the only modification that on the increments we use the estimates  \eqref{Gamma_cauchy}, \eqref{Pi-_cauchy} and \eqref{pin_cauchy} on 
$\Gtau$, $\Pi^{-(\tau)}$ and $\pi^{(\n)(\tau)}$, respectively. 
\end{proof}

For multi-indices $\beta$ satisfying $|\beta|<2$ we first note that 
\begin{equation}\label{expectPi-_cauchy}
|\mathbb{E} \Pi^{-(\tau)}_{x\beta t}(y)-\E\Pi^{-(\tau')}_{x\beta t}(y)|
\lesssim (\sqrt[4]{|\tau-\tau'|})^\epsilon \, (\sqrt[4]{t})^{\alpha-2-\epsilon}(\sqrt[4]{t}+|x-y|)^{|\beta|-\alpha} \, , 
\end{equation}
which follows by the same argument as in \cite[Proposition~4.7]{LOTT21}, 
again by telescoping \eqref{telescope}, and using on increments the estimates $\eqref{Pi-_cauchy}_{\prec\beta}$ and the already established $\eqref{Gamma_cauchy}_\beta^{\gamma}$ for $\gamma$ not purely polynomial. 
It remains to prove 
\begin{equation}\label{deltaPi-_cauchy}
\E^\frac{1}{q'} | \delta\Pi^{-(\tau)}_{x\beta t}(y) - \delta\Pi^{-(\tau')}_{x\beta t}(y) |^{q'}
\lesssim (\sqrt[4]{|\tau-\tau'|})^\epsilon \, (\sqrt[4]{t})^{\alpha-2-\epsilon} (\sqrt[4]{t}+|x-y|)^{|\beta|-\alpha} \bar w \, , 
\end{equation}
which by dualization and \eqref{wbar<Sobolev}, and using the spectral gap inequality \eqref{sg} will together with $\eqref{expectPi-_cauchy}_\beta$ analogous to \cite[Section~4.3]{LOTT21} imply $\eqref{Pi-_cauchy}_\beta$ for multi-indices $\beta$ with $|\beta|<2$. 
We postpone the argument for \eqref{deltaPi-_cauchy} for a moment, and continue with the following integration lemma. 

\begin{lemma}[Integration I']\label{int1}
Assume that $\eqref{Pi-_cauchy}_\beta$ holds.
Then $\eqref{Pi_cauchy}_\beta$ holds.
\end{lemma}

\begin{proof}
The proof follows the one of \cite[Proposition~4.3]{LOTT21} based on the solution formula 
\begin{equation}\label{representation}
\Pi^{(\tau)}_{x\beta}(y) = -\int_0^\infty dt\, (1-{\rm T}_x^{|\beta|})(\partial_0+\Delta)
\Pi^{-(\tau)}_{x\beta t}(y) \, . 
\end{equation}
A little bit of care has to be taken concerning integrability at $t=0$ and $t=\infty$ due to the loss of $\epsilon$ in the exponents, 
which we check in the following by pointing out how the exponents change compared to the aforementioned reference. 
We apply \eqref{representation} where we replace 
$\Pi^{(\tau)}_{x\beta}(y)$ by the increment 
$\Pi^{(\tau)}_{x\beta}(y)-\Pi^{(\tau')}_{x\beta}(y)$, 
and split the integral into the near field $t\leq|x-y|^4$ and the far field $t\geq|x-y|^4$.
In the near field, we split the integral into its contributions form $1$ and ${\rm T}_x^{|\beta|}$, 
where ${\rm T}_x^{|\beta|}$ denotes the operation of taking 
the Taylor polynomial of degree $<|\beta|$ centered at $x$.
On the former, the order of vanishing in $\sqrt[4]{t}$ is given by $\alpha-4-\epsilon$, which is integrable at $t=0$ since $\alpha-\epsilon>0$ by the choice of $\epsilon$ in \eqref{epsilon1}. 
On the latter, the order of vanishing in $\sqrt[4]{t}$ is given by $|\beta|-4-|\m|-\epsilon$ with $\m$ restricted to $|\m|<|\beta|$. 
This restriction of $\m$ implies by \eqref{epsilon1} that $|\beta|-|\m|-\epsilon>0$, 
hence also this contribution is integrable at $t=0$. 
For the far field contribution, we note that the Taylor remainder has a growth rate in $\sqrt[4]{t}$ given by $|\beta|-4-|\m|-\epsilon$ where this time $\m$ is restricted to $|\m|\geq|\beta|$. 
Therefore $|\beta|-|\m|-\epsilon<0$, which makes this contribution  integrable at $t=\infty$. 
\end{proof}

Using the estimates \eqref{Pi_cauchy} on $\Pi$ and \eqref{Gamma_cauchy} on $\G$, 
the proof of the following lemma follows the one of \cite[Proposition 4.4]{LOTT21}, 
in addition by telescoping \eqref{telescope}, 
and using that $|\beta|>|\n|$ implies $|\beta|>|\n|+\epsilon$ by the choice of $\epsilon$ in \eqref{epsilon1}. 

\begin{lemma}[Three-point argument I']\label{3pt1}
Assume that $\eqref{Pi_cauchy}_{\preceq\beta}$ holds, and that 
$\eqref{Gamma_cauchy}_{\beta}^{\gamma}$ holds for all $\gamma$ not purely polynomial. 
Then 
\begin{equation}\label{pin_cauchy} 
\mathbb{E}^\frac{1}{p}|\pi^{(\n)(\tau)}_{xy\beta} - \pi^{(\n)(\tau')}_{xy\beta}|^p 
\lesssim (\sqrt[4]{|\tau-\tau'|})^\epsilon \, |x-y|^{|\beta|-|\n|-\epsilon} \, , 
\end{equation}
with the understanding that $|\beta|-|\n|-\epsilon>0$ unless the left-hand side vanishes. 
In particular, $\eqref{Gamma_cauchy}_\beta^\gamma$ holds for all $\gamma$.
\end{lemma}

In the remaining part of this section we provide the argument for \eqref{deltaPi-_cauchy}. 
First, by applying the derivative $\delta$ to the exponential formula \eqref{exponentialformula} 
and using Leibniz rule, we obtain the following lemma.

\begin{lemma}[Algebraic argument II']\label{alg2}
Assume that $\eqref{pin_cauchy}_{\prec\beta}$\,and $\eqref{deltapin_cauchy}_{\prec\beta}$ hold. 
Then for all $\gamma$ not purely polynomial 
\begin{equation}\label{deltaGamma_cauchy} 
\mathbb{E}^\frac{1}{q'}|(\delta\Gtau_{xy})_{\beta}^{\gamma} - (\delta\Gtaut_{xy})_{\beta}^{\gamma}|^{q'} 
\lesssim (\sqrt[4]{|\tau-\tau'|})^\epsilon \, |x-y|^{|\beta|-|\gamma|-\epsilon} \bar w \, , 
\end{equation}
with the understanding that $|\beta|-|\gamma|-\epsilon>0$ unless the left-hand side vanishes. 
\end{lemma}

Next, we observe that analogous to Lemma~\ref{alg1}, 
the proof of \cite[Proposition~4.11]{LOTT21} in addition with telescoping \eqref{telescope}, the estimates \eqref{Gamma_cauchy} on $\G$ and \eqref{dpin_inc_cauchy} on $\d\pi^{(\n)}$ yield the following.

\begin{lemma}[Algebraic argument III']\label{alg3}
Assume that $\eqref{dpin_inc_cauchy}_{\prec\beta}$ holds, 
and that $\eqref{Gamma_cauchy}_\beta^\gamma$ holds for all $\gamma$ not purely polynomial.
Then for all $\gamma$ not purely polynomial
\begin{align}
&\mathbb{E}^\frac{1}{q'}\big|\big((\d\Gtau_{xy}-\d\Gtau_{xz}\Gtau_{zy}) Q \big)_\beta^\gamma - 
\big((\d\Gtaut_{xy}-\d\Gtaut_{xz}\Gtaut_{zy}) Q \big)_\beta^\gamma \big|^{q'} \\
&\lesssim \Big(\sum_{|\n|=1} \mathbf{1}_{\gamma(\n)=0}\, (\sqrt[4]{|\tau-\tau'|})^\epsilon \, |y-z|^{\kappa+\alpha-\epsilon} 
(|y-z|+|z-x|)^{|\beta|-|\gamma|-\alpha} \\  
&\hphantom{\lesssim} + \sum_{|\n|=1}\mathbf{1}_{\gamma(\n)=1}\, (\sqrt[4]{|\tau-\tau'|})^\epsilon \, |y-z|^{\kappa+\alpha-1-\epsilon} 
(|y-z|+|z-x|)^{|\beta|-|\gamma|-\alpha+1}\Big) \\ 
&\hphantom{\lesssim}\ \cdot (w_x(y)+w_x(z)) \, . \label{dGamma_inc_cauchy} 
\end{align}
\end{lemma}

Similarly, from \cite[Proposition~4.16]{LOTT21} and by telescoping \eqref{telescope} we obtain the following.

\begin{lemma}[Algebraic argument IV']\label{alg4}
Assume that $\eqref{dpin_cauchy}_{\prec\beta}$ holds, 
and that $\eqref{Gamma_cauchy}_{\preceq\beta}^{\gamma}$ holds for all $\gamma$ not purely polynomial.
Then $\eqref{dGamma_cauchy}_\beta^\gamma$ holds for all $\gamma$ not purely polynomial. 
\end{lemma}

Equipped with these estimates, we can carry out a reconstruction argument. 

\begin{lemma}[Reconstruction III']\label{rec3}
Assume that $\eqref{Pi_cauchy}_{\prec\beta}$,\,$\eqref{Gamma_cauchy}_{\prec\beta}$,\,$\eqref{Pi-_cauchy}_{\prec\beta}$\,and 
$\eqref{deltaPi_inc_cauchy}_{\prec\beta}$ hold, 
and that $\eqref{dGamma_cauchy}_{\preceq\beta}^\gamma$ and $\eqref{dGamma_inc_cauchy}_{\preceq\beta}^\gamma$ hold for all $\gamma$ not purely polynomial. 
Then 
\begin{align}
&\mathbb{E}^\frac{1}{q'} \big| \big(\delta\Pi^{-(\tau)}_x - \d\Gtau_{xz}Q\Pi^{-(\tau)}_z\big)_{\beta t}(y)
- \big(\delta\Pi^{-(\tau')}_x - \d\Gtaut_{xz}Q\Pi^{-(\tau')}_z\big)_{\beta t}(y) \big|^{q'} \\
&\lesssim  (\sqrt[4]{|\tau-\tau'|})^\epsilon \, (\sqrt[4]{t})^{\alpha-2-\epsilon} (\sqrt[4]{t}+|y-z|)^{\kappa}
(\sqrt[4]{t}+|y-z|+|x-z|)^{|\beta|-\alpha} \\
&\ \cdot (w_x(y)+w_x(z)) \, . \label{deltaPi-_inc_cauchy}
\end{align}
\end{lemma}

\begin{proof}
The proof follows the one of \cite[Proposition~4.12]{LOTT21}. 
We will first establish a continuity condition of $F^{(\tau)}_{xy}-F^{(\tau')}_{xy}$ in the secondary base point $y$, 
\begin{align}
&\E^\frac{1}{q'} \big| (F^{(\tau)}_{xy}-F^{(\tau')}_{xy})_{\beta t}(y) 
- (F^{(\tau)}_{xz}-F^{(\tau')}_{xz})_{\beta t}(y) \big|^{q'} \\
&\lesssim (\sqrt[4]{|\tau-\tau'|})^\epsilon \, (\sqrt[4]{t})^{\alpha-2-\epsilon} (\sqrt[4]{t}+|y-z|)^{\kappa+\alpha} \\
&\ \cdot (\sqrt[4]{t}+|y-z|+|x-z|)^{|\beta|-2\alpha} (w_x(y)+w_x(z)) \, ,
\label{eh27}
\end{align}
where we recall that $F^{(\tau)}_{xy}$ is defined in \eqref{Fxy}. 
Since the sum of the exponents $\alpha-2-\epsilon$ and $\kappa+\alpha$ is positive, which is a consequence of \eqref{epsilon2}, 
and since 
\begin{equation}
\lim_{t\to0} \E^\frac{1}{q'}| F^{(\tau)}_{xz\beta t}(z)-F^{(\tau')}_{xz\beta t}(z)|^{q'} = 0 \, ,
\end{equation}
which is a consequence of \cite[(4.68)]{LOTT21} by the triangle inequality,
a general reconstruction argument implies\footnote{
it actually implies the stronger estimate where $\kappa$ is replaced by $\kappa+\alpha$ and $|\beta|-\alpha$ is replaced by $|\beta|-2\alpha$,
which however we won't make use of} 
\begin{align}
\mathbb{E}^\frac{1}{q'} | F^{(\tau)}_{xz\beta t}(y) - F^{(\tau')}_{xz\beta t}(y)|^{q'} 
&\lesssim  (\sqrt[4]{|\tau-\tau'|})^\epsilon \, (\sqrt[4]{t})^{\alpha-2-\epsilon} (\sqrt[4]{t}+|y-z|)^{\kappa} \\
&\ \cdot (\sqrt[4]{t}+|y-z|+|x-z|)^{|\beta|-\alpha} (w_x(y)+w_x(z)) \, . 
\end{align}
Together with 
\begin{align}
&\E^\frac{1}{q'} \big| \big( \sum_k \z_k (\Pi^{(\tau)}_x (z))^k 
\Delta (\delta\Pi^{(\tau)}_x - \d\Gtau_{xz}Q\Pi^{(\tau)}_z) + \delta\xi_\tau 1\big)_{\beta t}(y) \\
&\qquad- \big( \sum_k \z_k (\Pi^{(\tau')}_x (z))^k 
\Delta (\delta\Pi^{(\tau')}_x - \d\Gtaut_{xz}Q\Pi^{(\tau')}_z) + \delta\xi_{\tau'} 1\big)_{\beta t}(y) \big|^{q'} \\
&\lesssim (\sqrt[4]{|\tau-\tau'|})^\epsilon \, (\sqrt[4]{t})^{\alpha-2-\epsilon} (\sqrt[4]{t}+|y-z|)^{\kappa} \\
&\ \cdot (\sqrt[4]{t}+|y-z|+|x-z|)^{|\beta|-\alpha} (w_x(y)+w_x(z)) \, ,
\label{eh28}
\end{align}
this implies the desired \eqref{deltaPi-_inc_cauchy}. 

We shall now give the argument for \eqref{eh27} and \eqref{eh28}.
Actually, the argument for \eqref{eh27} is identical to the one of \cite[(4.69)]{LOTT21}, with the same modification as always: 
use telescoping \eqref{telescope} and apply on increments the estimates of the assumption of the lemma. 
The same applies to \eqref{eh28}, 
however, the following additional estimate has to be carried out for $\beta=0$:
\begin{equation}
\E^\frac{1}{q'}|\delta\xi_{t+\tau}(y)-\delta\xi_{t+\tau'}(y)|^{q'}
\lesssim (\sqrt[4]{|\tau-\tau'|})^\epsilon \, (\sqrt[4]{t})^{\alpha-2+\kappa-\epsilon} w(y)\, .
\end{equation}
We obtain this estimate by interpolating between the two estimates
\begin{align}
\E^\frac{1}{q'}|\delta\xi_{t+\tau}(y)-\delta\xi_{t+\tau'}(y)|^{q'}
&\lesssim (\sqrt[4]{t})^{\alpha-2+\kappa} w(y) \, , \\
\E^\frac{1}{q'}|\delta\xi_{t+\tau}(y)-\delta\xi_{t+\tau'}(y)|^{q'}
&\lesssim |\tau-\tau'| \, (\sqrt[4]{t})^{\alpha-2+\kappa-4} w(y) \, .
\end{align}
The former is a consequence of the triangle inequality and 
\eqref{basecase} with $\n=\0$, since $\alpha-2+\kappa\leq0$ by \eqref{kappa2}.
For the latter, we assume w.l.o.g.~that $\tau'\leq\tau$, so that
\begin{equation}
\E^\frac{1}{q'}| \delta\xi_{t+\tau}(y)-\delta\xi_{t+\tau'}(y)|^{q'} 
\leq \int_{t+\tau'}^{t+\tau} ds\, \E^\frac{1}{q'}| \partial_s \delta\xi_s(y) |^{q'} \, .
\end{equation}
By the defining property \eqref{psi} of $\psi$ we have $\partial_s\delta\xi_s = (\partial_0^2-\Delta^2)\delta\xi_s$, 
and the claim follows from \eqref{basecase} with $|\n|=4$. 
\end{proof}

As in \cite[Proposition~4.18]{LOTT21}, this allows to achieve the estimate \eqref{deltaPi-_cauchy} 
by appealing to 
\begin{align}
\delta\Pi^{-(\tau)}_x - \delta\Pi^{-(\tau')}_x 
&= (\delta\Pi^{-(\tau)}_x - \d\Gtau_{xz}Q\Pi^{-(\tau)}_z) - (\delta\Pi^{-(\tau')}_x - \d\Gtaut_{xz}Q\Pi^{-(\tau')}_z) \\
&\,+ \d\Gtau_{xz}Q\Pi^{-(\tau)}_z - \d\Gtaut_{xz}Q\Pi^{-(\tau')}_z \, , 
\end{align}
and averaging in the secondary base point $z$ using \eqref{average_ball_w}.

\begin{lemma}\label{averaging}
Assume that $\eqref{Pi-_cauchy}_{\prec\beta}$, $\eqref{deltaPi-_cauchy}_{\prec\beta}$, 
$\eqref{dGamma_cauchy}_{\beta}$ and $\eqref{deltaPi-_inc_cauchy}_{\beta}$ hold,
and that $\eqref{Gamma_cauchy}_\beta^\gamma$ and $\eqref{deltaGamma_cauchy}_\beta^\gamma$ 
hold for all $\gamma$ not purely polynomial. 
Then $\eqref{deltaPi-_cauchy}_\beta$ holds.
\end{lemma}

The same arguments as in Lemma~\ref{int1} apply to see the following.

\begin{lemma}[Integration II']\label{int2}
Assume that $\eqref{deltaPi-_cauchy}_\beta$ holds. 
Then $\eqref{deltaPi_cauchy}_\beta$ holds.
\end{lemma}

This in turn serves as input for the next lemma, 
the proof of which is analogous to the one of \cite[Proposition~4.10]{LOTT21}, using in addition telescoping \eqref{telescope}.

\begin{lemma}[Three-point argument II']\label{3pt2}
Assume that $\eqref{Pi_cauchy}_{\prec\beta}$, $\eqref{deltaPi_cauchy}_{\preceq\beta}$ hold, 
and that $\eqref{Gamma_cauchy}_\beta^\gamma$ and $\eqref{deltaGamma_cauchy}_\beta^\gamma$ hold for all $\gamma$ not purely polynomial.
Then 
\begin{equation}\label{deltapin_cauchy}
\E^{\frac{1}{q'}} |\delta\pi^{(\n)(\tau)}_{xy\beta} - \delta\pi^{(\n)(\tau')}_{xy\beta}|^{q'} 
\lesssim (\sqrt[4]{|\tau-\tau'|})^\epsilon |y-x|^{|\beta|-|\n|-\epsilon} \, \bar w \, , 
\end{equation}
with the understanding that $|\beta|-|\n|-\epsilon>0$ unless the left-hand side vanishes. 
In particular, $\eqref{deltaGamma_cauchy}_\beta^\gamma$ holds for all $\gamma$.
\end{lemma}

Furthermore, \eqref{deltaPi-_inc_cauchy} can be integrated to obtain the following.

\begin{lemma}[Integration III']\label{int3}
Assume that $\eqref{deltaPi-_inc_cauchy}_\beta$ holds.
Then 
\begin{align}
&\mathbb{E}^\frac{1}{q'} \big|\big(\delta\Pi^{(\tau)}_x-\delta\Pi^{(\tau)}_x(y)
-\d\Gtau_{xy} Q \Pi^{(\tau)}_y\big)_\beta (z) \\
&\qquad- \big(\delta\Pi^{(\tau')}_x-\delta\Pi^{(\tau')}_x(y)
-\d\Gtaut_{xy} Q \Pi^{(\tau')}_y\big)_\beta (z) \big|^{q'} \\
&\lesssim (\sqrt[4]{|\tau-\tau'|})^\epsilon \, |z-y|^{\kappa+\alpha-\epsilon} 
(|z-y|+|y-x|)^{|\beta|-\alpha} (w_x(y)+w_x(z)) \, . 
\label{deltaPi_inc_cauchy}
\end{align}
\end{lemma}
\begin{proof}
The proof follows \cite[Proposition~4.14]{LOTT21}, 
where as in Lemma~\ref{int1} some care has to be taken concerning integrability at $t=0$ and $t=\infty$.
In the near field range $t\leq|y-z|^4$ it is enough to mention that by \eqref{epsilon2} we have 
$\alpha-1+\kappa-\epsilon>0$, and that by \eqref{epsilon1} we have $\alpha-\epsilon>0$.
In the far field range $t\geq\max\{|y-z|^4,|x-z|^4\}$ 
the factor $t^{-\epsilon}$ actually improves integrability,
and in the intermediate range $|y-z|^4\leq t\leq|x-z|^4$ the question of integrability doesn't come up.
\end{proof}

Once more, we can as in \cite[Proposition~4.15]{LOTT21} use this estimate,
together with \eqref{telescope}, to obtain the next lemma.

\begin{lemma}[Three-point argument III']\label{3pt3}
Assume that $\eqref{Pi_cauchy}_{\prec\beta}$, 
$\eqref{pin_cauchy}_{\prec\beta}$ and $\eqref{deltaPi_inc_cauchy}_\beta$ hold, and that $\eqref{Gamma_cauchy}_{\prec\beta}^\gamma$, 
$\eqref{dGamma_cauchy}_\beta^\gamma$ and $\eqref{dGamma_inc_cauchy}_\beta^\gamma$ hold for all $\gamma$ not purely polynomial. 
Then for all $\n$ with $|\n|\leq1$, 
\begin{align}
&\E^\frac{1}{q'} \big| \big(\d\pi^{(\n)(\tau)}_{xy} - \d\pi^{(\n)(\tau)}_{xz} -\d\Gtau_{xz}\pi^{(\n)(\tau)}_{zy}\big)_\beta \\
&\qquad- \big(\d\pi^{(\n)(\tau')}_{xy} - \d\pi^{(\n)(\tau')}_{xz} - \d\Gtaut_{xz}\pi^{(\n)(\tau')}_{zy} \big)_\beta \big|^{q'} \\
&\lesssim (\sqrt[4]{|\tau-\tau'|})^\epsilon \, |y-z|^{\kappa+\alpha-|\n|-\epsilon} (|y-z|+|z-x|)^{|\beta|-\alpha} 
(w_x(y)+w_x(z)) \, . \label{dpin_inc_cauchy}
\end{align}
In particular, $\eqref{dGamma_inc_cauchy}_\beta^\gamma$ holds for all $\gamma$.
\end{lemma}

Finally, the same strategy as in \cite[Proposition~4.17]{LOTT21} yields the following lemma, 
which concludes the induction step.

\begin{lemma}[Three-point argument IV']\label{3pt4}
Assume that $\eqref{Pi_cauchy}_{\prec\beta}$, 
$\eqref{deltaPi_cauchy}_{\prec\beta}$ and
$\eqref{deltaPi_inc_cauchy}_{\prec\beta}$ hold, 
and that $\eqref{dGamma_cauchy}_\beta^\gamma$ holds for all $\gamma$ not purely polynomial. 
Then for all $\n$ with $|\n|\leq1$,
\begin{equation}\label{dpin_cauchy}
\E^\frac{1}{q'} | \d\pi^{(\n)(\tau)}_{xy\beta} - \d\pi^{(\n)(\tau')}_{xy\beta} |^{q'} 
\lesssim (\sqrt[4]{|\tau-\tau'|})^\epsilon \, |x-y|^{\kappa+|\beta|-|\n|-\epsilon} \, w_x(y) \, . 
\end{equation}
In particular, $\eqref{dGamma_cauchy}_\beta^\gamma$ holds for all $\gamma$.
\end{lemma}


\section{Convergence of models}\label{sec:convergence}

In this section we provide the proof of Theorem~\ref{thm2}. 
A first technical ingredient is the following lemma, 
which is concerned with upgrading the probabilistic convergence $\xi_n\to\xi$ 
in the space of tempered distributions to convergence with respect to 
an annealed H\"older-type norm. 

\begin{lemma}\label{lem:convergence}
Assume that the laws of $\xi_n,\xi$ satisfy Assumption~\ref{ass}, 
uniformly in $n$, 
and that $\xi_n\to\xi$ either almost surely or w.r.t.~$\E^\frac{1}{p}|\cdot|^p$ 
on the space of tempered distributions. 
Then $\xi_n\to\xi$ with respect to the annealed norm 
\begin{equation}\label{annealed_xi}
\| \xi\|_{\alpha-2,\epsilon} 
:= \sup_{t>0} \sup_{x\in\R^{1+d}} \frac{\E^\frac{1}{p'}| \xi_t(x)|^{p'} }{
(\sqrt[4]{t})^{\alpha-2-\epsilon} (1+\sqrt[4]{t}+|x|)^{2\epsilon}}
\end{equation}
for every $\epsilon>0$, 
where $p'<\infty$ is arbitrary in case of assuming almost sure convergence, 
and $p'=p$ in case of assuming convergence w.r.t.~$\E^\frac{1}{p}|\cdot|^p$. 
\end{lemma}

\noindent
Here, we keep track of the small scale regularity 
as well as the large scale regularity 
as an artefact of working on the whole space. 
The reason for losing an $\epsilon$ on small and large scales, 
although working with annealed norms, 
is a consequence of ``trading regularity for convergence''. 
This was already present in \eqref{Pi-_cauchy}, 
where even for $\xi_\tau$ we had to give up an $\epsilon$ 
of regularity to achieve the convergence $\xi_\tau\to\xi$ as $\tau\to0$. 

The proof of Lemma~\ref{lem:convergence} is an application 
of the Arzela-Ascoli theorem, and is deferred to Section~\ref{sec:proofs_continuity}. 
Another ingredient for the proof of Theorem~\ref{thm2} is the following lemma, 
which for $\tau>0$ establishes continuity of the map 
$\xi\mapsto(\Pi^{(\tau)},\Gtau)$ in a suitable (annealed) topology.

\begin{lemma}\label{lem:continuity}
Assume that $\E$ satisfies Assumption~\ref{ass}. 
Fix $\tau>0$, $p<\infty$ and a multi-index $\beta$. 
Then there exists $\epsilon>0$ small enough and $p'<\infty$ large enough, 
such that the map $\xi\mapsto (\Pi^{(\tau)}_\beta,\Gtau_\beta)$ is continuous, 
provided the space of $\xi$ is equipped with the norm 
$\|\cdot\|_{\alpha-2,\epsilon}$ defined 
in \eqref{annealed_xi} with integrability exponent $p'$, 
and the space of $(\Pi^{(\tau)}_\beta,\Gtau_\beta)$ is equipped with the norm 
\begin{equation}\label{annealed_model}
\|(\Pi,\G)\|_\beta 
:= \|\Pi\|_\beta 
+ [\Pi]_\beta
+ \|\G\|_\beta \, ,
\end{equation}
where
\begin{align}
\|\Pi\|_\beta
&:= \sup_{x\neq y} \frac{\E^\frac{1}{p}|\Pi_{x\beta}(y)|^p}{|x-y|^{|\beta|-\epsilon}(1+|x|+|y|)^{2\epsilon}} \, , \label{normPi} \\
[\Pi]_\beta
&:= \sup_{x,y\neq z}\sup_{|\n|=2} \frac{\E^\frac{1}{p}|\partial^\n\Pi_{x\beta}(y)-\partial^\n\Pi_{x\beta}(z)|^p}{|y\-z|^{\alpha-\epsilon} 
(1\+|x\-y|\+|x\-z|)^{|\beta|-\alpha} 
(1\+|x|\+|y|\+|z|)^{2\epsilon}} \, ,\hspace{5ex} \label{normPi2} \\
\|\G\|_\beta
&:= \sup_{x\neq y} \sup_{\gamma\prec\beta} 
\frac{\E^\frac{1}{p}|(\G_{xy})_\beta^\gamma|^p}{|x-y|^{|\beta|-|\gamma|-\epsilon}(1+|x|+|y|)^{2\epsilon}} \, . \label{normGamma} 
\end{align}
\end{lemma}

One should expect Lemma~\ref{lem:continuity} to be true, 
since on the one hand, \eqref{normPi} and \eqref{normGamma}, 
which analogously to \eqref{Pi-} imply 
$\E^\frac{1}{p}|\Pi^-_{x\beta t}(y)|^p 
\leq (\sqrt[4]{t})^{\alpha-2-\epsilon} 
(\sqrt[4]{t}+|x-y|)^{|\beta|-\alpha} 
(1+\sqrt[4]{t}+|x|+|y|)^{2\epsilon}$, 
yield good enough estimates on small and large scales 
in order for the integral in the solution formula \eqref{representation} 
to converge at $t=0$ and $t=\infty$; 
on the other hand, 
\eqref{normPi2} yields enough small scale regularity 
in order for the products $\Pi^k\Delta\Pi$ in \eqref{defPi-components} of 
$\Pi^{-}$ to be well defined. 

We give a quantitative proof of Lemma~\ref{lem:continuity} 
establishing even H\"older continuity in Section~\ref{sec:proofs_continuity}, 
the arguments of which are an adaptation of 
the arguments of the proof of \cite[Remark~2.3]{LOTT21}, 
which establishes for $|\n|=2$
\begin{align}
|c^{(\tau)}_\beta| 
&\lesssim (\sqrt[4]{\tau})^{|\beta|-2} \, , \label{c_tau} \\
\E^\frac{1}{p}| \partial^\n \Pi^{(\tau)}_{x\beta}(y)|^p 
&\lesssim (\sqrt[4]{\tau})^{\alpha-2}(\sqrt[4]{\tau}+|x-y|)^{|\beta|-\alpha} \, , 
\label{Pi_tau} \\
\E^\frac{1}{p}| \partial^\n \Pi^{(\tau)}_{x\beta}(y)
- \partial^\n \Pi^{(\tau)}_{x\beta}(z)|^p 
&\lesssim (\sqrt[4]{\tau})^{-2} |y\-z|^\alpha 
(\sqrt[4]{\tau}\+|x\-y|\+|x\-z|)^{|\beta|-\alpha} \, . \label{Pi_Holder} 
\end{align}
One just has to ensure that the $\epsilon$-loss in the exponents 
does not cause any problems. 

With these two preliminary lemmas, 
we can give the proof of Theorem~\ref{thm2}.

\begin{proof}[Proof of Theorem~\ref{thm2}]
Let $(\E_n)_{n\in\N}$ be a sequence of ensembles satisfying Assumption~\ref{ass} uniformly in $n$, 
that converges weakly to an ensemble $\E$ which then automatically satisfies Assumption~\ref{ass}. 
It is convenient for the proof to work with random variables on a common probability space. 
We therefore appeal to Skorohod's representation theorem 
to obtain random variables on a common probability space, 
which we denote by $\xi_n,\xi$, 
with their corresponding law given by the ensembles $\E_n,\E$, 
and such that up to a subsequence that we do not relabel $\xi_n\to\xi$ almost surely on the space of tempered distributions. 
A little bit of care has to be taken here, since we apply Skorohod's representation theorem 
on the space of tempered distributions which is not a metric space.\footnote{
While this could be avoided, it has the advantage that the proofs of the three cases convergences in law resp.~in probability resp.~in $L^p$ are very similar.} 
A suitable version of this theorem can be found in \cite[Example~5.1.I]{BBK05}, 
and requires to check that the $\xi_n$ are tight and the laws of $\xi_n,\xi$ are Radon measures.
The former follows from the uniform moment bound 
$\E^\frac{1}{p}|(\xi_n)_t(x)|^p\lesssim (\sqrt[4]{t})^{\alpha-2}$, 
see \eqref{Pi-}, 
which is a consequence of the spectral gap assumption.
The latter is a consequence of the following: by Kolmogorov's criterion and again by the moment bound, 
the law of $\xi_n$ concentrates on a compact subset of $\S'$, 
namely a weighted and sufficiently negative H\"older space; 
as a Borel probability measure supported on a compact set, it is therefore a Radon measure. 

Recall from Lemma~\ref{lem:convergence}, that the almost sure convergence 
implies $\xi_n\to\xi$ with respect to $\|\cdot\|_{\alpha-2,\epsilon}$ 
for every $\epsilon>0$. 
Furthermore, the map $\xi\mapsto(\Pi^{(\tau)}_\beta,\Gtau_\beta)$ is continuous for fixed $\tau>0$ by Lemma~\ref{lem:continuity}, 
provided the space of $\xi$ is equipped with $\|\cdot\|_{\alpha-2,\epsilon}$ for $\epsilon>0$ small enough, 
and the space of $(\Pi^{(\tau)}_\beta,\Gtau_\beta)$ is equipped with $\|\cdot\|_\beta$ defined in \eqref{annealed_model}. 
Note, that the norm $\|\cdot\|_\beta$ satisfies by definition 
\begin{align}
\E^\frac{1}{p}|\Pi_{x\beta}(y)|^p 
&\leq |x-y|^{|\beta|-\epsilon} (1+|x|+|y|)^{2\epsilon} \, 
\|(\Pi,\G)\|_\beta \, , \label{eh29}\\
\E^\frac{1}{p}|(\G_{xy})_\beta^\gamma|^p 
&\leq |x-y|^{|\beta|-|\gamma|-\epsilon} (1+|x|+|y|)^{2\epsilon} \, 
\|(\Pi,\G)\|_\beta \, .
\end{align}
These preliminary remarks allow for the conclusion as we argue now.
Denoting the models associated to $\xi$, $\xi_\tau$, $\xi_n$ and $(\xi_n)_\tau$ by 
$\Pi$, $\Pi^{(\tau)}$, $\Pi^{(n)}$ and $\Pi^{(n)(\tau)}$, respectively, we have 
\begin{align}
\E^\frac{1}{p}| \Pi^{(n)}_{x\beta}(y) - \Pi_{x\beta}(y)|^p 
&\leq \E^\frac{1}{p}| \Pi^{(n)}_{x\beta}(y) - \Pi^{(n)(\tau)}_{x\beta}(y)|^p \\
&\,+ \E^\frac{1}{p}| \Pi^{(n)(\tau)}_{x\beta}(y) - \Pi^{(\tau)}_{x\beta}(y)|^p 
+ \E^\frac{1}{p}| \Pi^{(\tau)}_{x\beta}(y) - \Pi_{x\beta}(y)|^p \, .
\end{align}
Since the Cauchy property \eqref{Pi_cauchy} was uniform in the class of ensembles satisfying Assumption~\ref{ass}, 
for given $\eta>0$ there exists $\tau>0$ such that the first and the third right-hand side contribution are bounded by 
$\eta \, |x-y|^{|\beta|-\epsilon}$, uniformly in $n$. 
The second right-hand side term is by \eqref{eh29}
bounded by $|x-y|^{|\beta|-\epsilon}(1+|x|+|y|)^{2\epsilon} \|(\Pi^{(n)(\tau)}-\Pi^{(\tau)}, \Gamma^{*(n)(\tau)}-\Gtau)\|_\beta$, 
which for $n$ large enough and fixed $\tau>0$ is dominated by $\eta\,|x-y|^{|\beta|-\epsilon}(1+|x|+|y|)^{2\epsilon}$ due to continuity. 
Altogether we obtain
\begin{equation}
\lim_{n\to\infty} \sup_{x\neq y} 
|x-y|^{-|\beta|+\epsilon} (1+|x|+|y|)^{-2\epsilon} \, \E^\frac{1}{p}| \Pi^{(n)}_{x\beta}(y) - \Pi_{x\beta}(y)|^p = 0 \, . 
\end{equation}
The same game can be played with $\G$, to obtain
\begin{equation}
\lim_{n\to\infty} \sup_{x\neq y} |x-y|^{-|\beta|+|\gamma|+\epsilon} (1+|x|+|y|)^{-2\epsilon} \, 
\E^\frac{1}{p}| (\Gamma^{*(n)}_{xy})_\beta^\gamma - (\G_{xy})_\beta^\gamma|^p = 0 \, . 
\end{equation}
We turn this annealed convergence by Kolmogorov's criterion into a quenched convergence result, 
possibly along a further subsequence. 
A suitable version of Kolmogorov's criterion can be found in \cite[Proposition~B.3]{HS23}, 
which we apply here with $F_x=|z|^{-|\beta|+\epsilon}(\Pi^{(n)}_{x\beta}(x+z)-\Pi_{x\beta}(x+z))$, 
and where we note that the necessary continuity follows from Lemma~\ref{lem:consequence}~iii).
This yields for $x\in\R^{1+d}$ and $\kappa>\epsilon$
\begin{equation}
\E^\frac{1}{p}\big| \sup_{x'\neq y'\in B_1(x)} |x'-y'|^{-|\beta|+\kappa} \big(\Pi^{(n)}_{x'\beta}(y') - \Pi_{x'\beta}(y') 
\big) \big|^p \lesssim (1+|x|)^{2\epsilon} o_n(1) \, . 
\end{equation}
Noting that the supremum over $x\in\R^{1+d}$ in the definition \eqref{normModelPi} of $\vertiii{\Pi}_\beta$ 
can be replaced by $x\in (1+d)^{-1/2}\Z^{1+d}$, 
and using $\E(\sup_{x\in\Z^{1+d}}|f(x)|)^p\leq\sum_{x\in\Z^{1+d}}\E|f(x)|^p$, 
which is a consequence of the obvious 
$\sup_{x\in\Z^{1+d}} |f(x)|\leq (\sum_{x\in\Z^{1+d}} |f(x)|^p)^{1/p}$, we get
\begin{equation}
\E\vertiii{\Pi^{(n)}-\Pi}^p_\beta 
\lesssim \sum_{x\in(1+d)^{-1/2}\Z^{1+d}} 
\frac{(1+|x|)^{p2\epsilon}}{(1+|x|)^{p2\kappa}} \, o_n(1) \, .
\end{equation}
By $\kappa>\epsilon$ we can choose $p<\infty$ large enough 
such that this series is finite. 
The analogous result can be obtained for $\G$, hence 
\begin{align}
\lim_{n\to\infty} \E^\frac{1}{p} 
\vertiii{\Pi^{(n)}-\Pi}^p_\beta&=0 \, , \\
\lim_{n\to\infty} \E^\frac{1}{p} 
\vertiii{\Gamma^{*(n)} - \G }^p_\beta &= 0 \, ,
\end{align}
which in turn yields the desired convergence in law. 
That we had to pass to a subsequence for Skorohod's theorem is not an issue, 
since for any subsequence of $(\Pi^{(n)},\Gamma^{*(n)})$ we can start the proof again 
with the corresponding subsequence of $\E_n$, 
and obtain convergence in law of a further subsequence of $(\Pi^{(n)},\Gamma^{*(n)})$ to $(\Pi,\G)$, 
establishing convergence in law of the whole sequence $(\Pi^{(n)},\Gamma^{*(n)})$. 

To argue in favor of the same statement, with convergence in law 
replaced by convergence in probability, we proceed similarly. 
Instead of appealing to Skorohod's representation theorem, 
we pass to a subsequence to obtain almost sure convergence of $\xi_{n_k}$ to $\xi$ in the space of tempered distributions. 
From here, we follow the above proof to obtain $L^p$ convergence of 
$(\Pi^{(n_k)},\Gamma^{*(n_k)})$ to $(\Pi,\G)$, 
which implies almost sure convergence of a further subsequence. 
For any other subsequence of $(\Pi^{(n)},\Gamma^{*(n)})$ we 
can repeat this argument with the corresponding subsequence of $\xi_n$, 
which yields almost sure convergence of a further subsequence of $(\Pi^{(n)},\Gamma^{*(n)})$ to $(\Pi,\G)$. 
By the subsequence criterion \cite[Lemma~4.2]{Kal02}, 
this implies convergence in probability of the whole sequence 
$(\Pi^{(n)},\Gamma^{*(n)})$ to $(\Pi,\G)$. 

Assuming $L^p$ convergence for all $p<\infty$, 
we infer directly from Lemma~\ref{lem:convergence} 
the convergence $\xi_n\to\xi$ with respect to the annealed norm defined in \eqref{annealed_xi}, 
and the desired $L^p$ convergence of $(\Pi^{(n)},\Gamma^{*(n)})$ follows as above. 
\end{proof}

\subsection{Proof of Lemma~\ref{annealed_xi} and Lemma~\ref{lem:continuity}}\label{sec:proofs_continuity}

\begin{proof}[Proof of Lemma~\ref{annealed_xi}]
Let $\eta>0$, we will show that $\lim_{n\to\infty}\|\xi_n-\xi\|_{\alpha-2,\epsilon}<\eta$. 
Recall from \eqref{Pi-} that the spectral gap inequality implies $\E^\frac{1}{p}|\xi_t(x)|^p \lesssim (\sqrt[4]{t})^{\alpha-2}$, 
where the implicit constant depends on $p$, but not on $t>0$ or $x\in\R^{1+d}$. 
The same holds by assumption for $\xi_n$, where the implicit constant does not depend on $n$ either. 
Therefore, for $f_n:(0,\infty)\times\R^{1+d}\to\R$ defined by 
\begin{equation}
f_n(t,x):=(\sqrt[4]{t})^{2-\alpha+\epsilon}(1+\sqrt[4]{t}+|x|)^{-2\epsilon} \, 
\E^\frac{1}{p}|(\xi_n)_t(x)-\xi_t(x)|^p \, , 
\end{equation}
we obtain by the triangle inequality
\begin{equation}\label{fn_bounded}
f_n(t,x) \lesssim 
\frac{(\sqrt[4]{t})^\epsilon}{(1+\sqrt[4]{t}+|x|)^{2\epsilon}} \, .
\end{equation}
Hence we can find a compact set $K\subset(0,\infty)\times\R^{1+d}$ 
such that the supremum of $f_n$ on the complement of $K$ is bounded by $\eta$, uniformly in $n$. 
Notice that $f_n$ is chosen such that 
$\|\xi_n-\xi\|_{\alpha-2,\epsilon}=\sup_{t,x}f_n(t,x)$. 
To conclude, it therefore remains to establish 
\begin{equation}
\lim_{n\to\infty} \sup_{(t,x)\in K} f_n(t,x) = 0 \, .
\end{equation}
This in turn, is a consequence of Arzel\`a Ascoli, 
which yields a uniformly convergent subsequence $f_{n_k}$. 
If $\xi_n\to\xi$ almost surely on $\S'$, 
then the pointwise limit of $f_n(t,x)$ is $0$ by the dominated convergence theorem. 
If $\xi_n\to\xi$ w.r.t.~$\E^\frac{1}{p}|\cdot|^p$ on $\S'$, 
then again the pointwise limit of $f_n(t,x)$ is $0$.
In either case, we conclude that the uniform limit of $f_{n_k}$ has to be $0$, 
and as a consequence the whole sequence $f_n$ converges uniformly to $0$. 
It remains to check the assumptions of the Arzel\`a Ascoli theorem. 
By \eqref{fn_bounded}, the sequence $f_n$ is uniformly bounded on the compact set $K$.
For equicontinuity, we first note that by the triangle inequality
\begin{align}
|f_n(t,x)-f_n(t',x')| &\leq
\E^\frac{1}{p}\big| (\sqrt[4]{t})^{2-\alpha+\epsilon}(1+\sqrt[4]{t}+|x|)^{-2\epsilon} (\xi_n)_t(x) \\
&\, - (\sqrt[4]{t'})^{2-\alpha+\epsilon}(1+\sqrt[4]{t'}+|x'|)^{-2\epsilon} (\xi_n)_{t'}(x') \big|^p \\
&\, + \E^\frac{1}{p}\big| (\sqrt[4]{t})^{2-\alpha+\epsilon}(1+\sqrt[4]{t}+|x|)^{-2\epsilon} \xi_t(x) \\
&\, - (\sqrt[4]{t'})^{2-\alpha+\epsilon}(1+\sqrt[4]{t'}+|x'|)^{-2\epsilon} \xi_{t'}(x') \big|^p \, .
\end{align}
The first two right-hand side terms are further bounded by 
\begin{align}
&\big|(\sqrt[4]{t})^{2-\alpha+\epsilon}(1+\sqrt[4]{t}+|x|)^{-2\epsilon}
-(\sqrt[4]{t'})^{2-\alpha+\epsilon}(1+\sqrt[4]{t'}+|x'|)^{-2\epsilon} \big|
\, \E^\frac{1}{p}| (\xi_n)_t(x)|^p \\
& + (\sqrt[4]{t'})^{2-\alpha+\epsilon}(1+\sqrt[4]{t'}+|x'|)^{-2\epsilon} \, 
\E^\frac{1}{p}\big| (\xi_n)_t(x)-(\xi_n)_{t'}(x') \big|^p \, .
\end{align}
The first term satisfies the desired continuity, 
since $t\mapsto(\sqrt[4]{t})^{2-\alpha+\epsilon}(1+\sqrt[4]{t}+|x|)^{-2\epsilon}$ is (uniformly) continuous on $K$, 
and since $\E^\frac{1}{p}|(\xi_n)_t(x)|^p\lesssim (\sqrt[4]{t})^{\alpha-2}$ is bounded uniformly in $n$ and $(t,x)\in K$. 
For the second term, we note that $(\sqrt[4]{t'})^{2-\alpha+\epsilon}(1+\sqrt[4]{t'}+|x'|)^{-2\epsilon}$ is uniformly bounded for $(t',x')\in K$, 
and it remains to establish equicontinuity of 
$\E^\frac{1}{p} | (\xi_n)_t(x)-(\xi_n)_{t'}(x') |^p$. 
This is a consequence of 
\begin{equation}
\E^\frac{1}{p} | (\xi_n)_t(x)-(\xi_n)_{t}(x') |^p 
\lesssim (\sqrt[4]{t})^{-2} |x-x'|^\alpha \, ,
\end{equation}
which was established in \cite[(6.3)]{LOTT21},
together with
\begin{equation}
\E^\frac{1}{p} | (\xi_n)_t(x')-(\xi_n)_{t'}(x') |^p 
\leq |t-t'| \, \E^\frac{1}{p}| (\partial_0^2-\Delta^2)(\xi_n)_{\bar t}(x')|^p \lesssim |t-t'| \, (\sqrt[4]{\bar t}\,)^{\alpha-6} 
\end{equation}
for some $\bar t$ between $t$ and $t'$, 
which is a consequence of the mean value theorem, the semigroup property \eqref{semigroup} and the moment bound \eqref{momentbound}. 
The terms involving $\xi$ instead of $\xi_n$ can be treated in the same way. 
\end{proof}

\begin{proof}[Proof of Lemma~\ref{lem:continuity}]
Let $(\Pi^{(\tau)},\Gtau)$ and $(\widetilde\Pi^{(\tau)},\tGtau)$ be the models associated to $\xi$ and $\tilde\xi$. 
We will show that $\|(\Pi^{(\tau)},\Gtau)-(\widetilde\Pi^{(\tau)},\tGtau)\|_\beta$ 
is arbitrarily small, provided $\|\xi-\tilde\xi\|_{\alpha-2,\epsilon}$ is sufficiently small. 
Since the only norm we put on $\xi,\tilde\xi$ is $\|\cdot\|_{\alpha-2,\epsilon}$, 
we will for the ease of notation simply write $\|\xi\|$ instead. 
The proof proceeds inductively, for which it is convenient to introduce the norms
\begin{align}
\|\Pi\|^-_\beta 
&:= \sup_{t>0}\sup_{x,y} \frac{\E^\frac{1}{p}|\Pi^-_{x\beta t}(y)|^p}{(\sqrt[4]{t})^{\alpha-2-\epsilon} 
(\sqrt[4]{t}+|x-y|)^{|\beta|-\alpha} (1+\sqrt[4]{t}+|x|+|y|)^{2\epsilon}} \, , \\
[\Pi]^-_{\beta} 
&:= \sup_{x,y\neq z} \frac{\E^\frac{1}{p}|\Pi^-_{x\beta}(y)-\Pi^-_{x\beta}(z)|^p}{|y-z|^{\alpha-\epsilon} 
(1+|x-y|+|x-z|)^{|\beta|-\alpha} 
(1+|x|+|y|+|z|)^{2\epsilon}} \, .
\end{align}
With this notation at hand, 
we will show in Step~1 that the linear map $\xi\mapsto\Pi^{-(\tau)}_0=\xi_\tau$ is continuous, 
i.e.~we establish 
\begin{equation}\label{step1}
\|\Pi^{(\tau)}\|^-_0 
+ [\Pi^{(\tau)}]^-_{0}
\lesssim_\tau \|\xi\| \, , \tag{S1}
\end{equation}
where we want to emphasize that the implicit constant in $\lesssim_\tau$ does depend on $\tau$ (and blows up as $\tau\to0$). 
In Step~2 we establish one part of the continuity of the map $\Pi^{-(\tau)}_\beta\mapsto\Pi^{(\tau)}_\beta$, 
namely 
\begin{equation}\label{step2}
\|\Pi^{(\tau)} - \widetilde\Pi^{(\tau)}\|_\beta 
\lesssim \|\Pi^{(\tau)}-\widetilde\Pi^{(\tau)}\|^-_\beta \, . \tag{S2}
\end{equation}
In Step~3 we will prove that 
\begin{equation}\label{step3}
\|\Gtau-\tGtau\|_\beta\lesssim \|\Pi^{(\tau)}-\widetilde\Pi^{(\tau)}\|_{\preceq\beta} \, , \tag{S3}
\end{equation}
where $\|\cdot\|_{\preceq\beta}:=\sum_{\beta'\preceq\beta}\|\cdot\|_{\beta'}$. 
We will use this in Step~4 to show 
\begin{equation}\label{step4}
[\Pi^{(\tau)}-\widetilde\Pi^{(\tau)} ]_{\beta} 
\lesssim_\tau \|\Pi^{(\tau)}-\widetilde\Pi^{(\tau)}\|^-_{\preceq\beta} 
+ [\Pi^{(\tau)}-\widetilde\Pi^{(\tau)}]^-_{\preceq\beta} \, , \tag{S4} 
\end{equation}
and in Step~5 we establish for $\beta\neq0$ 
\begin{equation}\label{step5}
[\Pi^{(\tau)}-\widetilde\Pi^{(\tau)}]^-_{\beta} 
\lesssim_\tau \|\Pi^{(\tau)}-\widetilde\Pi^{(\tau)}\|_{\prec\beta} 
+ [\Pi^{(\tau)}-\widetilde\Pi^{(\tau)}]_{\prec\beta} 
+ \|\Pi^{(\tau)}-\widetilde\Pi^{(\tau)}\|^-_{\prec\beta} \, . \tag{S5} 
\end{equation}
Finally, in Step~6 we show for $\beta\neq0$ that 
\begin{align}
\|\Pi^{(\tau)}\-\widetilde\Pi^{(\tau)}\|^-_\beta 
\lesssim_\tau \big(
&\|\Pi^{(\tau)}\-\widetilde\Pi^{(\tau)}\|_{\prec\beta} 
\+ [\Pi^{(\tau)}\-\widetilde\Pi^{(\tau)}]_{\prec\beta} 
\+ \|\Pi^{(\tau)}\-\widetilde\Pi^{(\tau)}\|^-_{\prec\beta}
\big)^{\epsilon/(2+\epsilon)} \\
+ &\|\Pi^{(\tau)}\-\widetilde\Pi^{(\tau)}\|_{\prec\beta} 
\+ [\Pi^{(\tau)}\-\widetilde\Pi^{(\tau)}]_{\prec\beta} 
\+ \|\Pi^{(\tau)}\-\widetilde\Pi^{(\tau)}\|^-_{\prec\beta} \, ,
\tag{S6} \label{step6}
\end{align}
which finishes the argument. 

Note that the reason for $p'$ in the definition \eqref{annealed_xi} 
of $\|\xi\|$ 
being different from $p$ in the definition \eqref{annealed_model} 
of $\|(\Pi,\G)\|_\beta$, 
is that along the way we appeal to H\"older's inequality. 
Due to the triangularity \eqref{triProduct_prec}, \eqref{triGamma_prec}, \eqref{tridGamma_prec} and the coercitivity of $\prec$, 
we will do so a finite number of times to reach the multi-index $\beta$, 
and hence it is possible to choose $p'<\infty$ large enough. 
We will be cavalier in the following and not distinguish 
in the notation between $p'$ and $p$, 
and implicitly always assume we made the appropriate change in exponents. 
For $\epsilon$, the following choice will turn out to be suitable. 
We choose $\epsilon>0$ satisfying \eqref{epsilon1} and  
\begin{equation}\label{epsilon3}
\epsilon<n-|\beta'|, \quad
\text{for all } \beta' \text{ and } n\in\N \text{ satisfying } |\beta'|\leq|\beta| \text{ and } |\beta'|<n \, ,
\end{equation}
which is an admissible choice since the set of homogeneities is locally finite. 

{\bf Step 1 \rm (Proof of \eqref{step1})\bf.} 
By the definition of $\|\xi\|$ in \eqref{annealed_xi}, 
we see that $\E^\frac{1}{p}| \xi_{\tau+t}(x) |^p \leq (\sqrt[4]{t+\tau})^{\alpha-2-\epsilon} (1+\sqrt[4]{t+\tau}+|x|)^{2\epsilon} \|\xi\|$, 
which since $\alpha-2-\epsilon<0$ and $\tau\lesssim 1$ yields
$\|\Pi^{(\tau)}\|^-_0\lesssim\|\xi\|$. 
To argue for $[\Pi^{(\tau)}]^-_{0}\lesssim_\tau \|\xi\|$ we use the triangle inequality to obtain 
\begin{align}
\E^\frac{1}{p}| \xi_\tau(y)-\xi_\tau(z)|^p
&\leq \E^\frac{1}{p}| \xi_\tau(y)-\xi_{\tau+t}(y)|^p \\
&\,+ \E^\frac{1}{p}| \xi_{\tau+t}(y)-\xi_{\tau+t}(z)|^p 
+\E^\frac{1}{p}| \xi_{\tau+t}(z)-\xi_\tau(z)|^p \, .
\end{align}
The first right-hand side term can be rewritten as 
$\E^\frac{1}{p}|\int_0^t ds\,\partial_s\xi_{\tau+s}(y)|^p$, 
which by the defining property \eqref{psi} of $\psi$, 
the moment bound \eqref{momentbound} and the definition of $\|\xi\|$ can be estimated by  
\begin{equation}
\int_0^t ds\, (\sqrt[4]{\tau+s})^{\alpha-2-\epsilon-4}(1+\sqrt[4]{\tau+s}+|y|)^{2\epsilon} \|\xi\| \, .
\end{equation}
Since $\alpha-\epsilon>0$ by \eqref{epsilon1} and $\tau\lesssim1$, this is further bounded by 
\begin{equation}
(\sqrt[4]{\tau})^{-2} (\sqrt[4]{t})^{\alpha-\epsilon} (1+\sqrt[4]{t}+|y|)^{2\epsilon}\|\xi\| \, .
\end{equation}
The same bound holds for the third right-hand side term with $y$ replaced by $z$. 
The second right-hand side term we bound using the mean value theorem and the definition of $\|\xi\|$ by
\begin{equation}
\big(|y-z| (\sqrt[4]{\tau+t})^{\alpha-2-\epsilon-1}
+ |y-z|^2 (\sqrt[4]{\tau+t})^{\alpha-2-\epsilon-2}\big) (1+\sqrt[4]{\tau+t}+|y|+|z|)^{2\epsilon} \|\xi\| \, , 
\end{equation}
which since $\alpha-\epsilon-1<0$ and $\tau\lesssim1$ is further estimated by 
\begin{equation}
(\sqrt[4]{\tau})^{-2} (|y-z| (\sqrt[4]{t})^{\alpha-\epsilon-1}
+ |y-z|^2 (\sqrt[4]{t})^{\alpha-\epsilon-2}) 
(1+\sqrt[4]{t}+|y|+|z|)^{2\epsilon} \|\xi\| \, .
\end{equation}
Choosing $t=|y-z|^4$, we obtain altogether 
\begin{equation}
\E^\frac{1}{p}| \xi_\tau(y)-\xi_\tau(z)|^p
\lesssim (\sqrt[4]{\tau})^{-2} |y-z|^{\alpha-\epsilon} (1+|y|+|z|)^{2\epsilon}\|\xi\| \, ,
\end{equation}
which yields the desired $[\Pi^{(\tau)}]^-_{0}\lesssim_\tau \|\xi\|$.

{\bf Step 2 \rm (Proof of \eqref{step2})\bf.}
We use the representation formula \eqref{representation} of $\Pi^{(\tau)}$. 
By linearity it is enough to show 
$\|\Pi^{(\tau)} \|_\beta\leq\|\Pi^{(\tau)}\|^-_\beta$. 
The proof is analogous to the one of Lemma~\ref{int1}. 
In the near field $0\leq t\leq|x-y|^4$ we split the contributions from $1$ and ${\rm T}_x^{|\beta|}$, and bound the former by the definition of $\|\Pi^{(\tau)}\|_\beta^-$ by 
\begin{equation}
\int_0^{|x-y|^4} dt\, (\sqrt[4]{t})^{\alpha-2-\epsilon-2} 
(\sqrt[4]{t}+|x-y|)^{|\beta|-\alpha} 
(1+\sqrt[4]{t}+|x|+|y|)^{2\epsilon} \|\Pi^{(\tau)}\|_\beta^- \, ,
\end{equation}
which since $\alpha-\epsilon>0$ is bounded by 
$|x-y|^{|\beta|-\epsilon}(1+|x|+|y|)^{2\epsilon} \|\Pi^{(\tau)}\|^-_\beta$. 
The contribution from ${\rm T}_x^{|\beta|}$ is again by the definition of $\|\Pi^{(\tau)}\|_\beta^-$ bounded by 
\begin{equation}
\int_0^{|x-y|^4} dt\, \sum_{|\n|<|\beta|} (\sqrt[4]{t})^{\alpha-2-\epsilon-2-|\n|} (\sqrt[4]{t})^{|\beta|-\alpha}
(1+\sqrt[4]{t}+|x|)^{2\epsilon}  \|\Pi^{(\tau)}\|_\beta^- |x-y|^{|\n|} \, . 
\end{equation}
By the choice of $\epsilon$ in \eqref{epsilon1}, $|\beta|>|\n|$ implies $|\beta|-|\n|-\epsilon>0$, hence also this contribution can be further estimated by $|x-y|^{|\beta|-\epsilon}(1+|x|+|y|)^{2\epsilon} \|\Pi^{(\tau)}\|^-_\beta$. 
In the far field $|x-y|^4\leq t<\infty$ we use the Taylor remainder to estimate this contribution by 
\begin{equation}
\int_{|x-y|^4}^\infty dt \sum_{\substack{|\n|\geq|\beta|\\n_0+\dots+n_d<|\beta|+1}} 
(\sqrt[4]{t})^{\alpha-2-\epsilon-2-|\n|} (\sqrt[4]{t})^{|\beta|-\alpha}
(1+\sqrt[4]{t}+|x|)^{2\epsilon} \|\Pi^{(\tau)}\|_\beta^- |x-y|^{|\n|} \, . 
\end{equation}
Since $\beta$ is not purely polynomial, 
$|\beta|\leq|\n|$ implies $|\beta|<|\n|$, 
and the choice of $\epsilon$ in \eqref{epsilon3} 
implies $|\beta|-|\n|+\epsilon<0$. 
Hence also this contribution is further estimated by 
$|x-y|^{|\beta|-\epsilon}(1+|x|+|y|)^{2\epsilon} \|\Pi^{(\tau)}\|^-_\beta$, 
which finishes the proof of $\|\Pi^{(\tau)} \|_\beta\lesssim \|\Pi^{(\tau)}\|^-_\beta$. 

{\bf Step 3 \rm (Proof of \eqref{step3})\bf.}
As a first ingredient we prove 
\begin{equation}\label{eh30}
\| (\Gtau-\tGtau)P\|_\beta
\lesssim \|\Gtau-\tGtau\|_{\prec\beta} \, ,
\end{equation}
for which we have to show that for $\gamma$ not purely polynomial
\begin{equation}
\E^\frac{1}{p} | (\Gtau_{xy})_\beta^\gamma - (\tGtau_{xy})_\beta^\gamma|^p 
\lesssim |x-y|^{|\beta|-|\gamma|-\epsilon} (1+|x|+|y|)^{2\epsilon} \|\Gtau-\tGtau\|_{\prec\beta} \, . 
\end{equation}
The proof is analogous to the one of Lemma~\ref{alg1}. 
After appealing to the exponential formula \eqref{exponentialformula} and telescoping \eqref{telescope}, 
we observe that by $\pi^{(\n)(\tau)}_{xy\beta'} = (\Gtau_{xy}-\id)_{\beta'}^{e_\n}$, see \eqref{Gamma_zn}, 
we have to estimate linear combinations of terms of the form 
\begin{equation}
(\tGtau_{xy}-\id)_{\beta_1}^{e_{\n_1}} \cdots (\tGtau_{xy}-\id)_{\beta_{i-1}}^{e_{\n_{i-1}}}
(\Gtau_{xy}-\tGtau_{xy})_{\beta_i}^{e_{\n_i}}
(\Gtau_{xy}-\id)_{\beta_{i+1}}^{e_{\n_{i+1}}} \cdots (\Gtau_{xy}-\id)_{\beta_k}^{e_{\n_k}}
\end{equation}
for $\beta_1,\dots,\beta_k\prec\beta$. 
By H\"older's inequality we can use on the first $i-1$ and the last $k-1$ terms the 
bound \eqref{eh14} on $\Gtau$, 
which was uniform in the class of ensembles satisfying Assumption~\ref{ass}, 
while on the increment we appeal to the definition of $\|\Gtau\|_{\beta_i}$ to obtain an estimate by 
\begin{equation}
|x-y|^{|\beta_1|+\cdots+|\beta_k|-|\n_1|-\cdots-|\n_k|-\epsilon} (1+|x|+|y|)^{2\epsilon} \|\Gtau-\tGtau\|_{\beta_i} \, . 
\end{equation}
As in the proof of Lemma~\ref{alg1}, the exponent of $|x-y|$ equals $|\beta|-|\gamma|-\epsilon$, 
so that we obtain \eqref{eh30}. 

We turn to purely polynomial multi-indices $\gamma$, where the argument is based on \cite[(4.16)]{LOTT21}, which reads
\begin{equation}
\sum_\n (\Gtau_{xy}-\id)_\beta^{e_\n} (z-y)^\n 
= \Pi^{(\tau)}_x(z) - \Pi^{(\tau)}_y(z) - (\Gtau_{xy}-\id) P \Pi^{(\tau)}_y(z) \, . 
\end{equation}
As in the proof of \cite[Proposition~4.4]{LOTT21}, by using in addition 
telescoping \eqref{telescope}, 
we can argue in favor of 
\begin{align}
\E^\frac{1}{p}|(\Gtau_{xy}-\tGtau_{xy})_\beta^{e_\n} |^p
&\lesssim |x-y|^{|\beta|-|\n|} (1+|x|+|y|)^{2\epsilon} \\
&\ \cdot \big( \| \Pi^{(\tau)} - \widetilde\Pi^{(\tau)} \|_{\preceq\beta} 
+ \|(\Gtau_{xy}-\tGtau_{xy})P\|_\beta\big) \, ,
\end{align}
which together with the already established \eqref{eh30} yields \eqref{step3}. 

{\bf Step 4 \rm (Proof of \eqref{step4})\bf.}
We will establish bounds of increments of $\partial_1^2\Pi^{(\tau)}_{x\beta}$, 
the remaining derivatives can be bounded similarly. 
As a first ingredient, we prove for all $\n\neq\0$
\begin{align}
&\E^\frac{1}{p}| \partial^\n \partial_1^2 \Pi^{(\tau)}_{x\beta t}(x)|^p \\ 
&\lesssim (\sqrt[4]{t})^{\alpha-\epsilon-|\n|} 
(1+\sqrt[4]{t})^{|\beta|-\alpha}
(1+\sqrt[4]{t}+|x|)^{2\epsilon} 
\big( \|\Pi^{(\tau)}\|^-_{\beta} + [\Pi^{(\tau)}]^-_{\beta} \big) \, . \label{eh31}
\end{align}
Next, we note that \eqref{Pi_Holder} yields by the moment bound \eqref{momentbound} 
\begin{align}
\E^\frac{1}{p} | \partial^\n \partial_1^2 \Pi^{(\tau)}_{x\beta t}(x)|^p
&\leq \int_{\R^{1+d}} dy\, |\partial^\n\psi_t(x-y)| \, \E^\frac{1}{p}| \partial_1^2\Pi^{(\tau)}_{x\beta}(y) 
- \partial_1^2\Pi^{(\tau)}_{x\beta}(x) |^p \\
&\lesssim_\tau \int_{\R^{1+d}} dy\, |\partial^\n\psi_t(x-y)| \, (1+|x-y|)^{|\beta|-\alpha} |x-y|^\alpha \\
&\lesssim_\tau (\sqrt[4]{t})^{\alpha-|\n|} (1+\sqrt[4]{t})^{|\beta|-\alpha} \, . \label{eh32}
\end{align}
By using $\partial_1^2\Gtau_{xy}\Pi^{(\tau)}_y = \partial_1^2\Pi^{(\tau)}_x$, see \eqref{recenter}, 
we observe 
\begin{align}
&\E^\frac{1}{p} \big| \partial^\n \partial_1^2 \Pi^{(\tau)}_{x\beta t}(y) 
- \partial^\n \partial_1^2 \widetilde\Pi^{(\tau)}_{x\beta t}(y)\big|^p  \\
&\leq \sum_{\gamma\preceq\beta} \E^\frac{1}{p}\big| (\Gtau_{xy})_\beta^\gamma \big( \partial^\n\partial_1^2\Pi_{y\gamma t}(y) 
- \partial^\n\partial_1^2\widetilde\Pi^{(\tau)}_{y\gamma t}(y)\big) \big|^p \\
&\ + \sum_{\gamma\preceq\beta} \E^\frac{1}{p}\big| \big(\Gtau_{xy}-\tGtau_{xy}\big)_\beta^\gamma \, 
\partial^\n\partial_1^2\widetilde\Pi^{(\tau)}_{y\gamma t}(y) \big|^p \, .
\end{align}
After applying H\"older's inequality, 
we appeal for the first right-hand side term to \eqref{eh14} and 
to \eqref{eh31} with $\Pi^{(\tau)}$ replaced by $\Pi^{(\tau)}-\widetilde\Pi^{(\tau)}$, 
and for the second right-hand side term to \eqref{step3} and \eqref{eh32}, 
to obtain an estimate by 
\begin{align}
&\sum_{\gamma\preceq\beta} |x-y|^{|\beta|-|\gamma|} 
(\sqrt[4]{t})^{\alpha-\epsilon-|\n|} 
(1+\sqrt[4]{t})^{|\gamma|-\alpha}
(1+\sqrt[4]{t}+|y|)^{2\epsilon} \\[-2ex] 
&\hspace{5ex} \cdot \big( \|\Pi^{(\tau)}-\widetilde\Pi^{(\tau)}\|^-_{\gamma} + [\Pi^{(\tau)}-\widetilde\Pi^{(\tau)}]^-_{\gamma} \big) \\
&+ \sum_{\gamma\preceq\beta} |x-y|^{|\beta|-|\gamma|-\epsilon}(1+|x|+|y|)^{2\epsilon} 
\|\Pi^{(\tau)}-\widetilde\Pi^{(\tau)}\|_{\preceq\beta} \,
(\sqrt[4]{t})^{\alpha-|\n|} (1+\sqrt[4]{t})^{|\gamma|-\alpha} \, ,
\end{align}
which yields
\begin{align}
&\E^\frac{1}{p} \big| \partial^\n \partial_1^2 \Pi^{(\tau)}_{x\beta t}(y) 
- \partial^\n \partial_1^2 \widetilde\Pi^{(\tau)}_{x\beta t}(y)\big|^p \\
&\lesssim_\tau (\sqrt[4]{t})^{\alpha-\epsilon-|\n|} 
(1+\sqrt[4]{t}+|x-y|)^{|\beta|-\alpha} (1+\sqrt[4]{t}+|x|+|y|)^{2\epsilon} \\ 
&\ \cdot \big( \|\Pi^{(\tau)}-\widetilde\Pi^{(\tau)}\|^-_{\preceq\beta} 
+ [\Pi^{(\tau)}-\widetilde\Pi^{(\tau)}]^-_{\preceq\beta} \big) \, .
\end{align}
As in \cite[Proof of Remark~2.3]{LOTT21}, 
this weak bound can be upgraded to a pointwise bound, yielding \eqref{step4}. 

It remains to provide the argument for \eqref{eh31}. 
For this, we note that the in Step~2 established $\|\Pi^{(\tau)} \|_\beta\lesssim \|\Pi^{(\tau)}\|^-_\beta$ 
implies by the moment bound \eqref{momentbound} 
\begin{equation}\label{eh33}
\E^\frac{1}{p}| \partial^\n \partial_1^2 \Pi^{(\tau)}_{x\beta t}(x) |^p 
\lesssim (\sqrt[4]{t})^{|\beta|-\epsilon-2-|\n|} 
(1+\sqrt[4]{t}+|x|)^{2\epsilon} \, 
\|\Pi^{(\tau)}\|^-_\beta \, . 
\end{equation}
Equipped with this, we first consider the regime $t\geq1$. 
In this case, \eqref{eh33} is a stronger estimate than \eqref{eh31}. 
In the regime $t\leq1$, we distinguish the case $|\beta|>3$ and $|\beta|<3$. 
For the former, we use again \eqref{eh33}, 
and argue again that it is stronger than \eqref{eh31}. 
To see this, we split the exponent of $\sqrt[4]{t}$ into $\alpha-\epsilon-|\n|$ and $|\beta|-\alpha-2$, 
and argue that the latter is positive: 
by $|\beta|>3$, this is a consequence of $\alpha<1$. 
We turn to the case $|\beta|<3$. 
Here, we appeal to the solution formula \eqref{representation} 
and note that the Taylor polynomial drops out due to $|\beta|<3$ and $\n\neq\0$, so that
\begin{equation}
\partial^\n\partial_1^2\Pi^{(\tau)}_{x\beta t}(x) 
= -\int_t^\infty ds\, \partial^\n\partial_1^2 (\partial_0+\Delta)
\Pi^{-(\tau)}_{x\beta s}(x) \, .
\end{equation}
The $\E^\frac{1}{p}|\cdot|^p$-norm of the far field contribution $s\geq1$ 
is by the definition of $\|\Pi^{(\tau)}\|^-_\beta$ estimated by 
\begin{equation}
\int_1^\infty ds\, (\sqrt[4]{s})^{|\beta|-2-\epsilon-|\n|-4}(1+\sqrt[4]{s}+|x|)^{2\epsilon} \|\Pi^{(\tau)}\|^-_\beta \, .
\end{equation}
Since $|\beta|<3$, we have by the choice of $\epsilon$ in \eqref{epsilon3} that $|\beta|+\epsilon<3$. 
Thus $|\beta|+\epsilon-|\n|-6 < -|\n|-3$, which by $\n\neq\0$ is $\leq-4$. 
Hence the integral converges at $s=\infty$ and is $\lesssim (1+|x|)^{2\epsilon} \|\Pi^{(\tau)}\|^-_\beta$, 
which by $t\leq1$ is trivially estimated by the right-hand side of $\eqref{eh31}$. 
In the intermediate regime of $t\leq s\leq1$ we rewrite
\begin{align}
&\int_t^1 ds\, \partial^\n\partial_1^2(\partial_0+\Delta)\Pi^{-(\tau)}_{x\beta s}(x) \\
&= \int_t^1 ds \int_{\R^{1+d}} dz \, \partial^\n\partial_1^2(\partial_0+\Delta)\psi_s(x-z)
\big( \Pi^{-(\tau)}_{x\beta}(z) - \Pi^{-(\tau)}_{x\beta}(x) \big) \, . 
\end{align}
By the definition of $[\Pi^{(\tau)}]^-_{\beta}$, 
the $\E^\frac{1}{p}|\cdot|^p$-norm of this contribution is estimated by 
\begin{align}
\int_t^1 ds \int_{\R^{1+d}} dz \, &| \partial^\n\partial_1^2(\partial_0+\Delta)\psi_s(x-z) | \\
&\cdot |x-z|^{\alpha-\epsilon} 
(1+|x-z|)^{|\beta|-\alpha}
(1+|x|+|z|)^{2\epsilon} \, [\Pi^{(\tau)}]^-_\beta \, ,
\end{align}
which by the moment bound \eqref{momentbound} is estimated by 
\begin{align}
&\int_t^1 ds \, (\sqrt[4]{s})^{\alpha-\epsilon-|\n|-4} 
(1+\sqrt[4]{s})^{|\beta|-\alpha}
(1+\sqrt[4]{s}+|x|)^{2\epsilon} \, [\Pi^{(\tau)}]^-_\beta \\
&\lesssim \big((1+|x|)^{2\epsilon} 
+ (\sqrt[4]{t})^{\alpha-\epsilon-|\n|} 
(1+\sqrt[4]{t})^{|\beta|-\alpha}
(1+\sqrt[4]{t}+|x|)^{2\epsilon} \big) \, [\Pi^{(\tau)}]^-_\beta \, .
\end{align}
Again, since $t\leq1$ and $\n\neq\0$, this is trivially estimated by the right-hand side of \eqref{eh31}. 

{\bf Step 5 \rm (Proof of \eqref{step5})\bf.}
In preparation for later, we first establish 
\begin{equation}\label{eh37}
\E^\frac{1}{p}| \Delta \Pi^{(\tau)}_{x\beta}(x) |^p 
\lesssim (1+|x|)^{2\epsilon} \big( \|\Pi\|_\beta + [\Pi]_\beta \big) \, .
\end{equation}
Indeed, using that $\psi_t$ integrates to $1$, we rewrite 
\begin{equation}
\Delta \Pi^{(\tau)}_{x\beta}(x) 
=  \Delta \Pi^{(\tau)}_{x\beta t}(x) 
+ \int_{\R^{1+d}} dz \, \psi_t(x-z) \, \big(
\Delta \Pi^{(\tau)}_{x\beta}(x) - \Delta \Pi^{(\tau)}_{x\beta}(z) \big) \, ,
\end{equation}
to see from the definitions of $\|\Pi^{(\tau)}\|_\beta$ and $[\Pi^{(\tau)}]_\beta$ that 
\begin{align}
&\E^\frac{1}{p}| \Delta \Pi^{(\tau)}_{x\beta}(x) |^p \\
&\lesssim (\sqrt[4]{t})^{|\beta|-\epsilon-2} (1+\sqrt[4]{t}+|x|)^{2\epsilon} \, \|\Pi^{(\tau)}\|_\beta \\
&\, + \int_{\R^{1+d}} dz \, |\psi_t(x-z)| 
|x-z|^{\alpha-\epsilon} 
(1+|x-z|)^{|\beta|-\alpha} 
(1+|x|+|z|)^{2\epsilon} \, [\Pi^{(\tau)}]_\beta \\
&\lesssim 
\big((\sqrt[4]{t})^{\alpha-\epsilon-2} + (\sqrt[4]{t})^{\alpha-\epsilon} \big)
(1+\sqrt[4]{t})^{|\beta|-\alpha} 
(1+\sqrt[4]{t}+|x|)^{2\epsilon} \, \big(\|\Pi^{(\tau)}\|_\beta + [\Pi^{(\tau)}]_\beta \big) \, ,
\end{align}
which for $t=1$ yields \eqref{eh37}. 
Using $\Delta\Pi^{(\tau)}_{x\beta}(y) = (\Gtau_{xy}\Delta\Pi^{(\tau)}_{y})_{\beta}(y)$, 
we obtain from the definition of $\|\Gtau\|_\beta$ together with 
\eqref{step3} and \eqref{Pi_tau}, 
and from \eqref{eh14} and \eqref{eh37} 
\begin{align}
&\E^\frac{1}{p}| \Delta \Pi^{(\tau)}_{x\beta}(y) - \Delta \widetilde\Pi^{(\tau)}_{x\beta}(y) |^p \\ 
&\leq \E^\frac{1}{p}\big| \big( (\Gtau_{xy}-\tGtau_{xy}) 
\Delta\Pi^{(\tau)}_{y} \big)_\beta(y) \big|^p 
+ \E^\frac{1}{p}\big|\big( \tGtau_{xy} (\Delta\Pi^{(\tau)}_{y} 
- \Delta\widetilde\Pi^{(\tau)}_{y})\big)_\beta(y) \big|^p \\
&\lesssim_\tau  \sum_{\gamma\prec\beta} |x-y|^{|\beta|-|\gamma|-\epsilon} 
(1+|x|+|y|)^{2\epsilon} \|\Pi^{(\tau)}-\widetilde\Pi^{(\tau)}\|_{\preceq\beta} \\
&\, +\sum_{\gamma\preceq\beta} |x-y|^{|\beta|-|\gamma|} (1+|y|)^{2\epsilon} 
\big(\|\Pi^{(\tau)}-\widetilde\Pi^{(\tau)}\|_\gamma + [\Pi^{(\tau)}-\widetilde\Pi^{(\tau)}]_\gamma \big) \, , 
\end{align}
where the first sum restricts to $\gamma\prec\beta$ since $\Gtau-\tGtau=\Gtau-\id+\id-\tGtau$ is strictly triangular. 
This implies $|\beta|>|\gamma|$, which by the choice of $\epsilon$ in \eqref{epsilon1} upgrades to 
$|\beta|-|\gamma|-\epsilon>0$. 
Thus, also using $|\cdot|\geq\alpha$, we have 
$|x-y|^{|\beta|-|\gamma|-\epsilon} 
\leq (1+|x-y|)^{|\beta|-|\gamma|-\epsilon} 
\leq (1+|x-y|)^{|\beta|-\alpha}$, 
which together with $|x-y|^{|\beta|-|\gamma|} 
\leq (1+|x-y|)^{|\beta|-\alpha}$
implies 
\begin{align}
\E^\frac{1}{p}| \Delta \Pi^{(\tau)}_{x\beta}(y) - \Delta \widetilde\Pi^{(\tau)}_{x\beta}(y) |^p 
&\lesssim_\tau  (1+|x-y|)^{|\beta|-\alpha} 
(1+|x|+|y|)^{2\epsilon} \\ 
&\ \cdot \big(\|\Pi^{(\tau)}-\widetilde\Pi^{(\tau)}\|_{\preceq\beta} 
+ [\Pi^{(\tau)}-\widetilde\Pi^{(\tau)}]_{\preceq\beta} \big) \, . 
\label{eh40}
\end{align}
Equipped with this, we show 
\begin{equation}\label{eh39}
|c^{(\tau)}_\beta - \tilde c^{(\tau)}_\beta |
\lesssim_\tau  \|\Pi^{(\tau)}-\widetilde\Pi^{(\tau)}\|_{\prec\beta} 
+ [\Pi^{(\tau)}-\widetilde\Pi^{(\tau)}]_{\prec\beta} 
+ \|\Pi^{(\tau)}-\widetilde\Pi^{(\tau)}\|^-_{\prec\beta} \, .
\end{equation}
Recall from \cite[(2.38)]{LOTT21} that $c^{(\tau)}_\beta$ is chosen 
in such a way that 
\begin{equation}\label{bphz}
\lim_{t\to\infty} \E \Pi^{-(\tau)}_{x\beta t}(x) = 0 \,.
\end{equation}
Hence, \eqref{eh39} is true for $\beta=0$ by $c^{(\tau)}_0 = 0 =  \tilde c^{(\tau)}_0$ as a consequence of the centeredness in Assumption~\ref{ass}. 
For the remaining multi-indices, we proceed by induction and assume $\eqref{eh39}_{\prec\beta}$.
Note that 
\begin{equation}
c^{(\tau)}_\beta = \underbrace{\E(\Pi^{-(\tau)}_x+c^{(\tau)} )_{\beta t}(x)}_{=: \it I} 
+ \underbrace{ \int_t^\infty ds\, \partial_s \E\Pi^{-(\tau)}_{x\beta s}(x) }_{=: \it II} \, ,
\end{equation}
where the choice $t=1$ and $x=0$ will turn out to be convenient, 
and recall from \cite[(4.19)]{LOTT21} that 
\begin{equation}
\partial_s \E\Pi^{-(\tau)}_{x\beta s}(y) 
= \int_{\R^{1+d}} dz \, (\partial_0^2-\Delta^2) \psi_{s/2}(y-z) \, 
\E\big( (\Gtau_{xz}-\id)\Pi^{-(\tau)}_{z\,s/2}\big)_\beta (z) \, .
\end{equation}
Using this, we obtain from \eqref{eh30} \& \eqref{step3}, \eqref{eh14}, 
\eqref{Pi-} and the definition of $\|\Pi^{(\tau)}\|^-_\beta$
\begin{align}
&\big| \partial_s \E (\Pi^{-(\tau)}_{x \beta}-\widetilde\Pi^{-(\tau)}_{x \beta})_{s}(y) \big| \\
&\lesssim \int_{\R^{1+d}} dz \, |(\partial_0^2-\Delta^2) 
\psi_{s/2}(y-z)| \\ 
&\ \cdot \Big(
\sum_{\gamma\prec\beta} |x-z|^{|\beta|-|\gamma|-\epsilon} (1+|x|+|z|)^{2\epsilon} 
\|\Pi^{(\tau)}-\widetilde\Pi^{(\tau)}\|_{\prec\beta} (\sqrt[4]{s})^{|\gamma|-2} \\ 
&\ \hphantom{\Big(} + \sum_{\gamma\prec\beta} |x-z|^{|\beta|-|\gamma|} 
(\sqrt[4]{s})^{|\gamma|-2-\epsilon}
(1+\sqrt[4]{s}+|z|)^{2\epsilon} 
\|\Pi^{(\tau)}-\widetilde\Pi^{(\tau)}\|^-_{\gamma} \Big) \, , 
\end{align}
which by the moment bound \eqref{momentbound} is bounded by 
\begin{align}
&(\sqrt[4]{s})^{\alpha-2-\epsilon-4} (\sqrt[4]{s}+|x-y|)^{|\beta|-\alpha} 
(1+\sqrt[4]{s}+|x|+|y|)^{2\epsilon} \\
&\ \cdot \big( \|\Pi^{(\tau)}-\widetilde\Pi^{(\tau)}\|_{\prec\beta} 
+ \|\Pi^{(\tau)}-\widetilde\Pi^{(\tau)}\|^-_{\prec\beta} \big) \, . 
\end{align}
Since $|\beta|<2$ upgrades to $|\beta|-2+\epsilon<0$ by the choice of $\epsilon$ in \eqref{epsilon3}, 
this yields
\begin{align}
|{\it II} - \widetilde{\it II} |
&= \Big| \int_t^\infty ds\, \partial_s \E\big(\Pi^{-(\tau)}_{x\beta} - \widetilde\Pi^{-(\tau)}_{x\beta}\big)_s(x) \Big| \\
&\lesssim (\sqrt[4]{t})^{|\beta|-2-\epsilon} (1+\sqrt[4]{t}+|x|)^{2\epsilon} 
\big( \|\Pi^{(\tau)}-\widetilde\Pi^{(\tau)}\|_{\prec\beta} 
+ \|\Pi^{(\tau)}-\widetilde\Pi^{(\tau)}\|^-_{\prec\beta} \big) \, , 
\end{align}
which for $t=1$ and $x=0$ is estimated by the right-hand side of \eqref{eh39}. 

We turn to the estimate of ${\it I} - \widetilde{\it I}$.
Since we are in the smooth setting with $\tau>0$ fixed, 
we can appeal to the hierarchy from \eqref{defPi-} in the form of 
\begin{equation}\label{eh43}
\Pi_x^{-(\tau)}+c^{(\tau)} = P\underbrace{ \sum_{k\geq0} \z_k
(\Pi^{(\tau)}_x)^k \Delta\Pi^{(\tau)}_x }_{=:\it I'}
+ \xi_\tau 1
- \underbrace{\sum_{k\geq1} \tfrac{1}{k!} (\Pi^{(\tau)}_x)^k (D^{(\0)})^k c^{(\tau)} }_{=:\it I''} \, ,
\end{equation}
where we note that since we are in the induction step (hence $\beta\neq0$), 
the term $\xi_\tau 1$ is irrelevant here. 
To estimate $|\E ({\it I'}-{\it \widetilde I'})_{\beta t}(x)|$ 
we appeal to telescoping \eqref{telescope} 
and H\"older's inequality, 
to the definition of $\|\Pi^{(\tau)}\|_\beta$, 
\eqref{estPi}, \eqref{Pi_tau}, $\eqref{eh40}_{\prec\beta}$ 
and the moment bound \eqref{momentbound} to obtain 
\begin{equation}
|\E ({\it I'}-{\it \widetilde I'})_{\beta t}(x)|
\lesssim_\tau (1+\sqrt[4]{t})^{|\beta|} (1+\sqrt[4]{t}+|x|)^{2\epsilon} 
\big( \|\Pi^{(\tau)}-\widetilde\Pi^{(\tau)}\|_{\prec\beta} 
+ [\Pi^{(\tau)}-\widetilde\Pi^{(\tau)}]_{\prec\beta} \big) \, , 
\end{equation}
which for $t=1$ and $x=0$ is again estimated by the right-hand side of \eqref{eh39}. 
Similarly, we estimate $|\E ({\it I''}-{\it \widetilde I''})_{\beta t}(x)|$ 
by the right-hand side of \eqref{eh39}: 
appealing to the definition of $\|\Pi^{(\tau)}\|_\beta$, 
\eqref{estPi}, \eqref{c_tau}, $\eqref{eh39}_{\prec\beta}$ 
and the moment bound \eqref{momentbound} leads with the 
choice of $t=1$ and $x=0$ to the desired result. 
This concludes the argument for \eqref{eh39}.

Having \eqref{eh40} and \eqref{eh39} we are ready to prove \eqref{step5}. 
For this, we note that the quantity we have to estimate 
is the $\beta$-component of 
\begin{align}
&\big(\Pi^{-(\tau)}_x(y) 
- \Pi^{-(\tau)}_x(z) \big)
- \big( \widetilde\Pi^{-(\tau)}_x(y) 
- \widetilde\Pi^{-(\tau)}_x(z) \big) \\
&= \sum_{k\geq0} \z_k (\Pi^{(\tau)}_x)^k(y) 
\big(\Delta\Pi^{(\tau)}_x(y) 
- \Delta\Pi^{(\tau)}_x(z) 
- \Delta\widetilde\Pi^{(\tau)}_x(y) 
+ \Delta\widetilde\Pi^{(\tau)}_x(z))\big) \label{part1} \\
&\,+ \sum_{k\geq1} \z_k 
\big( (\Pi^{(\tau)}_x)^k(y) 
- (\widetilde\Pi^{(\tau)}_x)^k(y) \big)
\big(\Delta\widetilde\Pi^{(\tau)}_x(y) 
- \Delta\widetilde\Pi^{(\tau)}_x(z) \big) \label{part2} \\
&\,+ \sum_{k\geq1} \z_k 
\big( (\Pi^{(\tau)}_x)^k(y) 
- (\Pi^{(\tau)}_x)^k(z)
- (\widetilde\Pi^{(\tau)}_x)^k(y) 
+ (\widetilde\Pi^{(\tau)}_x)^k(z) \big)
\Delta\Pi^{(\tau)}_x(z) \label{part3} \\
&\,+ \sum_{k\geq1} \z_k 
\big( (\widetilde\Pi^{(\tau)}_x)^k(y) 
- (\widetilde\Pi^{(\tau)}_x)^k(z) \big)
\big(\Delta\Pi^{(\tau)}_x(z) 
- \Delta\widetilde\Pi^{(\tau)}_x(z) \big) \label{part4} \\
&\,- \sum_{k\geq1} \tfrac{1}{k!} 
\big( (\Pi^{(\tau)}_x)^k(y) 
- (\Pi^{(\tau)}_x)^k(z)
- (\widetilde\Pi^{(\tau)}_x)^k(y) 
+ (\widetilde\Pi^{(\tau)}_x)^k(z) \big)
(D^{(\0)})^k c^{(\tau)} \qquad \label{part5} \\
&\,- \sum_{k\geq1} \tfrac{1}{k!} 
\big( (\widetilde\Pi^{(\tau)}_x)^k(y) 
- (\widetilde\Pi^{(\tau)}_x)^k(z) \big)
(D^{(\0)})^k (c^{(\tau)}-\widetilde c^{(\tau)} ) \, . \label{part6} 
\end{align}
We treat the terms of \eqref{part1} -- \eqref{part6} separately.
For \eqref{part1} we appeal to H\"older's inequality, \eqref{estPi} 
and the definition of $[\Pi^{(\tau)}]^-_{\beta_{k+1}}$ to obtain 
\begin{align}
\E^\frac{1}{p}| \eqref{part1}_\beta |^p
&\lesssim \sum_{k\geq0}\sum_{e_k+\beta_1+\dots+\beta_{k+1}=\beta} 
|x-y|^{|\beta_1|+\cdots+|\beta_k|} |y-z|^{\alpha-\epsilon} \, 
[\Pi^{(\tau)}-\widetilde\Pi^{(\tau)}]^-_{\beta_{k+1}} \\
&\ \cdot (1+|x-y|+|x-z|)^{|\beta_{k+1}|-\alpha} 
(1+|x|+|y|+|z|)^{2\epsilon} \, , 
\end{align}
which due to $|\beta_1|+\cdots+|\beta_{k+1}|=|\beta|$ 
is bounded by 
\begin{equation}
|y-z|^{\alpha-\epsilon} 
(1+|x-y|+|x-z|)^{|\beta|-\alpha} 
(1+|x|+|y|+|z|)^{2\epsilon} 
[\Pi^{(\tau)}-\widetilde\Pi^{(\tau)}]^-_{\prec\beta} \, . 
\end{equation}
We turn to \eqref{part2}, where we note that by telescoping 
\begin{equation}
(\Pi^{(\tau)}_x)^k(y) 
- (\widetilde\Pi^{(\tau)}_x)^k(y) 
= \sum_{k_1+k_2=k-1} (\Pi^{(\tau)}_x)^{k_1}(y) 
(\widetilde\Pi^{(\tau)}_x)^{k_1}(y) 
\big( \Pi^{(\tau)}_x(y) - \widetilde\Pi^{(\tau)}_x(y) \big) \, . 
\end{equation}
Hence we may appeal to H\"older's inequality, 
\eqref{estPi}, the definition of $\|\Pi^{(\tau)}\|_{\beta'}$ 
and \eqref{Pi_Holder} to obtain 
\begin{align}
\E^\frac{1}{p}| \eqref{part2}_\beta |^p
&\lesssim_\tau \sum_{k\geq1}\sum_{e_k+\beta_1+\dots+\beta_{k+1}=\beta} 
|x-y|^{|\beta_1|+\cdots+|\beta_k|-\epsilon} (1+|x|+|y|)^{2\epsilon} \\
&\ \cdot \|\Pi^{(\tau)}-\widetilde\Pi^{(\tau)}\|_{\beta_k} 
|y-z|^\alpha (1+|x-y|+|x-z|)^{|\beta_{k+1}|-\alpha} \, , 
\end{align}
which again by $|\beta_1|+\cdots+|\beta_{k+1}|=|\beta|$ 
is bounded by 
\begin{equation}
|y-z|^{\alpha-\epsilon} 
(1+|x-y|+|x-z|)^{|\beta|-\alpha} 
(1+|x|+|y|+|z|)^{2\epsilon} 
\|\Pi^{(\tau)}-\widetilde\Pi^{(\tau)}\|_{\prec\beta} \, . 
\end{equation}
We turn to \eqref{part3}, which can be estimated as \eqref{part2}, 
by replacing \eqref{Pi_Holder} by \eqref{Pi_tau} 
and using in addition 
\begin{align}
&\E^\frac{1}{p}| \Pi^{(\tau)}_{x\beta}(y)-\Pi^{(\tau)}_{x\beta}(z) 
-\widetilde\Pi^{(\tau)}_{x\beta}(y) + \widetilde\Pi^{(\tau)}_{x\beta}(z)|^p \\
&\lesssim |y-z|^{\alpha-\epsilon} (1+|x-y|+|x-z|)^{|\beta|-\alpha} 
(1+|x|+|y|+|z|)^{2\epsilon} 
\|\Pi^{(\tau)}-\widetilde\Pi^{(\tau)}\|_{\preceq\beta} \, . 
\label{eh45}
\end{align}
In turn, this last estimate follows from 
$\Pi^{(\tau)}_x(y)-\Pi^{(\tau)}_x(z) = \Gtau_{xz}\Pi_z(y)$, 
see \eqref{recenter}, together with \eqref{estPi}, \eqref{eh14}, \eqref{step3}, 
and the definitions of $\|\G\|_\beta$ and $\|\Pi^{(\tau)}\|_\beta$. 

Also \eqref{part4} can be estimated analogously, 
using in addition \eqref{eh40} and 
\begin{equation}\label{eh44}
\E^\frac{1}{p}| \Pi^{(\tau)}_{x\beta}(y) - \Pi^{(\tau)}_{x\beta}(z)|^p
\lesssim |y-z|^{\alpha} (|x-y|+|x-z|)^{|\beta|-\alpha} \, ,
\end{equation}
which follows again from \eqref{recenter}, \eqref{estPi} and \eqref{eh14}. 

We turn to \eqref{part6}, where we first note that by telescoping 
\begin{equation}
(\widetilde\Pi^{(\tau)}_x)^k(y) 
- (\widetilde\Pi^{(\tau)}_x)^k(z) 
= \sum_{k_1+k_2=k-1} (\widetilde\Pi^{(\tau)}_x)^{k_1}(y) 
(\widetilde\Pi^{(\tau)}_x)^{k_2}(z) 
\big( \widetilde\Pi^{(\tau)}_x(y) - \widetilde\Pi^{(\tau)}_x(z)\big) \, . 
\end{equation}
Using this, together with H\"older's inequality, 
\eqref{estPi} and \eqref{eh44}, we obtain 
\begin{align}
\E^\frac{1}{p}| \eqref{part6}_\beta |^p 
&\lesssim \sum_{k\geq1} \sum_{\beta_1+\cdots+\beta_{k+1}=\beta} 
(|x-y|+|x-z|)^{|\beta_1|+\cdots+|\beta_k|-\alpha} 
|y-z|^{\alpha} \\
&\ \cdot |\big((D^{(\0)})^k (c^{(\tau)}-\widetilde c^{(\tau)})\big)_{\beta_{k+1}} | \, . 
\end{align}
On the one hand, the last factor is by \eqref{eh39} estimated by 
$ \|\Pi^{(\tau)}-\widetilde\Pi^{(\tau)}\|_{\prec\beta} 
+ [\Pi^{(\tau)}-\widetilde\Pi^{(\tau)}]_{\prec\beta} 
+ \|\Pi^{(\tau)}-\widetilde\Pi^{(\tau)}\|^-_{\prec\beta} $.
On the other hand, the last factor is by the definition \eqref{D0} of $D^{(\0)}$ only non-vanishing, 
if $|\beta_{k+1}|\geq\alpha(k+1)$; 
by the definition \eqref{homogeneity} of the homogeneity this implies 
$|\beta_1|+\cdots+|\beta_k| 
= |\beta|+k\alpha-|\beta_{k+1}| 
\leq |\beta|-\alpha$, thus
$(|x-y|+|x-z|)^{|\beta_1|+\cdots+|\beta_k|-\alpha} 
\leq (1+|x-y|+|x-z|)^{|\beta|-\alpha}$. 
In total we get 
\begin{align}
\E^\frac{1}{p}| \eqref{part6}_\beta |^p 
&\lesssim |y-z|^{\alpha-\epsilon} (1+|x-y|+|x-z|)^{|\beta|-\alpha} 
(1+|x|+|y|+|z|)^{2\epsilon} \\
&\ \cdot \big(
\|\Pi^{(\tau)}-\widetilde\Pi^{(\tau)}\|_{\prec\beta} 
+ [\Pi^{(\tau)}-\widetilde\Pi^{(\tau)}]_{\prec\beta} 
+ \|\Pi^{(\tau)}-\widetilde\Pi^{(\tau)}\|^-_{\prec\beta} \big) \, .
\end{align}
Finally, \eqref{part5} can be estimated by the same arguments, 
using in addition \eqref{eh45}. 

{\bf Step 6 \rm (Proof of \eqref{step6})\bf.}
We distinguish the cases $|\beta|>2$ and $|\beta|<2$.
For the former, we proceed identically to the proof of Lemma~\ref{rec1} 
and start with  
\begin{align}
&\mathbb{E}^\frac{1}{p}|(\Pi^{-(\tau)}_{x}-\Pi^{-(\tau)}_{y})_{\beta t}(y) 
- (\widetilde\Pi^{-(\tau)}_{x}-\widetilde\Pi^{-(\tau)}_{y})_{\beta t}(y)|^p \\
&\lesssim (\sqrt[4]{t})^{\alpha-2-\epsilon} 
(\sqrt[4]{t}+|x-y|)^{|\beta|-\alpha} 
(1+\sqrt[4]{t}+|x|+|y|)^{2\epsilon} \\
&\ \cdot \big(\|\Pi^{(\tau)}-\widetilde\Pi^{(\tau)}\|_{\prec\beta} 
+ [\Pi^{(\tau)}-\widetilde\Pi^{(\tau)}]_{\prec\beta} 
+ \|\Pi^{(\tau)}-\widetilde\Pi^{(\tau)}\|^-_{\prec\beta} \big) \\
&\ \cdot \big(\|\Pi^{(\tau)}-\widetilde\Pi^{(\tau)}\|_{\prec\beta} 
+ [\Pi^{(\tau)}-\widetilde\Pi^{(\tau)}]_{\prec\beta} 
+ \|\Pi^{(\tau)}-\widetilde\Pi^{(\tau)}\|^-_{\prec\beta} \big)^{\bar\epsilon} \, . 
\end{align}
Indeed, this estimate follows from \eqref{recenter-}, 
telescoping \eqref{telescope}, and
the already established \eqref{eh30} \& \eqref{step3} and 
$\eqref{step6}_{\prec\beta}$. 
By a general reconstruction argument, see \cite[Lemma~4.8]{LO22}, 
this implies 
\begin{align}
&\mathbb{E}^\frac{1}{p}| \Pi^{-(\tau)}_{x \beta t}(y) 
- \widetilde\Pi^{-(\tau)}_{x \beta t}(y) |^p \\ 
&\lesssim (\sqrt[4]{t})^{\alpha-2-\epsilon} 
(\sqrt[4]{t}+|x-y|)^{|\beta|-\alpha} 
(1+\sqrt[4]{t}+|x|+|y|)^{2\epsilon} \\
&\ \cdot \big(\|\Pi^{(\tau)}-\widetilde\Pi^{(\tau)}\|_{\prec\beta} 
+ [\Pi^{(\tau)}-\widetilde\Pi^{(\tau)}]_{\prec\beta} 
+ \|\Pi^{(\tau)}-\widetilde\Pi^{(\tau)}\|^-_{\prec\beta} \big) \\
&\ \cdot \big(\|\Pi^{(\tau)}-\widetilde\Pi^{(\tau)}\|_{\prec\beta} 
+ [\Pi^{(\tau)}-\widetilde\Pi^{(\tau)}]_{\prec\beta} 
+ \|\Pi^{(\tau)}-\widetilde\Pi^{(\tau)}\|^-_{\prec\beta} \big)^{\bar\epsilon} \, , 
\end{align}
which yields \eqref{step6} for $|\beta|>2$. 

We turn to the case $|\beta|<2$, and start with 
\begin{align}
&\E^\frac{1}{p}| \Pi^{-(\tau)}_{x\beta}(y) - \widetilde\Pi^{-(\tau)}_{x\beta}(y) |^p 
\lesssim_\tau  (1+|x-y|)^{|\beta|-\alpha} 
(1+|x|+|y|)^{2\epsilon} \\
&\ \cdot \big(\|\Pi^{(\tau)}-\widetilde\Pi^{(\tau)}\|_{\prec\beta} 
+ [\Pi^{(\tau)}-\widetilde\Pi^{(\tau)}]_{\prec\beta} 
+ \|\Pi^{(\tau)}-\widetilde\Pi^{(\tau)}\|^-_{\prec\beta} \big) \, . 
\label{eh38}
\end{align}
Indeed, this follows as the estimates of ${\it I'}-{\it \widetilde I'}$ 
and ${\it I''}-{\it \widetilde I''}$ from \eqref{eh43} in Step~5. 
The only modification is that here we include $c^{(\tau)}$ into ${\it I''}$, 
hence the sum in ${\it I''}$ starts now with $k=0$, 
and we therefore have to include in addition 
the already established $\eqref{eh39}_\beta$ in our assumption. 

We use \eqref{eh38} to make a step towards \eqref{step6} and prove 
\begin{align}
&\sup_{0<t<1} \sup_{x,y} \frac{\E^\frac{1}{p}|\Pi^{-(\tau)}_{x\beta t}(y) - \widetilde\Pi^{-(\tau)}_{x\beta t}(y)|^p}{
(\sqrt[4]{t})^{\alpha-2-\epsilon} 
(\sqrt[4]{t}+|x-y|)^{|\beta|-\alpha} (1+\sqrt[4]{t}+|x|+|y|)^{2\epsilon}} \\[1ex]
&\lesssim_\tau \|\Pi^{(\tau)}-\widetilde\Pi^{(\tau)}\|_{\prec\beta} 
+ [\Pi^{(\tau)}-\widetilde\Pi^{(\tau)}]_{\prec\beta} 
+ \|\Pi^{(\tau)}-\widetilde\Pi^{(\tau)}\|^-_{\prec\beta} \, . \label{eh41}
\end{align}
Indeed, \eqref{eh38} yields
\begin{align}
&\E^\frac{1}{p}|\Pi^{-(\tau)}_{x\beta t}(y) 
- \widetilde\Pi^{-(\tau)}_{x\beta t}(y)|^p \\ 
&\leq \int_{\R^{1+d}} dz\, |\psi_t(y-z)| \, 
\E^\frac{1}{p}| \Pi^{-(\tau)}_{x\beta}(z) - \widetilde\Pi^{-(\tau)}_{x\beta}(z)|^p \\
&\lesssim_\tau \int_{\R^{1+d}} dz\, |\psi_t(y-z)| \, 
(1+|x-z|)^{|\beta|-\alpha} (1+|x|+|z|)^{2\epsilon} \\ 
&\ \cdot \big(\|\Pi^{(\tau)}-\widetilde\Pi^{(\tau)}\|_{\prec\beta} 
+ [\Pi^{(\tau)}-\widetilde\Pi^{(\tau)}]_{\prec\beta} 
+ \|\Pi^{(\tau)}-\widetilde\Pi^{(\tau)}\|^-_{\prec\beta} \big) \, , 
\end{align}
which by the moment bound \eqref{momentbound} and for $t\leq1$ is bounded by 
\begin{equation}
(1+|x-y|)^{|\beta|-\alpha} (1+|x|+|y|)^{2\epsilon} 
\big(\|\Pi^{(\tau)}-\widetilde\Pi^{(\tau)}\|_{\prec\beta} 
+ [\Pi^{(\tau)}-\widetilde\Pi^{(\tau)}]_{\prec\beta} 
+ \|\Pi^{(\tau)}-\widetilde\Pi^{(\tau)}\|^-_{\prec\beta} \big) \, . 
\end{equation}
Then \eqref{eh41} follows from the fact that the map 
$t\mapsto (\sqrt[4]{t})^{\alpha-2-\epsilon}(\sqrt[4]{t}+|x-y|)^{|\beta|-\alpha}$ 
is monotone decreasing on $(0,1]$ and takes at $1$ the value $(1+|x-y|)^{|\beta|-\alpha}$. 
The monotonicity is seen as follows: 
since $\alpha-2-\epsilon<0$, the map $t\mapsto (\sqrt[4]{t})^{(\alpha-2-\epsilon)/(|\beta|-\alpha)}$ 
is monotone decreasing; 
since $|\beta|<2$ implies $1+\frac{\alpha-2-\epsilon}{|\beta|-\alpha}<0$, also 
$t\mapsto (\sqrt[4]{t})^{1+(\alpha-2-\epsilon)/(|\beta|-\alpha)}$ is monotone decreasing. 

It remains to bound the opposite regime of \eqref{eh41}, namely $t\geq1$. 
For this, we first note that by the triangle inequality 
\begin{align}
\E^\frac{1}{p}| \Pi^{-(\tau)}_{x\beta t}(y) 
- \widetilde\Pi^{-(\tau)}_{x\beta t}(y) |^p 
&\leq \E^\frac{1}{p} \Big| \int_t^T ds\, \partial_s 
\big( \Pi^{-(\tau)}_{x\beta s}(y) 
- \widetilde\Pi^{-(\tau)}_{x\beta s}(y) \big) \Big|^p \\ 
&\, + \E^\frac{1}{p}| \Pi^{-(\tau)}_{x\beta T}(y)|^p 
+ \E^\frac{1}{p}| \widetilde\Pi^{-(\tau)}_{x\beta T}(y)|^p \, . 
\end{align}
Since the estimate \eqref{Pi-} implies $\E^\frac{1}{p}|\Pi^{-(\tau)}_{x\beta T}(y)|^p\to0$ as $T\to\infty$ for $|\beta|<2$, 
we obtain 
\begin{equation}
\E^\frac{1}{p}| \Pi^{-(\tau)}_{x\beta t}(y) 
- \widetilde\Pi^{-(\tau)}_{x\beta t}(y) |^p 
\leq \int_t^\infty ds\, \E^\frac{1}{p} \big| \partial_s 
\big( \Pi^{-(\tau)}_{x\beta s}(y) 
- \widetilde\Pi^{-(\tau)}_{x\beta s}(y) \big) \big|^p \, . 
\end{equation}
To estimate this right-hand side, 
we first note that the definition of $[\Pi^{(\tau)}]^-_\beta$ together with the moment bound \eqref{momentbound} yields 
\begin{align}
\E^\frac{1}{p}| \Pi^{-(\tau)}_{x\beta t}(y) - \Pi^{-(\tau)}_{x\beta t}(z) |^p
&\lesssim |y-z|^{\alpha-\epsilon} 
(1+\sqrt[4]{t}+|x-y|+|x-z|)^{|\beta|-\alpha} \\ 
&\ \cdot (1+\sqrt[4]{t}+|x|+|y|+|z|)^{2\epsilon} 
[\Pi^{(\tau)}]^-_\beta \, . 
\end{align}
Interpolating with 
\begin{equation}
\E^\frac{1}{p}| \Pi^{-(\tau)}_{x\beta t}(y) - \Pi^{-(\tau)}_{x\beta t}(z) |^p
\lesssim (\sqrt[4]{t})^{\alpha-2} 
(\sqrt[4]{t}+|x-y|+|x-z|)^{|\beta|-\alpha} \, , 
\end{equation}
which is a consequence of \eqref{Pi-} by the triangle inequality, 
we obtain for every $\bar\epsilon\in(0,1)$ 
\begin{align}
&\E^\frac{1}{p}| \Pi^{-(\tau)}_{x\beta t}(y) - \Pi^{-(\tau)}_{x\beta t}(z) |^p
\lesssim (\sqrt[4]{t})^{(\alpha-2)(1-\bar\epsilon)} 
|y-z|^{(\alpha-\epsilon)\bar\epsilon} \\
&\ \cdot (1\+\sqrt[4]{t}\+|x\-y|\+|x\-z|)^{|\beta|-\alpha} 
(1\+\sqrt[4]{t}\+|x|\+|y|\+|z|)^{2\epsilon\bar\epsilon} 
\big([\Pi^{(\tau)}]^-_\beta\big)^{\bar\epsilon} \, . \label{eh46}
\end{align}
Equipped with this, 
we first use the defining property \eqref{psi} of $\psi_t$, 
the semigroup property \eqref{semigroup}, 
and that $(\partial_0^2-\Delta^2)\psi_{s/2}$ integrates to $0$, 
to obtain 
\begin{align}
\partial_s \big( \Pi^{-(\tau)}_{x\beta s}(y) 
- \widetilde\Pi^{-(\tau)}_{x\beta s}(y) \big)  
&= \int_{\R^{1+d}} dz \, 
(\partial_0^2-\Delta^2) \psi_{s/2}(y-z) \\
&\ \cdot \big( \Pi^{-(\tau)}_{x\beta s/2}(z) - \Pi^{-(\tau)}_{x\beta s/2}(y)
- \widetilde\Pi^{-(\tau)}_{x\beta s/2}(z) +\widetilde\Pi^{-(\tau)}_{x\beta s/2}(y)\big) \, ,
\end{align}
which by \eqref{eh46} and the moment bound \eqref{momentbound} yields
\begin{align}
&\E^\frac{1}{p} \big| \partial_s 
\big( \Pi^{-(\tau)}_{x\beta s}(y) 
- \widetilde\Pi^{-(\tau)}_{x\beta s}(y) \big) \big|^p \\ 
&\lesssim  \int_{\R^{1+d}} dz \, 
|(\partial_0^2-\Delta^2) \psi_{s/2}(y-z)| \, 
(\sqrt[4]{s})^{(\alpha-2)(1-\bar\epsilon)} 
|y-z|^{(\alpha-\epsilon)\bar\epsilon} \\
&\ \cdot (1+\sqrt[4]{s}+|x-y|+|x-z|)^{|\beta|-\alpha} 
(1+\sqrt[4]{s}+|x|+|y|+|z|)^{2\epsilon\bar\epsilon} 
\big([\Pi^{(\tau)}\-\widetilde\Pi^{(\tau)}]^-_\beta\big)^{\bar\epsilon} \\
&\lesssim (\sqrt[4]{s})^{\alpha-2+\bar\epsilon(2-\epsilon)-4}
(1\+\sqrt[4]{s}\+|x\-y|)^{|\beta|-\alpha} 
(1\+\sqrt[4]{s}\+|x|\+|y|)^{2\epsilon\bar\epsilon} 
\big([\Pi^{(\tau)}\-\widetilde\Pi^{(\tau)}]^-_\beta\big)^{\bar\epsilon} \, .
\end{align}
For $s$ large, the effective exponent on $\sqrt[4]{s}$ equals 
$|\beta|-6+\bar\epsilon(2+\epsilon)$. 
Choosing $\bar\epsilon=\epsilon/(2+\epsilon)$, 
this exponent equals $|\beta|-6+\epsilon$. 
The choice of $\epsilon$ in \eqref{epsilon3} implies that $|\beta|<2$ 
strengthens to $|\beta|+\epsilon<2$, and thus the exponent is
$<-4$, which makes $(\sqrt[4]{s})^{|\beta|-6+\epsilon}$ 
integrable at $s=\infty$. 
We therefore get for this choice of $\bar\epsilon$
\begin{align}
&\E^\frac{1}{p}| \Pi^{-(\tau)}_{x\beta t}(y) 
- \widetilde\Pi^{-(\tau)}_{x\beta t}(y) |^p \\
&\lesssim \int_t^\infty ds\, 
(\sqrt[4]{s})^{\alpha-2+\bar\epsilon(2-\epsilon)-4}
(1+\sqrt[4]{s}+|x-y|)^{|\beta|-\alpha} 
(1+\sqrt[4]{s}+|x|+|y|)^{2\epsilon\bar\epsilon} \\ 
&\ \cdot \big([\Pi^{(\tau)}-\widetilde\Pi^{(\tau)}]^-_\beta\big)^{\bar\epsilon} \\
&\lesssim (\sqrt[4]{t})^{\alpha-2+\bar\epsilon(2-\epsilon)}
(1\+\sqrt[4]{t}\+|x-y|)^{|\beta|-\alpha} 
(1\+\sqrt[4]{t}\+|x|\+|y|)^{2\epsilon\bar\epsilon} 
\big([\Pi^{(\tau)}\-\widetilde\Pi^{(\tau)}]^-_\beta\big)^{\bar\epsilon} \, .
\end{align}
This is further bounded by 
\begin{equation}
(\sqrt[4]{t})^{\alpha-2-\epsilon} 
(1+\sqrt[4]{t}+|x-y|)^{|\beta|-\alpha} 
(1+\sqrt[4]{t}+|x|+|y|)^{\epsilon+\bar\epsilon(2+\epsilon)} 
\big([\Pi^{(\tau)}-\widetilde\Pi^{(\tau)}]^-_\beta\big)^{\bar\epsilon} \, , 
\end{equation}
which yields by $\bar\epsilon=\epsilon/(2+\epsilon)$ 
\begin{equation}
\sup_{t\geq1} \sup_{x,y} \frac{\E^\frac{1}{p}|\Pi^{-(\tau)}_{x\beta t}(y) 
- \widetilde\Pi^{-(\tau)}_{x\beta t}(y)|^p}{
(\sqrt[4]{t})^{\alpha-2-\epsilon} 
(\sqrt[4]{t}+|x-y|)^{|\beta|-\alpha} (1+\sqrt[4]{t}+|x|+|y|)^{2\epsilon}} 
\lesssim \big([\Pi^{(\tau)}-\widetilde\Pi^{(\tau)}]^-_\beta\big)^{\bar\epsilon} \, . 
\end{equation}
Together with the already established $\eqref{step5}_\beta$ 
this finishes the proof of \eqref{step6}.\qedhere
\end{proof}


\appendix

\section{Weights}\label{weights}

In this section we adapt the weights $\bar{w}, w(y), w_x(y)$ from 
\cite[(4.23), (4.37), (4.46)]{LOTT21} appropriately, 
so that we can move beyond the limitation of the regularity index 
$2-\alpha-D/2$ of the norm $\|\cdot\|_*$ from the spectral gap assumption 
\eqref{sg} being positive. 
We provide all properties of these weights that we used along the way 
and that are used in the aforementioned reference. 

We start with the weight $\bar{w}$, which we define by 
\begin{equation}\label{w_bar}
\bar w := \Big(\int_{\R^{1+d}} dz\, \E^\frac{2}{q}\big| 
(-\partial_0^2+\Delta^2)^{\frac{1}{4}(\alpha-2+D/2)} \, \delta\xi(z) \big|^q\Big)^{\frac{1}{2}} \, . 
\end{equation}
Note that in view of $q\leq2$ by Minkowski's inequality
\begin{equation}\label{wbar<Sobolev}
\bar{w}\leq \E^\frac{1}{q} \big\|\delta\xi\big\|_{\dot{H}^{\alpha-2+D/2}}^q \, ,
\end{equation}
which is necessary to obtain 
from an estimate like \eqref{deltaPi-_cauchy} 
by an application of the spectral gap inequality \eqref{sg} and dualization 
an estimate like \eqref{Pi-_cauchy}. 

We turn to $w(y)$, which we define by 
\begin{equation}\label{w_gain} 
w(y) := \Big( \int_{\R^{1+d}} dz\, |y-z|^{-2\kappa} \, \E^\frac{2}{q}\big| 
(-\partial_0^2+\Delta^2)^{\frac{1}{4}(\alpha-2+D/2)} \, \delta\xi(z) \big|^q\Big)^{\frac{1}{2}} \, . 
\end{equation}
In reconstruction Lemma~\ref{rec3} we used the estimate 
\begin{equation}\label{basecase}
\E^\frac{1}{q}| \partial^\n \delta\xi_t(y) |^{q} 
\lesssim (\sqrt[4]{t})^{\alpha-2+\kappa-|\n|} w(y) \, ,
\end{equation}
which by  
\begin{align}
&\partial^\n\delta\xi_t(y) = \int
dz\, \partial^\n(-\partial_0^2\+\Delta^2)^{-\frac{1}{4}(\alpha-2+D/2)} \psi_t(y\-z) (-\partial_0^2\+\Delta^2)^{\frac{1}{4}(\alpha-2+D/2)}\delta\xi(z), 
\end{align}
follows from Cauchy-Schwarz and 
\begin{equation}
\Big(\int_{\R^{1+d}} dz\, |z|^{2\kappa} \, 
\big|(-\partial_0^2+\Delta^2)^{-\frac{1}{4}(\alpha-2+D/2)}
\psi_t(z) \big|^2\Big)^\frac{1}{2}
\lesssim (\sqrt[4]{t})^{\alpha-2+\kappa} \, , 
\end{equation}
which is a consequence of the scaling identity \eqref{scaling_psi}. 
Furthermore, $w(y)$ behaves well under (square) averaging, 
\begin{align}
\Big(\int_{\R^{1+d}} dz\, |\psi_t(y-z)| \, w^2(z) \Big)^\frac{1}{2}
&\lesssim \min \big( w(y), (\sqrt[4]{t})^{-\kappa} \bar{w} \big) \, , 
\label{average_psi_w}\\
\Big( \fint_{B_\lambda(y)} dz\, w^2(z) \Big)^\frac{1}{2}
&\lesssim \lambda^{-\kappa} \bar{w} \label{average_ball_w} \, .
\end{align}
The former is a consequence of 
\begin{equation}\label{eh47}
\int_{\R^{1+d}} dz\, |\psi_t(y-z)| \, |x-z|^{-2\kappa} 
\lesssim (\sqrt[4]{t}+|x-y|)^{-2\kappa}\, , 
\end{equation}
which relies on $\kappa<D/2$, see \eqref{kappa1}, 
and a similar bound for averages over balls establishes the latter. 

Finally, we define $w_x(y)$ by 
\vspace{-1ex}
\begin{equation}\label{w_centered}
\vspace{-1ex}
w_x(y) := |y-x|^{-\kappa} \bar w + w(y) \, , 
\end{equation}
which by \eqref{average_psi_w} and \eqref{eh47} satisfies 
\vspace{-1ex}
\begin{equation}
\int_{\R^{1+d}}dz\, |\psi_t(y-z)|\, w_x^2(z)
\lesssim w_x^2(y) \, .
\end{equation}
%

\section*{Acknowledgements}

The author would like to thank Felix Otto and Rishabh Gvalani 
for many discussions during the course of this work, 
and Lucas Broux, Ajay Chandra, Rhys Steele, Pavlos Tsatsoulis and Christian Wagner for helpful comments.
This work was completed while the author held a position at 
the Max--Planck Institute for Mathematics in the Sciences; 
funding and working conditions are gratefully

%
%
%
%
\pdfbookmark{References}{references}
\addtocontents{toc}{\protect\contentsline{section}{References}{\thepage}{references.0}}
%
%
\bibliographystyle{alphaurl}
\small
\bibliography{characterization}{}
%
%
\end{document}